\documentclass[11pt]{article}

\usepackage{amsmath,amssymb,enumitem,amsthm,stmaryrd,multirow}
\usepackage{graphics,pstricks,pst-node,xy}
\xyoption{all}

\numberwithin{equation}{section}
\newtheorem{thm}{Theorem}[section]
\newtheorem{prp}[thm]{Proposition}
\newtheorem{lmm}[thm]{Lemma}
\newtheorem{crl}[thm]{Corollary}

\newtheorem{thmnum}{Theorem}
\newtheorem{crlnum}{Corollary}

\theoremstyle{definition}
\newtheorem{dfn}[thm]{Definition}
\newtheorem{rmk}[thm]{Remark}

\newcounter{saveenumi}

\def\BE#1{\begin{equation}\label{#1}}
\def\EE{\end{equation}}
\def\lan{\langle}
\def\ran{\rangle}
\def\lr#1{\lan#1\ran}

\def\ov#1{\overline{#1}}
\def\ti#1{\tilde{#1}}
\def\wt#1{\widetilde{#1}}
\def\e_ref#1{(\ref{#1})}
\def\smsize#1{\begin{small}#1\end{small}}

\def\sf#1{\textsf{#1}}
\def\tn#1{\textnormal{#1}}
\def\Lau#1{\llceil{#1}\rrceil}
\def\coeff#1{\llbracket{#1}\rrbracket}
\def\bigcoeff#1{\big\llbracket{#1}\big\rrbracket}

\def\lra{\longrightarrow}
\def\Lra{\Longrightarrow}
\def\Llra{\Longleftrightarrow}

\def\1{\mathbf{1}}
\def\al{\alpha}
\def\be{\beta}
\def\ga{\gamma}
\def\de{\delta}
\def\ep{\epsilon}
\def\io{\iota}

\def\om{\omega}
\def\si{\sigma}
\def\th{\theta}

\def\Ga{\Gamma}
\def\La{\Lambda}
\def\Om{\Omega}
\def\Th{\Theta}
\def\De{\Delta}

\def\i{\infty}
\def\hb{\hbar}

\def\cA{\mathcal A}

\def\C{\mathbb C}
\def\cC{\mathcal C}

\def\fC{\mathfrak C}
\def\d{\mathfrak d}
\def\bfd{\mathbf d}

\def\E{\mathbf e}
\def\F{\mathcal F}

\def\H{\mathcal H}
\def\cH{\mathcal H}

\def\I{\mathfrak i}

\def\K{\mathcal K}
\def\cL{\mathcal L}
\def\fL{\mathfrak L}
\def\cM{\mathcal M}
\def\M{\mathfrak M}
\def\N{\mathcal N}

\def\O{\mathcal O}
\def\cO{\mathcal O}
\def\P{\mathbb P}

\def\Pn{\mathbb P^{n-1}}

\def\fR{\mathfrak R}
\def\Q{\mathbb Q}
\def\cQ{\mathcal{Q}}
\def\bfQ{\mathbf{Q}}

\def\fs{\mathfrak s}
\def\bS{\mathbb S}
\def\cS{\mathcal S}
\def\T{\mathbb T}

\def\X{\mathfrak X}

\def\cU{\mathcal U}
\def\V{\mathcal V}
\def\cX{\mathcal{X}}
\def\cY{\mathcal Y}
\def\Z{\mathbb Z}
\def\cZ{\mathcal Z}
\def\a{\mathbf a}

\def\nd{\textnormal{d}}

\def\ne{\textnormal{e}}
\def\x{\mathbf x}

\def\Ext{\textnormal{Ext}}
\def\Edg{\textnormal{Edg}}

\def\ev{\textnormal{ev}}

\def\GW{\textnormal{GW}}

\def\mod{\textnormal{mod~}}

\def\rk{\textnormal{rk}}
\def\Proj{\tn{Proj}}
\def\PGL{\textnormal{PGL}}
\def\Rs#1{\underset{#1}{\mathfrak R}}
\def\SQ{\textnormal{SQ}}
\def\Sym{\textnormal{Sym}}

\def\Ver{\textnormal{Ver}}

\def\vir{\textnormal{vir}}

\begin{document}

\title{Mirror Symmetry for Stable Quotients Invariants}
\author{Yaim Cooper
and Aleksey Zinger\thanks{Partially supported by NSF grants DMS-0635607 and DMS-0846978}}
\date{\today}

\maketitle

\begin{abstract}
The moduli space of stable quotients introduced by Marian-Oprea-Pandharipande provides
a natural compactification of the space of morphisms from nonsingular curves
to a nonsingular projective variety and carries a natural virtual class.
We show that the analogue of Givental's $J$-function for the resulting twisted projective
invariants is described by the same mirror hypergeometric series as the corresponding
Gromov-Witten invariants (which arise from the moduli space of stable maps),
but without the mirror transform (in the Calabi-Yau case).
This implies that the stable quotients and Gromov-Witten twisted invariants agree
if there is enough ``positivity", but not in all cases.
As a corollary of the proof, we show that certain twisted Hurwitz numbers arising in the stable
quotients theory are also described by a fundamental object associated with
this hypergeometric series.
We thus completely answer some of the questions posed by Marian-Oprea-Pandharipande
concerning their invariants.
Our results suggest
a deep connection between the stable quotients invariants
of complete intersections and the geometry of the mirror families.
As in Gromov-Witten theory, computing Givental's $J$-function (essentially
a generating function for genus 0 invariants with 1 marked point) is key to computing
stable quotients invariants of higher genus and with more marked points;
we exploit this in forthcoming papers.
\end{abstract}

\tableofcontents

\section{Introduction}
\label{intro_sec}

\noindent
Gromov-Witten invariants of a smooth projective variety~$X$ are certain counts of curves in~$X$
that arise from integrating against the virtual class of the moduli space of stable maps.
These are known to possess striking structures which are often completely unexpected
from the classical point of view.
For example, the genus~0 Gromov-Witten invariants of a quintic threefold,
i.e.~a degree~5 hypersurface in~$\P^4$, are related by a so-called
\sf{mirror formula} to a certain hypergeometric series.
This relation was explicitly predicted in~\cite{CdGP}
and mathematically confirmed in~\cite{Gi} and~\cite{LLY} in the 1990s.
In fact, the prediction of~\cite{CdGP} has been shown to be a special case of
mirror symmetry for certain twisted Gromov-Witten invariants
of projective complete intersections of sufficiently small
total multi-degree \cite{Gi2,LLY3};
these invariants are associated with direct sums of line bundles (positive and negative)
over~$\P^n$.
This relation is often described by assembling two-point Gromov-Witten invariants
(but without constraints on the second marked point) into a generating function,
known as Givental's $J$-function.
In most cases (in particular, when the anticanonical class of the corresponding
complete intersection is at least twice the hyperplane class),
the $J$-function is precisely equal to the appropriate hypergeometric series.
In certain borderline cases, they differ by a simple exponential factor.
In the remaining Calabi-Yau cases, the correcting factors are more complicated
and the two power series also differ by a change of the power series variable,
known as the mirror~map.\\

\noindent
The gauged linear $\si$-model of \cite{Witten} counts rational curves in toric
complete intersections by integrating over the natural toric compactifications
of the spaces of rational maps into the ambient toric variety.
Based on physical considerations, it is shown in~\cite{MP} that
the (three-point) Gromov-Witten and gauged linear $\si$-model
generating functions for the well-studied quintic threefold are related
by the mirror map, with a minor additional adjustment;
see \cite[(4.24),(4.28)]{MP}, for example.
This suggests that the mirror map relating the A (symplectic) side of mirror symmetry
to the B (complex geometric) side may be more reflective of the choice of
curve counting theory on the A-side
than of the mirror symmetry itself.
Unfortunately from the mathematical standpoint,
the compactifying spaces in the gauged linear $\si$-model
do not possess many of the nice properties of the spaces of stable maps
and require fixing a complex structure on the domain of the maps.\\

\noindent
The moduli spaces of stable quotients, $\ov{Q}_{g,m}(X,d)$, constructed in~\cite{MOP09}
provide an alternative to
the moduli spaces of stable maps, $\ov\M_{g,m}(X,d)$, for compactifying spaces of
degree~$d$ morphisms from genus~$g$ nonsingular curves with $m$~marked points to
a projective variety~$X$ (with a choice of polarization).\footnote{These ``compactifications",
$\ov{Q}_{g,m}(X,d)$ and $\ov\M_{g,m}(X,d)$, are generally just compact spaces
containing the spaces of morphisms; the latter need not be dense in
$\ov{Q}_{g,m}(X,d)$ or $\ov\M_{g,m}(X,d)$.}
In this paper, we show that the genus~0 stable quotients theory,
just like  the gauged linear $\si$-model, of Calabi-Yau projective complete intersections
is related to their genus~0 Gromov-Witten theory essentially by the mirror map;
see \e_ref{GWvsSQmirr_e}.
Based on the approaches of~\cite{bcov1} and~\cite{g0ci}, this relationship between the
stable quotients and Gromov-Witten invariants should extend to higher genera;
we expect to confirm this in the genus~1 case in~\cite{CZ3}.
In~\cite{CZ2}, it is shown that the genus~0 three-point
stable quotients and Gromov-Witten invariants of Calabi-Yau projective complete intersection
threefolds are related precisely by the mirror map.
The mirror formula obtained in this paper is central to the computations
in~\cite{CZ3} and~\cite{CZ2}.
Thus, our paper provides further evidence that the mirror map
is an entirely A-side feature and suggests that the stable quotients theory may be
the curve counting theory most directly related to the B-side of mirror symmetry.
In light of the results in this paper,
we also hope that certain properties of the mirror map, such as the integrality
of its coefficients \cite{LY2,KR}, can be explained geometrically by comparing
the stable quotients and Gromov-Witten invariants.\\

\noindent
The moduli space $\ov{Q}_{g,m}(\Pn,d)$ consists of equivalence classes
of tuples
$$(\cC,y_1,\ldots,y_m;S\subset\C^n\!\otimes\!\cO_{\cC}),$$
where $(\cC,y_1,\ldots,y_m)$ is a genus~$g$ nodal curve with $m$ marked points
and $S\subset\C^n\!\otimes\!\cO_{\cC}$ is a subsheaf of rank~1 and degree~$-d$,
that satisfy certain stability and torsion properties; see Section~\ref{SQdfn_sec}.
This moduli space is smooth if $g\!=\!0$ or $(g,m)\!=\!(1,0)$
and carries a virtual class in all cases.
There is a natural surjective contraction morphism
$$c\!: \ov\M_{g,m}(\Pn,d)\lra\ov{Q}_{g,m}(\Pn,d), $$
which is not injective for $d\!>\!0$ and generally contracts a lot of boundary strata.
For example, $\ov{Q}_{1,0}(\Pn,d)$ is irreducible and has Picard rank just~2; see
\cite[Theorem~4.1]{Co11}.
Thus, the moduli spaces of stable quotients are much more efficient
compactifications than the moduli spaces of stable~maps.
However, in the case $X\!=\!\Pn$ and $(g,m)\!=\!(0,3)$, this compactification
is larger than the gauged linear $\si$-model compactification;
see \cite[Section~3.7]{MP}.\\

\noindent
As in the case of stable maps, there are evaluation morphisms,
$$\ev_i\!: \ov{Q}_{g,m}(\Pn,d)\lra\Pn, \qquad i=1,2,\ldots,m,$$
corresponding to each marked point.\footnote{The morphism $\ev_i$
sends a tuple $(\cC,y_1,\ldots,y_m,S)$ to the line $S_{y_i}\!\subset\!\C^n$
if $S$ is viewed as a line subbundle of the trivial rank~$n$ bundle over~$\cC$.}
There is also a universal curve
$$\pi\!: \cU\lra \ov{Q}_{g,m}(\Pn,d)$$
with $m$ sections $\si_1,\ldots,\si_m$ (given by the marked points) and
a universal rank~1 subsheaf
$$\cS\subset \C^n\!\otimes\!\cO_{\cU}\,.$$
For each $i\!=\!1,2,\ldots,m$, let
$$\psi_i=-\pi_*(\si_i^2)   \in H^2\big( \ov{Q}_{g,m}(\Pn,d)\big)$$
be the first Chern class of the universal cotangent line bundle, as usual.
By \cite[Theorems~2,3]{MOP09}, the moduli space $\ov{Q}_{g,m}(\Pn,d)$ carries
a canonical virtual class and
\BE{VFC_e} c_*\big[\ov\M_{g,m}(\Pn,d) \big]^{\vir}=\big[\ov{Q}_{g,m}(\Pn,d) \big]^{\vir}. \EE
Since the evaluation morphisms $\ev_i$ and the $\psi$-classes on the two moduli spaces
commute with $c$ and $c^*$, respectively, \e_ref{VFC_e} implies that the (untwisted)
Gromov-Witten and stable quotients invariants of~$\P^{n-1}$, obtained by
integrating pull-backs of cohomology classes on~$\Pn$ by~$\ev_i$ and powers of $\psi$-classes
against the two virtual classes, are the same; see \cite[Theorem~3]{MOP09}.
In this paper, we study twisted invariants in genus~0, arising from sums of line bundles over~$\Pn$;
they relate invariants of projective complete intersections to
the invariants of the ambient space.\\

\noindent
For $l\!\in\!\Z^{\ge0}$ and $l$-tuple $\a\!=\!(a_1,\ldots,a_l)\!\in\!(\Z^*)^l$ of nonzero integers,
let
\begin{gather*}
|\a|=\sum_{k=1}^l|a_k|,  \qquad
\lr\a=\prod_{a_k>0}\!a_k\bigg/\!\!\prod_{a_k<0}\!a_k\,, \qquad
\a!=\prod_{a_k>0}\!\!a_k!\,, \qquad
\a^{\a}=\prod_{k=1}^l a_k^{|a_k|}\,,\\
\ell^{\pm}(\a)=\big|\{k\!:\,(\pm1)a_k\!>\!0\}\big|, \qquad
\ell(\a)=\ell^+(\a)-\ell^-(\a).
\end{gather*}
If in addition $n\!\in\!\Z^+$ and $d\!\in\!\Z^+$, let
\BE{Vprdfn_e}
\dot\V_{n;\a}^{(d)}=\bigoplus_{a_k>0}R^0\pi_*\big(\cS^{*a_k}(-\si_1)\big)
 \oplus \bigoplus_{a_k<0}R^1\pi_*\big(\cS^{*a_k}(-\si_1)\big)
 \lra \ov{Q}_{0,2}(\Pn,d),\EE
where $\pi\!:\cU\!\lra\!\ov{Q}_{0,2}(\Pn,d)$ is the universal curve;
this sheaf is locally free.
The Euler class of the analogue of this sheaf in Gromov-Witten theory describes
the genus~0 invariants of the total space of the vector bundle
\BE{VBtot_e}
\bigoplus_{a_k<0}\cO_{\P^{n-1}}(a_k)\big|_{X_{(a_k)_{a_k>0}}}\lra X_{(a_k)_{a_k>0}}\,,\EE
where $X_{(a_k)_{a_k>0}}\subset\Pn$ is a nonsingular complete intersection of multi-degree
$(a_k)_{a_k>0}$.
The situation in the stable quotients theory is similar.
If $a_k\!>\!0$ for all~$k$, the moduli space $\ov{Q}_{0,2}(X_{\a},d)$ carries a natural virtual
fundamental class and the resulting invariants of $X_{\a}$ are described
by the Euler class of~\e_ref{Vprdfn_e};
see \cite[Theorem~4.5.2]{CKM} and \cite[Proposition~6.2.3]{CKM}, respectively.\\

\noindent
The stable quotients analogue of Givental's $J$-function is given~by
\BE{Zdfn_e}
Z_{n;\a}(x,\hb,q) \equiv
1+\sum_{d=1}^{\i}\!q^d\ev_{1*}\left[\frac{e(\dot\V_{n;\a}^{(d)})}{\hb\!-\!\psi_1}\right]
\in H^*(\Pn)\big[\big[\hb^{-1},q\big]\big],\EE
where $\ev_1:\ov{Q}_{0,2}(\Pn,d)\lra\Pn$ is as before
and $x\!\in\!H^2(\Pn)$ is the hyperplane class.
For example, if $|\a|\!=\!n$, this power series
is equivalent to the set of numbers
$$\int_{\ov{Q}_{0,2}(\Pn,d)}\!\!\!\E(\dot\V_{n;\a}^{(d)})\psi_1^p\ev_1^*x^{n-2-p}\,,
\qquad d\in \Z^+,~0\!\le\!p \le n\!-\!2\,.$$
By \cite[Proposition~6.2.3]{CKM},
\BE{SQnumsdfn_e}\begin{split}
\SQ_{n;\a}^{(d)}\big(\tau_p(x^{n-2-\ell(\a)-p}),1\big)
&\equiv \int_{[\ov{Q}_{0,2}(X_{n;\a},d)]^{\vir}}\!\!\!\psi_1^p\ev_1^*x^{n-2-\ell(\a)-p}\\
&=\lr\a\int_{\ov{Q}_{0,2}(\Pn,d)}\!\!\!\E(\dot\V_{n;\a}^{(d)})\psi_1^p\ev_1^*x^{n-2-p}
\quad\forall~p\le n\!-\!2\!-\!\ell(\a)\,;\end{split}\EE
in particular, these numbers vanish if $p\!\le\!\ell^-(\a)\!-\!2$
(because $x^{n-p}\!=\!0$ on~\e_ref{VBtot_e} if $p\!\le\!\ell^+(\a)$).
The usual Givental's $J$-function, which we denote by~$Z_{n;\a}^{\GW}(x,\hb,q)$,
is defined as in~\e_ref{Zdfn_e} with $\ov{Q}_{0,2}(\Pn,d)$ replaced by~$\ov\M_{0,2}(\Pn,d)$.\\

\noindent
The hypergeometric series describing Givental's $J$-function in Gromov-Witten theory
is given~by
\BE{Ydfn_e}
Y_{n;\a}(x,\hb,q)\equiv\sum_{d=0}^{\i}q^d
\frac{\prod\limits_{a_k>0}\prod\limits_{r=1}^{a_kd}(a_kx\!+\!r\hb)
\prod\limits_{a_k<0}\!\!\prod\limits_{r=0}^{-a_kd-1}\!\!(a_kx\!-\!r\hb)}
{\prod\limits_{r=1}^{d}(x\!+\!r\hb)^n}
\in \Q[x]\big[\big[\hb^{-1},q\big]\big]\,.\EE
In the pure Calabi-Yau case, i.e.~$a_k\!>\!0$ for all $k$ and $|\a|\!=\!n$, we also need
the power series
\BE{Idfn_e}I_{n;\a}(q)=
\begin{cases}
1,&\hbox{if}~|\a|\!-\!\ell^-(\a)\!<\!n;\\
Y_{n;\a}(0,1,q) =\sum\limits_{d=0}^{\i}q^d\frac{\prod\limits_{k=1}^l(a_kd)!}{(d!)^n},&
\hbox{if}~|\a|\!-\!\ell^-(\a)\!=\!n.
\end{cases}\EE
By the following theorem, the stable quotients analogue of Givental's $J$-function
is also described by the hypergeometric series~\e_ref{Ydfn_e}, but
in a more straightforward way.

\begin{thmnum}\label{main_thm}
If $l\!\in\!\Z^{\ge0}$, $n\!\in\!\Z^+$, and $\a\!\in\!(\Z^*)^l$ are such that
$|\a|\!\le\!n$, then the stable quotients analogue of Givental's $J$-function
satisfies
\BE{ZvsY_e} Z_{n;\a}(x,\hb,q)=\frac{Y_{n;\a}(x,\hb,q)}{I_{n;\a}(q)}
\in H^*(\Pn)\big[\big[\hb^{-1},q\big]\big].\EE
\end{thmnum}

\begin{crlnum}\label{Fano2_crl}
If $l\!\in\!\Z^{\ge0}$, $n\!\in\!\Z^+$, and $\a\!\in\!(\Z^*)^l$ are such that
$|\a|\!\le\!n$ and $|\a|\!-\!\ell^-(\a)\!\le\!n\!-\!2$,  then
$$Z_{n;\a}^{\GW}(x,\hb,q)=Z_{n;\a}(x,\hb,q)\,.$$
\end{crlnum}

\begin{crlnum}\label{CY_crl}
If $l\!\in\!\Z^{\ge0}$, $n\!\in\!\Z^+$, and $\a\!\in\!(\Z^*)^l$ are such that
$|\a|\!=\!n$, then
\BE{GWvsSQmirr_e}
Z_{n;\a}^{\GW}(x,\hb,Q)=
\ne^{-J_{n;\a}(q)x/\hb}Z_{n;\a}(x,\hb,q)\,,
\qquad \hbox{where}\quad Q=q\cdot\ne^{J_{n;\a}(q)},\EE
with the standard change of variables $q\!\lra\!Q$ of Gromov-Witten theory described~by
\BE{SQvsJ_e} \lr\a J_{n;\a}(q)
=\sum_{d=1}^{\i}q^d\SQ_{n;\a}^{(d)}\big(\tau_0(x^{n-2-\ell(\a)}),1\big) .\EE
\end{crlnum}

\noindent
If $|\a|\!\le\!n$ and $|\a|\!-\!\ell^-(\a)\!\le\!n\!-\!2$, \e_ref{ZvsY_e} also holds with
$Z_{n;\a}(x,\hb,Q)$ replaced by $Z_{n;\a}^{\GW}(x,\hb,Q)$;
see \cite[Theorem~9.1]{Gi2} for the $\ell^-(\a)\!=\!0$ case
and \cite[Theorem~5.1]{Elezi} for the $\ell^-(\a)\!\ge\!1$ case.
Thus, Corollary~\ref{Fano2_crl} is an immediate consequence of Theorem~\ref{main_thm}.\\

\noindent
If $|\a|\!=\!n$ and $\ell^-(\a)\!\le\!1$, the Gromov-Witten analogue of~\e_ref{ZvsY_e}
involves a mirror transform between the power series variable on the left-hand side
(now denoted by~$Q$) and the power series variable~$q$ on the right-hand side.
It takes the~form
\BE{GWCY_e}
Z_{n;\a}^{\GW}(x,\hb,Q)=\ne^{-J_{n;\a}(q)x/\hb}\,\frac{Y_{n;\a}(x,\hb,q)}{I_{n;\a}(q)}\,,
\qquad \hbox{where}\quad Q=q\cdot\ne^{J_{n;\a}(q)},\EE
for an explicit power series $J_{n;\a}(q)\in q\cdot\Q[[q]]$;
see \cite[Theorem~11.8]{Gi2} for the $\ell^-(\a)\!=\!0$ case
and \cite[Theorem~5.1]{Elezi} for the $\ell^-(\a)\!=\!1$ case.
Along with~\e_ref{GWCY_e}, Theorem~\ref{main_thm} immediately implies
the $\ell^-(\a)\!\le\!1$ case of~\e_ref{GWvsSQmirr_e};
the $\ell^-(\a)\!\ge\!2$ case of~\e_ref{GWvsSQmirr_e}, where $J_{n;\a}(q)\!=\!0$,
follows from Corollary~\ref{Fano2_crl}.
By~\e_ref{Zdfn_e} and~\e_ref{SQnumsdfn_e}, the right-hand side of~\e_ref{SQvsJ_e}
is the coefficient of~$(h^{-1})^0$ in $Z_{n;\a}(x,\hb,q)$ times~$\lr\a$.
By the string relation of Gromov-Witten theory \cite[Section~26.3]{MirSym},
the coefficient of~$(h^{-1})^0$ in $Z_{n;\a}^{\GW}(x,\hb,q)$ is zero.
Thus, \e_ref{SQvsJ_e} follows from~\e_ref{GWvsSQmirr_e}.\\

\noindent
In the remaining case, i.e.~$|\a|\!=\!n\!-\!1$ and $\ell^-(\a)\!=\!0$,
the Gromov-Witten analogue of~\e_ref{ZvsY_e} is the relation
\BE{Fano1_e}Z_{n;\a}^{\GW}(x,\hb,q)=\ne^{-{\a}!q/\hb}\,\frac{Y_{n;\a}(x,\hb,q)}{I_{n;\a}(q)};\EE
see \cite[Theorem~10.7]{Gi2}.
Theorem~\ref{main_thm} implies that~\e_ref{Fano1_e} holds with $Y_{n;\a}(x,\hb,q)$
replaced by $Z_{n;\a}(x,\hb,q)$.
The same comparisons apply to the equivariant versions of Givental's $J$-function
for the stable quotients invariants, computed by Theorem~\ref{equiv_thm},
and of Givental's $J$-function for Gromov-Witten invariants
computed by
\cite[Theorems~9.5,10.7,11.8]{Gi2} in the $\ell^-(\a)\!=\!0$ case and
\cite[Theorem~5.3]{Elezi} in the $\ell^-(\a)\!\ge\!1$ case.
Thus, the Gromov-Witten and stable quotients invariants are  related essentially by the mirror map.
By \cite[Theorem~1]{CZ2}, the primary (without $\psi$-classes) genus~0 three-point
Gromov-Witten and stable quotients invariants of Calabi-Yau complete intersection threefolds are
related precisely by the change of variables $Q\!\lra\!q$ and the rescaling
$I_{n;\a}(q)$,  i.e.~the exponential factor
in~\e_ref{GWvsSQmirr_e} can be seen as an artifact of the presence of~$\hb$.\\

\noindent
Table~\ref{GWvsSQ_tbl} lists a few Gromov-Witten and stable quotients
invariants of the quintic threefold $X_{(5)}\!\subset\!\P^4$
obtained from \e_ref{GWCY_e} and~\e_ref{ZvsY_e}, respectively.
In the first column of this table,
\begin{equation*}\begin{split}
\GW_{5;(5)}^{(d)}\big(\tau_p(x^{2-p}),1\big)\equiv
\int_{[\ov\M_{0,2}(X_{(5)},d)]^{\vir}}\!\!\!\psi_1^p\ev_1^*x^{2-p}
=
5\int_{\ov\M_{0,2}(\P^4,d)}\!\!\!\E(\dot\V_{5;(5)}^{(d)})\psi_1^p\ev_1^*x^{3-p}\,,
\end{split}\end{equation*}
where $\dot\V_{5;(5)}^{(d)}$ is the usual analogue
of~\e_ref{Vprdfn_e} over $\ov\M_{0,2}(\P^4,d)$.
By the string, dilaton, and divisor relations \cite[Section~26.3]{MirSym},
\BE{GWrel_e}
\frac{\GW_{5;(5)}^{(d)}(\tau_1(x),1)}{d}=\deg[\ov\M_{0,0}(X_{(5)},d)]^{\vir}
=-\frac{\GW_{5;(5)}^{(d)}(\tau_2(1),1)}{2}.\EE
These relations are obtained using the forgetful maps
$$\ov\M_{0,2}(X_{(5)},d)\stackrel{f_2}\lra\ov\M_{0,1}(X_{(5)},d)
\stackrel{f_1}\lra\ov\M_{0,0}(X_{(5)},d),$$
which have no analogues in the stable quotients theory.
The middle term in~\e_ref{GWrel_e} does not have an analogue in the stable quotients theory either,
while the analogues of the outer terms in~\e_ref{GWrel_e} are not equal,
as Table~\ref{GWvsSQ_tbl} illustrates.
The numbers $\GW_d(\tau_0(x^2),1)$ vanish, since the classes $\ev_1^*x^3$ on
$\ov\M_{0,2}(\P^4,d)$ are the pullbacks by the forgetful morphisms~$f_1$
of the classes $\ev_1^*x^3$ and $\ov\M_{0,1}(\P^4,d)$.
The analogous stable quotients invariants do not vanish; see~\e_ref{SQvsJ_e}.\\

\begin{table}
\begin{center}
\begin{tabular}{||c|c|c|c|c||}
\hline\hline
$d$ &
\begin{small}
$\displaystyle\frac{\GW_{5;(5)}^{(d)}(\tau_1(x),1)}{d}=-\frac{\GW_{5;(5)}^{(d)}(\tau_2(1),1)}{2}$
\end{small}&
\begin{small} $\SQ_{5;(5)}^{(d)}(\tau_0(x^2),1)$\end{small} &
\begin{small} $\displaystyle\frac{\SQ_{5;(5)}^{(d)}(\tau_1(x),1)}{d}$\end{small} &
\begin{small}$-\displaystyle\frac{\SQ_{5;(5)}^{(d)}(\tau_2(1),1)}{2}$\end{small} \\
\hline
1& \begin{small}2875\end{small}& \begin{small}3850\end{small}&
\begin{small}2875\end{small}& \begin{small}2875\end{small}\\
\hline
2& $\frac{4876875}{8}$& \begin{small}3589125\end{small}&
$\frac{19660875}{8}$& $\frac{13731875}{8}$\\
\hline
3& $\frac{8564575000}{27}$& $\frac{16126540000}{3}$&
$\frac{76579948750}{27}$& $\frac{175851761875}{27}$\\
\hline
4& $\frac{15517926796875}{64}$& $\frac{19736572853125}{2}$&
$\frac{801135363990625}{192}$& $\frac{1123498525946875}{576}$\\
\hline
5& \begin{small}229305888887648\end{small}& \begin{small}\!\!20310770587807020\end{small}\!\!&
\!\!$\frac{14274970288322171}{2}$\!\!& \!\!$\frac{125303832133435229}{48}$\!\!\\
\hline\hline
\end{tabular}
\end{center}
\caption{Some genus~0 GW- and SQ-invariants of the quintic threefold $X_{(5)}$}
\label{GWvsSQ_tbl}
\end{table}

\noindent
It is interesting to observe that the numbers $d\SQ_d(\tau_0(x^2),1)$ are integers
if $(n,\a)\!=\!(5,(5))$ and $d\!\le\!1000$;
as noted in~\cite[Section~1]{CZ2}, the same is the case for the numbers
$d\SQ_d(\tau_0(x),\tau_0(x))$.
Two-point GW-invariants of such form are equal to three-point primary GW-invariants,
which are integers, when the target is a Calabi-Yau, due to symplectic topology
considerations; see \cite[Section~7.3]{MS} and \cite{RT}, for example.
Since the stable quotients invariants are purely algebro-geometric objects,
the apparent integrality of the primary invariants $d\SQ_d(\cdot,\cdot)$ suggests that
there should be an algebro-geometric reason behind the integrality of these numbers,
as well as of the closely related three-point GW-invariants.\\

\noindent
As in the case of mirror symmetry for Gromov-Witten invariants, Theorem~\ref{main_thm}
follows immediately from its $\T^n$-equivariant version,
Theorem~\ref{equiv_thm} in Section~\ref{equivthm_sec}.
The latter is proved using the Atiyah-Bott localization theorem~\cite{ABo} on
$\ov{Q}_{0,2}(\Pn,d)$, which reduces the equivariant version of the power series~\e_ref{Zdfn_e},
the power series~$\cZ_{n;\a}(\x,\hb,q)$ defined by~\e_ref{cZdfn_e} below,
to a sum of rational functions over certain graphs.
As in the case of Gromov-Witten invariants, $\cZ_{n;\a}(\x,\hb,q)$ is $\fC$-recursive
in the sense of Definition~\ref{recur_dfn}, with the collection~$\fC$ of structure
coefficients given by~\e_ref{Cdfn_e},
and satisfies the self-polynomiality condition of Definition~\ref{SPC_dfn};
the same is the case of the equivariant version of the power series~\e_ref{Ydfn_e},
the power series $\cY_{n;\a}(\x,\hb,q)$ defined by~\e_ref{cYdfn_e}.
Thus, the two power series
$$\cY_{n;\a}(\x,\hb,q),\cZ_{n;\a}(\x,\hb,q)\in
H_{\T}^*(\Pn)\big[\big[\hb^{-1},q\big]\big]$$
are determined by their $\mod(\hb^{-1})^2$ part; see Proposition~\ref{uniqueness_prp}.
It is straightforward to determine the $\mod(\hb^{-1})^2$-part of
the power series~$\cY_{n;\a}$.
The $\mod(\hb^{-1})^2$-part of Givental's $J$-function in Gromov-Witten theory is~1 in all cases
for a simple geometric reason.
This approach thus confirms the analogue of Theorem~\ref{main_thm} in
{\it Gromov-Witten} theory
and thus mirror symmetry for the genus~0 Gromov-Witten invariants
of projective complete intersections.\\

\noindent
In the {\it stable quotients} theory,
the situation with the $\mod(\hb^{-1})^2$-part of $\cZ_{n;\a}$ is
different.
It is still~1, for dimensional reasons, if $|\a|\!\le\!n\!-\!2$.
If $|\a|\!=\!n\!-\!1$, the $\mod(\hb^{-1})^2$-part of $\cZ_{n;\a}$
vanishes in the $q$-degrees~2 and higher;
it is straightforward to see that the coefficient of~$q^1$ $\mod(\hb^{-1})^2$
is $\a!/\hb$ if $\ell^-(\a)\!=\!0$ and~$0$
otherwise.\footnote{Even this is not necessary due to our approach to the Calabi-Yau case.}
So, in these cases, the proof of mirror symmetry for Gromov-Witten invariants carries over
to the stable quotients invariants.
However, in the Calabi-Yau case, $|\a|\!=\!n$, the $\mod(\hb^{-1})^2$-part of $\cZ_{n;\a}$
is not zero in all $q$-degrees if $\ell^-(\a)\!\le\!1$, and we see no a priori reason
for the coefficients of positive $q$-degrees to vanish even if $\ell^-(\a)\!\ge\!2$.
Thus, the proof of mirror symmetry for Gromov-Witten invariants can{\it not}
directly carry over to the stable quotients invariants in the Calabi-Yau cases.\\

\noindent
Since the coefficients of~$q^0$ on the two sides of the identity in Theorem~\ref{equiv_thm}
are the same (both are~1), it is equivalent to the equality of the auxiliary coefficients,
$\cY_i^r(d)$ and $\cZ_i^r(d)$, in the recursions~\e_ref{recurdfn_e2}
for~$\cY_{n;\a}$ and~$\cZ_{n;\a}$, respectively.
By a direct algebraic computation, the coefficients~$\cY_i^r(d)$ are expressible
in terms of  certain residues of~$\cY$; see Lemma~\ref{cYrec_lmm}.
Analyzing the relevant graphs, one can show that the coefficients~$\cZ_i^r(d)$ are
likewise expressible in terms of certain residues of~$\cZ$,
but in a different way; see Proposition~\ref{cZrec_prp}.
Thus, for each pair $(n,\a)$ with $|\a|\!\le\!n$,
the identity in Theorem~\ref{equiv_thm} is equivalent to certain identities for
the residues of~$\cY_{n;\a}$; see Lemma~\ref{cYcZcomp_lmm}.
Since $\cY_{n;\a}\!=\!\cZ_{n;\a}$ whenever $|\a|\!\le\!n\!-\!2$,
these identities hold whenever $|\a|\!\le\!n\!-\!2$.
By the proof of Proposition~\ref{cYcZrel_prp},
the validity of these identities is independent of~$n$,
and thus they hold for all pairs $(n,\a)$.
This yields Theorem~\ref{equiv_thm} and thus Theorem~\ref{main_thm}.\\

\noindent
The relations of Lemma~\ref{cYcZcomp_lmm} involve twisted Hurwitz numbers arising
from certain moduli spaces of weighted stable curves $\ov\cM_{0,2|d}$;
see Section~\ref{SQdfn_sec}.
These relations in turn uniquely determine the twisted Hurwitz numbers, even equivariantly,
in terms of a key power series associated with~$\cY_{n;\a}$;
see Theorems~\ref{main0_thm} and~\ref{equiv0_thm} in Sections~\ref{SQdfn_sec}
and~\ref{equivthm_sec}, respectively.
Based on developments in Gromov-Witten theory,
one would expect these closed formulas to be a key ingredient in computing twisted genus~1
stable quotients invariants and thus answering yet another question raised in~\cite{MOP09}.\\

\noindent
The proof that the equivariant version of Givental's $J$-function in Gromov-Witten theory
satisfies the self-polynomiality condition of Definition~\ref{SPC_dfn} uses
the localization theorem~\cite{ABo} to compute integrals over the moduli
space $\ov\M_{0,2}(\P^1\!\times\!\Pn,(1,d))$.
Our proof that the equivariant stable quotients analogue of  Givental's $J$-function
satisfies the self-polynomiality condition uses the moduli space
of stable pairs of quotients $\ov{Q}_{0,2}(\P^1\!\times\!\Pn,(1,d))$
in a similar way; see Section~\ref{SPC_sec}.
This moduli space is a special case of the moduli space
$$\ov{Q}_{g,m}\big(\P^{n_1-1}\!\times\!\ldots\!\times\!\P^{n_p-1},(d_1,\ldots,d_p)\big)$$
of stable $p$-tuples of quotients, which we describe in Section~\ref{SQdfn_sec}
by extending the notion of stable quotients introduced in~\cite{MOP09}.\\

\noindent
The Gromov-Witten analogues of Theorem~\ref{main_thm} and its equivariant version, Theorem~\ref{equiv_thm} in Section~\ref{equivthm_sec},
extend to the so-called \sf{concavex bundles} over products of
projective spaces, i.e.~vector bundles of the form
$$\bigoplus_{k=1}^l\cO_{\P^{n_1-1}\times\ldots\times\P^{n_p-1}}(a_{k;1},\ldots,a_{k;p})
\lra \P^{n_1-1}\!\times\!\ldots\!\times\!\P^{n_p-1},$$
where for each given $k\!=\!1,2,\ldots,l$ either $a_{k;1},\ldots,a_{k;p}\!\in\!\Z^{\ge0}$
or $a_{k;1},\ldots,a_{k;p}\!\in\!\Z^-$.
The stable quotients analogue of these bundles are the sheaves
\BE{gensheaf_e}
\bigoplus_{k=1}^l\cS_1^{*a_{k;1}}\!\otimes\!\ldots\!\otimes\cS_p^{*a_{k;p}}
\lra\cU
\lra \ov{Q}_{0,2}\big(\P^{n_1-1}\!\times\!\ldots\!\times\!\P^{n_p-1},(d_1,\ldots,d_p)
\big)\EE
with the same condition on~$a_{k;i}$,
where $\cS_i\!\lra\!\cU$ is the universal subsheaf corresponding to the $i$-th factor;
see Section~\ref{SQdfn_sec}.
In this case, we compare two power series
\begin{alignat}{1}
\label{genY_e}
Y_{n_1,\ldots,n_p;\a}(x_1,\ldots,x_p,\hb,q_1,\ldots,q_p)
&\in \Q[x_1,\ldots,x_p]\big[\big[\hb^{-1},q_1,\ldots,q_p\big]\big],\\
\label{genZ_e}
Z_{n_1,\ldots,n_p;\a}(x_1,\ldots,x_p,\hb,q_1,\ldots,q_p)&\in
H^*\big(\P^{n_1-1}\!\times\!\ldots\!\times\!\P^{n_p-1}\big)
\big[\big[\hb^{-1},q_1,\ldots,q_p\big]\big],
\end{alignat}
where $x_1,\ldots,x_p\!\in\!H^*(\P^{n_1-1}\!\times\!\ldots\!\times\!\P^{n_p-1})$
are the pullbacks of the hyperplane classes by the projection maps.
The coefficient of $q_1^{d_1}\!\ldots\!q_p^{d_p}$ in~\e_ref{genZ_e}
is defined by the same pushforward as in~\e_ref{Zdfn_e}, with the degree~$d$
of the stable quotients replaced by~$(d_1,\ldots,d_p)$.
The coefficient of $q_1^{d_1}\!\ldots\!q_p^{d_p}$ in~\e_ref{genY_e}
is given~by
$$\frac{\prod\limits_{a_{k;1}\ge0}
\prod\limits_{r=1}^{\sum\limits_{s=1}^p\!a_{k;s}d_s}
\hspace{-.2in}\big(\sum\limits_{s=1}^p\!a_{k;s}x_s\!+\!r\hb\big)
\prod\limits_{a_{k;1}<0}\!\!
\prod\limits_{r=0}^{-\sum\limits_{s=1}^p\!a_{k;s}d_s\,-1}
\hspace{-.3in}\big(\sum\limits_{s=1}^p\!a_{k;s}x_s\!-\!r\hb\big)}
{\prod\limits_{s=1}^p\prod\limits_{r=1}^{d_s}(x_s\!+\!r\hb)^{n_s}}\,.$$
The condition $|\a|\!\le\!n$ should be replaced by the conditions
$$|a_{1;s}|\!+\!\ldots\!+\!|a_{l;s}| \le n_s\qquad\forall~s=1,\ldots,p.$$
Our proof of Theorem~\ref{equiv_thm} (and thus of Theorem~\ref{main_thm}) extends
directly to this situation;
we will comment on the necessary modifications in each step of the proof.\\

\noindent
Mirror formulas for the two-point versions of \e_ref{Zdfn_e} and \e_ref{cZdfn_e},
i.e.~with $\ev_1$ and $(\hb\!-\!\psi_1)$ replaced by $\ev_1\!\times\!\ev_2$
and $(\hb_1\!-\!\psi_1)(\hb_2\!-\!\psi_2)$, as well as their generalizations
to products of projective spaces, can now be readily obtained using
the approaches of \cite{bcov0,Po} in Gromov-Witten theory; see~\cite{CZ2}.
They are related to the corresponding formulas in Gromov-Witten theory in
the same ways as the one-point formulas; see the paragraph following Theorem~\ref{main_thm}.
Similarly to developments in Gromov-Witten theory,
these two-point genus~0 formulas are one of the key steps in computing twisted genus~1
stable quotients invariants in~\cite{CZ3}.\\

\noindent
A notable feature of the mirror formula of Theorem~\ref{main_thm} and
its two-point analogue is that they are invariant under replacing
$(n,(a_1,\ldots,a_k))$ by $(n\!+\!1,(a_1,\ldots,a_k,1))$;
their extensions to products of projective spaces
have a similar feature.\footnote{This replacement does not change the total space
of the vector bundle~\e_ref{VBtot_e}}
This is consistent with \cite[Proposition~6.4.1]{CKM}.\\

\noindent
We would like to thank the referee for the detailed comments and suggestions
that helped improve the exposition,
R.~Pandharipande for bringing moduli spaces of stable
quotients and the problems addressed in this paper to our attention,
as well as for helpful comments,
A.~Popa for suggesting precise references for mirror theorems in
Gromov-Witten theory and pointing out typos in the original version of this paper,
and M.~Alim, I.~Ciocan-Fontanine,
M.~Gross, J.~Koll\'ar, R.~Plesser, and E.~Witten for useful discussions.
The second author would also like to thank the School of Mathematics at IAS
for its hospitality during the period when the results in this paper were obtained
and the paper itself was completed.

\section{Moduli spaces of stable quotients}
\label{SQdfn_sec}

\noindent
We begin this section by reviewing the notion of stable quotients for
products of projective spaces.
Propositions~\ref{prodmod_prp} and~\ref{prodsmooth_prp}
describing moduli spaces of such objects
are a special case of  \cite[Theorems 3.2.1,~4.0.1]{CK}
and precisely the statement of \cite[Example 7.2.6]{CK}, respectively.
We include proofs of these statements, extending~\cite{MOP09}
from the case of projective spaces, for the sake of completeness,
since \cite{CK} treats the general toric case and is thus more involved.
We then introduce related moduli spaces of weighted curves.
We conclude this section with a closed formula for twisted Hurwitz numbers arising
from integrals over these moduli spaces of curves; see Theorem~\ref{main0_thm}.\\

\noindent
By a \sf{nodal genus~$g$ curve}, we will mean a reduced connected scheme $\cC$ over $\C$
of pure dimension~1 with at worst nodal singularities and
$h^1(\cC,\cO_{\cC})\!=\!g$.
Let $\cC^*\!\subset\!\cC$ denote the nonsingular locus of such a curve.
A  \sf{quasi-stable genus~$g$ $m$-marked curve} is a tuple $(\cC,y_1,\ldots,y_m)$
consisting of a nodal genus~$g$ curve and distinct points $y_i\!\in\!\cC^*$.
A \sf{(corank~1) quasi-stable quotient of the trivial rank~$n$ sheaf} on such a curve
is a rank~1 subsheaf $S\!\subset\!\C^n\!\otimes\!\cO_{\cC}$ such that
the corresponding quotient sheaf~$Q$, given~by
$$0\lra S\lra \C^n\!\otimes\!\cO_{\cC}\lra Q\lra0\,,$$
is locally free on $(\cC\!-\!\cC^*)\!\cup\!\{y_1,\ldots,y_m\}$,
i.e.~at the nodes and markings of~C.
A tuple $(S_1,\ldots,S_p)$ of quasi-stable quotients on $(\cC,y_1,\ldots,y_m)$ is \sf{stable}
if the $\Q$-line bundle
$$\om_{\cC}(y_1\!+\!\ldots\!+\!y_m)\otimes
\big(S_1^*\otimes\ldots\otimes S_p^*\big)^{\ep} \lra \cC$$
is ample for all $\ep\!\in\!\Q^+$;
this implies that $2g\!-\!2\!+\!m\!\ge\!0$.
An \sf{isomorphism}
$$\phi\!:(\cC,y_1,\ldots,y_m,S_1,\ldots,S_p)\lra
(\cC',y_1',\ldots,y_m',S_1',\ldots,S_p')$$
between tuples of quasi-stable quotients is an isomorphism $\phi\!:\cC\!\lra\!\cC'$
such that
$$\phi(y_i)=y_i'~~~\forall~i=1,\ldots,m, \qquad
\phi^*S_j'=S_j\subset\C^{n_j}\!\otimes\!\cO_{\cC}~~~\forall~j=1,\ldots,p.$$
The automorphism group of any stable tuple of quotients is finite.

\begin{prp}\label{prodmod_prp}
If $g,m,d_1,\ldots,d_p\!\in\!\Z^{\ge0}$ and
$n_1,\ldots,n_p\!\in\!\Z^+$, the moduli space
\BE{Qgm_e}
\ov{Q}_{g,m}\big(\P^{n_1-1}\!\times\!\ldots\!\times\!\P^{n_p-1},
(d_1,\ldots,d_p)\big)\EE
parameterizing the stable $p$-tuples of quotients
\BE{SQtuple_e}\big(\cC,y_1,\ldots,y_m,S_1,\ldots,S_p\big),\EE
with $h^1(\cC,\cO_{\cC})\!=\!g$, $S_i\!\subset\!\C^{n_i}\!\otimes\!\cO_{\cC}$,
 and $\deg(S_i)\!=\!-d_i$, is a separated and proper
Deligne-Mumford stack of finite type over~$\C$
and carries a canonical two-term obstruction theory.
\end{prp}

\begin{proof}
The construction of $\ov{Q}_{g,m}(\Pn,d)$ in~\cite{MOP09}
carries through with minor changes.
We sketch the modification here.\\

\noindent
I. {\it Construction of the moduli space}.
Let $g,m,d_1,...,d_p$ satisfy
$$2g\!-\!2\!+\!m\!+\!\ep(d_1\!+\!...\!+\!d_p)>0 \qquad\forall~\ep\!>\!0.$$
Let $d=d_1\!+\!\ldots\!+\!d_p$.
Fix a stable $p$-tuple of  quotients $(\cC,y_1,\ldots,y_m,S_1,\ldots,S_p)$, where
\BE{SQses_e}0\lra S_i\lra \C^{n_i}\!\otimes\!\O_{\cC}\lra Q_i \lra 0\,.\EE
By assumption, the line bundle
$$\cL_{\ep} = \om_{\cC}(y_1\!+\!\ldots\!+\!y_m)
 \otimes\big(S_1^*\!\otimes\!\ldots\!\otimes\!S_p^* \big)^{\ep}$$
is ample for all $\ep\!>\!0$.
Fix $\ep=1/(d\!+\!1)$ and let $f=5(d\!+\!1)$.
By \cite[Lemma~5]{MOP09}, the line bundle $\cL_{\ep}^f$ is very ample
and has no higher cohomology.
Therefore,
$$h^0(\cC,\cL_{\ep}^f) = 1 - g + 5(d\!+\!1)(2g\!-\!2\!+\!m)+5d$$
is independent of the choice of the stable $p$-tuple of  quotients.
Let
$$V =H^0(\cC,\cL_{\ep}^f)^*\,.$$
The line bundle $\cL_{\ep}^f$ induces an embedding $\io\!:\cC\hookrightarrow\P(V)$.
Let \sf{Hilb} denote the Hilbert scheme of curves in $\P(V)$ of genus $g$ and degree
$$5(d\!+\!1)(2g\!-\!2\!+\!m)+5d=\deg \cL_{\ep}^f\,.$$
Each stable quotient gives rise to a point in
$$\cH = \textsf{Hilb} \times \P(V)^m\,,$$
where the last factors record the locations of the markings $y_1,\ldots,y_m$.\\

\noindent
Points in $\cH$ correspond to tuples $(\cC,y_1,\ldots,y_m)$.
Denote by $\cH'\!\subset\!\cH$ the quasi-projective subscheme consisting of
the tuples such~that
\begin{enumerate}[label=(\roman*),leftmargin=*]
\item the points $y_1,\ldots,y_m$ are contained in $\cC$,
\item the curve $(\cC,y_1,\ldots,y_m)$ is quasi-stable.
\setcounter{saveenumi}{\arabic{enumi}}
\end{enumerate}
Let $\pi\!:\cU'\!\lra\!\cH'$ be the universal curve over~$\cH'$.
For $i\!=\!1,\ldots,p$, let
$$\sf{Quot}(n_i,d_i)\lra\cH',$$
be the $\pi$-relative Quot scheme parameterizing rank $n_i\!-\!1$ degree $d_i$ quotients
$$0 \lra S_i \lra \C^{n_i}\!\otimes\!\O_{\cC} \lra Q_i \lra 0$$
on the fibers of~$\pi$.
Denote by $\cQ$ be the fiber product
$$\cQ = \sf{Quot}(n_1,d_1)\times_{\cH'}\ldots\times_{\cH'}
 \sf{Quot}(n_p,d_p)\lra \cH'$$
and $\cQ'\!\subset\!\cQ$ the  locally closed subscheme consisting of the tuples such that
\begin{enumerate}[label=(\roman*),leftmargin=*]
\setcounter{enumi}{\arabic{saveenumi}}
\item $Q_i$ is locally free at the nodes and at the marked points of $\cC$,
\item the restriction of $\O_{\P(V)}(1)$ to $\cC$ agrees with the line bundle
$$\big(\om_{\cC}(y_1\!+\!\ldots\!+\!y_m)\big)^{5(d+1)}\otimes
\big(S_1^*\!\otimes\!\ldots\!\otimes\!S_p^*\big)^5.$$
\end{enumerate}
The action of $\PGL(V)$ on $\cH$ induces actions on $\cH'$ and $\cQ'$.
A $\PGL(V)$-orbit in $\cQ'$ corresponds to a stable quotient up to isomorphism.
By stability, each orbit has finite stabilizers.
The moduli space~\e_ref{Qgm_e} is the stack quotient $[\cQ'/\PGL(V)]$.\\

\noindent
II. {\it Separateness}.
We prove that the moduli stack~\e_ref{Qgm_e} is separated by the valuative criterion.
Let $(\De,0)$ be a nonsingular pointed curve and $\De^0=\De\!-\!\{0\}$.
Take two flat families of quasi-stable pointed curves
$$\cX^j\lra\De, \qquad y^j_1,\ldots,y^j_m\!: \De\lra \cX^j,$$
and two flat families of stable quotients
$$0\lra S^j_i \lra \C^{n_i}\!\otimes\!\cO_{\cX_i}\lra Q^j_i \lra0,$$
with $j\!=\!1,2$ and $i\!=\!1,\ldots,p$.
Assume the two families are isomorphic away from the central fiber.
By \cite[Section 6.2]{MOP09}, an isomorphism between these two families over $\De\!-\!0$
extends to the families of curves $\cX^j\!\lra\!\De$ in a manner preserving the sections and
hence extends to each pair of families of stable quotients.\\

\noindent
III. {\it Properness.}
We prove the moduli stack~\e_ref{Qgm_e} is proper, again by the valuative criterion.
Let
$$\pi^0\!: \cX^0\lra\De^0, \qquad y_1,\ldots,y_m\!: \De^0\lra\cX^0$$
carry a flat family of stable $p$-tuples of  quotients
$$0\lra S_i \lra \C^{n_i}\!\otimes\!\O_{\cX^0}\lra Q_i \lra 0.$$
By  \cite[Section~6.3]{MOP09}, each stable quotient individually extends,
possibly after base-change, and hence the $p$-tuple extends.
In particular, the blowup procedure in \cite[Section~6.3]{MOP09} yielding
the sheaf~$\wt{S}$ in \cite[(18)]{MOP09} can be applied to each sheaf~$S_i$
separately to yield sheaves $\ti{S}_i$ over a flat family $\wt\cX\!\lra\!\De$
so that the corresponding quotients~$\ti{Q}_i$ are
locally free at the nodes and at the marked points of the central fiber.
After a base change and altering each quotient sheaf at finitely points,
we obtain a flat family of quasi-stable quotients $Q_i''$
over a flat family as in \cite[(19)]{MOP09}.
The final blowdown step of \cite[Section~6.3]{MOP09} is applied
with the unstable genus~0 curves~$P$ such that $S_i''|_P\!=\!\cO_P$
for all $i\!=\!1,\ldots,p$ and the line bundle~$\fL$ obtained from
the one in~\cite{MOP09} by replacing $\La^r(S'')$ with $S_1''\!\otimes\!\ldots\!\otimes\!S_p''$.
The resulting $p$-tuple of push-forward sheaves over the central fiber
is then stable.\\

\noindent
IV. {\it Obstruction Theory.}
We follow the argument in \cite[Section 3.2]{MOP09}.
Let $\phi\!:\cC\!\lra\!\cM_{g,m}$ be the universal curve over
the Artin stack of pointed curves and
$\bfQ(n,d)\!\lra\!\cM_{g,m}$ be the relative Quot scheme of
rank $n\!-\!1$ degree $d$ quotients of $\C^n\!\otimes\!\O_{\cC}$ along the fibers of~$\phi$.
Denote~by
$$\bfQ'(n,d)\subset \bfQ(n,d)$$
the locus consisting of locally free subsheaves
and by
$$\nu\!:\bfQ'\equiv \bfQ'(n_1,d_1)\times_{\cM_{g,m}}\ldots\times_{\cM_{g,m}}
\bfQ'(n_p,d_p) \times_{\cM_{g,m}}\cC\lra\cM_{g,m}$$
the fiber product.
The universal sequence of sheaves
$$0 \lra \cS \lra \C^n\!\otimes\!\O_{\cC} \lra Q \lra 0$$
over $\bfQ'(n,d)\!\times_{\cM_{g,m}}\!\cC$ gives rise to a universal sequence
$$ 0\lra \bigoplus_{i=1}^p\cS_i\lra
\bigoplus_{i=1}^p (\C^{n_i}\!\otimes\!\O_{\cC}) \lra \bigoplus_{i=1}^p Q_i \lra 0$$
over $\bfQ'\!\times_{\cM_{g,m}}\!\cC$.
Let $\pi\!:\bfQ'\!\times_{\cM_{g,m}}\!\cC\lra\bfQ'$ be the projection map.
By \cite[Proposition~4.4.4]{Sernesi} with
$$\K=\bigoplus_{i=1}^p\cS_i\,,\qquad
 \cH=\bigoplus_{i=1}^p\C^{n_i}\!\otimes\!\O_{\cC}\,,
  \quad\hbox{and}\quad  \F=\bigoplus_{i=1}^p Q_i\,,$$
the relative deformation-obstruction theory of $\nu\!:\bfQ'\!\lra\!\cM_{g,m}$
is given~by
$$RHom_{\pi}(\cS_1,Q_1)\oplus\ldots \oplus RHom_{\pi}(\cS_p,Q_p)
=\bigoplus_{i=1}^p R\pi_*Hom(\cS_i,Q_i);$$
the equality above holds because each $\cS_i$ is a locally free sheaf.
By \cite[Section~2]{MO}, $R\pi_*Hom(\cS_i,Q_i)$ can be resolved by a two-step complex
of vector bundles.
Thus,
$$\nu^A\!: \ov{Q}_{g,m}\big(\P^{n_1-1}\!\times\!\ldots\!\times\!\P^{n_p-1},
(d_1,\ldots,d_p)\big) \lra \cM_{g,m}$$
admits a two-term relative deformation-obstruction theory.
Along with the smoothness of~$\cM_{g,m}$,
this induces  an absolute two-term deformation-obstruction theory of
the moduli space~\e_ref{Qgm_e}; see \cite[Appendix~B]{GP}.
\end{proof}

\begin{prp}[{\cite[Example 7.2.6]{CK}}]\label{prodsmooth_prp}
If $g\!=\!0$ or $(g,m)\!=\!(1,0)$ and $d_1,\ldots,d_p,n_1,\ldots,n_p\!\ge\!1$,
the moduli space
\BE{Qgm_e2}
\ov{Q}_{g,m}\big(\P^{n_1-1}\!\times\!\ldots\!\times\!\P^{n_p-1},
(d_1,\ldots,d_p)\big)\EE
is a nonsingular irreducible Deligne-Mumford stack of the expected dimension.
\end{prp}

\begin{proof} By part IV in the proof of Proposition~\ref{prodmod_prp},
the moduli space~\e_ref{Qgm_e2} is smooth at a
point $(\cC,y_1,\ldots,y_m,S_1,\ldots,S_p)$ if
\BE{Ext1van_e}\bigoplus_{i=1}^p \Ext^1(S_i,Q_i) = 0.\EE
Since each $S_i$ is locally free, this is the case if
\BE{H1van_e} H^1(S_i^*\!\otimes\!Q_i)=0\EE
for each $i\!=\!1,\ldots,p$.
From the cohomology long exact sequence for the short exact sequence
$$ 0 \lra \O_{\cC} \lra \C^{n_i}\!\otimes\!S_i^* \lra Q_i\!\otimes\!S_i^* \lra 0,$$
we see that \e_ref{H1van_e} holds if $H^1(S_i^*)\!=\!0$.\\

\noindent
If $g\!=\!0$, $\cC$ is a rational curve and
thus there are no special line bundles on~$\cC$ that have a nonnegative degree on
every component of~$\cC$.
If $(g,m)\!=\!(1,0)$, then  $\cC$ is either a nonsingular curve of genus~1 or
a cycle of rational curves;
thus, there are no special line bundles of positive degree on~$\cC$
that have nonnegative degree on each component of~$\cC$.
In either case, we conclude that $H^1(S_i^*)\!=\!0$ for each $i\!=\!1,\ldots,p$
and so~\e_ref{Ext1van_e} holds.
Thus, the moduli space~\e_ref{Qgm_e2} is smooth at every point and
hence is a nonsingular Deligne-Mumford stack of the expected dimension.\\

\noindent
It remains to show that it is also irreducible.
Let $U$ denote the open locus in the moduli space where the domain curve is smooth.  
In the $g\!=\!0$ case, $U$ is dominated by the product of projective spaces
$(\P^1)^m\!\times\!\prod_i\Proj(H^0(\cO(d_i))^{n_i})$.  
In the $(g,m)\!=\!(1,0)$ case, $U$ is dominated by the bundle 
$\prod_i\Proj(H^0(\cO(dp))^{n_i})$ over $\ov\cM_{1,1}$,
where $p$ is the marked point. 
Thus, $U$ is irreducible in both cases.
Since the moduli space~\e_ref{Qgm_e2} is unobstructed, $U$ is dense in~\e_ref{Qgm_e2}
and thus the latter is also irreducible.
\end{proof}

%

%

\noindent
A stable tuple as in \e_ref{SQtuple_e} such that each quotient sheaf
$Q_i=\C^{n_i}\!\otimes\!\cO_{\cC}/S_i$ is locally free corresponds
to a stable morphism
$$\cC\lra
\P^{n_1-1}\!\times\!\ldots\!\times\!\P^{n_p-1}$$
with marked points $y_1,\ldots,y_m$.
As in the $p\!=\!1$ case considered in \cite[Section~3.1]{MOP09},
there are evaluation morphisms
$$\ev_i\!: \ov{Q}_{g,m}\big(\P^{n_1-1}\!\times\!\ldots\!\times\!\P^{n_p-1},
(d_1,\ldots,d_p)\big)\lra
\P^{n_1-1}\!\times\!\ldots\!\times\!\P^{n_p-1}$$
with  $i\!=\!1,2,\ldots,m$.
There is also a universal curve
$$\pi\!: \cU\lra
\ov{Q}_{g,m}\big(\P^{n_1-1}\!\times\!\ldots\!\times\!\P^{n_p-1},
(d_1,\ldots,d_p)\big)$$
with $m$ sections $\si_1,\ldots,\si_m$  and
universal rank~1 subsheaves $\cS_i\subset \C^{n_i}\!\otimes\!\cO_{\cU}$.\\

\noindent
We will also need a certain moduli space of weighted curves;
this is the stable quotients counterpart of the Deligne-Mumford moduli space of
stable genus~$g$ marked curves in Gromov-Witten theory.
A \sf{$d$-tuple of~flecks} on a  quasi-stable $m$-marked curve $(\cC,y_1,\ldots,y_m)$
is a $d$-tuple $(\hat{y}_1,\ldots,\hat{y}_d)$ of points
of $\cC^*\!-\!\{y_1,\ldots,y_m\}$.
Such a tuple is \sf{stable} if the $\Q$-line bundle
$$\om_{\cC}\big(y_1\!+\!\ldots\!+\!y_m+\ep(\hat{y}_1\!+\!\ldots\!+\!\hat{y}_d)\big)
\lra \cC$$
is ample for all $\ep\!\in\!\Q^+$;
this again implies that $2g\!-\!2\!+\!m\!\ge\!0$.
An \sf{isomorphism}
$$\phi\!:(\cC,y_1,\ldots,y_m,\hat{y}_1,\ldots,\hat{y}_d)
\lra (\cC',y_1',\ldots,y_m',\hat{y}_1',\ldots,\hat{y}_d')$$
between curves with $m$ marked points and $d$ flecks is an isomorphism $\phi\!:\cC\!\lra\!\cC'$
such that
$$\phi(y_i)=y_i'~~~\forall~i=1,\ldots,m, \qquad
\phi(\hat{y}_j)=\hat{y}_j'~~~\forall~j=1,\ldots,d.$$
The automorphism group of any stable curve with $m$ marked points and $d$ flecks is finite.

\begin{prp}\label{torsionmod_prp}
If $g,m,d\!\in\!\Z^{\ge0}$, the moduli space $\ov\cM_{g,m|d}$
parameterizing the stable genus~$g$ curves with $m$ marked points and $d$ flecks,
\BE{torsiontuple_e}\big(\cC,y_1,\ldots,y_m,\hat{y}_1,\ldots,\hat{y}_d\big),\EE
is a nonsingular, irreducible, proper Deligne-Mumford stack.
\end{prp}

\begin{proof}
The moduli space $\ov\cM_{g,m|d}$ is the moduli space of weighted pointed stable curves,
defined in \cite[Section~2]{Hassett}, with $m$~points of weight~1 and
$d$~points of weight $1/d$ (if $d\!>\!0$).
Thus, this proposition is a special case of \cite[Theorem~2.1]{Hassett}.
\end{proof}

\noindent
Any tuple as in~\e_ref{torsiontuple_e} induces a quasi-stable quotient
$$\cO_{\cC}\big(-\hat{y}_1-\ldots-\hat{y}_d\big)\subset
\cO_{\cC} \equiv \C^1\!\otimes\!\cO_{\cC}\,.$$
For any ordered partition $d\!=\!d_1\!+\!\ldots\!+\!d_p$ with $d_1,\ldots\!,d_p\!\in\!\Z^{\ge0}$,
this correspondence gives rise to a morphism
$$\ov\cM_{g,m|d}\lra
\ov{Q}_{g,m}\big(\P^0\!\times\!\ldots\!\times\!\P^0,(d_1,\ldots,d_p)\big).$$
In turn, this morphism induces an isomorphism
\BE{curvquot_e}\phi\!:\ov\cM_{g,m|d}\big/\bS_{d_1}\!\times\!\ldots\!\times\!\bS_{d_p}
\stackrel{\sim}{\lra}
\ov{Q}_{g,m}\big(\P^0\!\times\!\ldots\!\times\!\P^0,(d_1,\ldots,d_p)\big),\EE
with the symmetric group $\bS_{d_1}$ acting on $\ov\cM_{g,m|d}$ by permuting
the points $\hat{y}_1,\ldots,\hat{y}_{d_1}$,
$\bS_{d_2}$ acting on $\ov\cM_{g,m|d}$ by permuting
the points $\hat{y}_{d_1+1},\ldots,\hat{y}_{d_1+d_2}$, etc.\\

\noindent
There is again  a universal curve
$$\pi\!: \cU\lra \ov\cM_{g,m|d}$$
with sections $\si_1,\ldots,\si_m$ and $\hat\si_1,\ldots,\hat\si_d$.
Let
\BE{psicurvedfn_e}\psi_i=-\pi_*(\si_i^2),~
 \hat\psi_j=-\pi_*(\hat\si_j^2)  \in H^2\big(\ov\cM_{g,m|d}\big)\EE
be the first Chern classes of the universal cotangent line bundles.

\begin{lmm}[{\cite[Section 4.5]{MOP09}}]\label{M02_lmm}
If $d\!\in\!\Z^+$ and $a_1,a_2,b_1,\ldots,b_d\!\in\!\Z^{\ge0}$, then
\BE{psiint_e}\int_{\ov\cM_{0,2|d}}\!\!\!\!\psi_1^{a_1}\psi_2^{a_2}
\hat\psi_1^{b_1}\!\ldots\!\hat\psi_d^{b_d}
=\binom{d\!-\!1}{a_1,a_2}
\cdot\begin{cases}
1,&\tn{if}~b_1,\ldots,b_d\!=\!0;\\
0,&\tn{otherwise}.
\end{cases}\EE
\end{lmm}

\begin{proof}
If $d\!>\!1$, there is a forgetful morphism
$$f\!: \ov\cM_{0,2|d}\lra\ov\cM_{0,2|d-1},$$
dropping the fleck $\hat{y}_d$.
For $i\!=\!1,2$, let $D_i\subset\ov\cM_{0,2|d}$ denote the divisor whose generic element consists
of two components, with one of them containing $y_i$ and $\hat{y}_d$
(and no other marked points).
By \e_ref{psicurvedfn_e},
\BE{psipullback_e}\psi_i=f^*\psi_i+D_i~~~\forall~i=1,2, \qquad
\hat\psi_j=f^*\psi_j~~~\forall~j=1,\ldots,d\!-\!1.\EE
Under the canonical identification of $D_i\approx\ov\cM_{0,2|d-1}\times\ov\cM_{0,2|1}$
with $\ov\cM_{0,2|d-1}$,
\BE{psirest_e}\begin{split}
&D_i|_{D_i}=-\psi_i, \qquad  D_1\cdot D_2=0, \qquad
\psi_i|_{D_i},\hat\psi_d|_{D_i}=0,\\
&\psi_{3-i}|_{D_i}=\psi_{3-i}, \qquad
\hat\psi_j|_{D_i}=\hat\psi_j~~\forall~j=1,\ldots,d\!-\!1.
\end{split}\EE
If the left-hand side of~\e_ref{psiint_e} is not zero, the sum of the exponents is $d\!-\!1$.
Thus, by symmetry, we can assume that $b_d\!=\!0$.
By \e_ref{psipullback_e} and \e_ref{psirest_e},
\begin{equation*}\begin{split}
\int_{\ov\cM_{0,2|d}}\!\!\!\!\psi_1^{a_1}\psi_2^{a_2}
\hat\psi_1^{b_1}\!\ldots\!\hat\psi_d^{b_d}
=\int_{\ov\cM_{0,2|d-1}}\!\!\!\!\psi_1^{a_1-1}\psi_2^{a_2}
\hat\psi_1^{b_1}\!\ldots\!\hat\psi_d^{b_d}
+\int_{\ov\cM_{0,2|d-1}}\!\!\!\!\psi_1^{a_1}\psi_2^{a_2-1}
\hat\psi_1^{b_1}\!\ldots\!\hat\psi_d^{b_d}.
\end{split}\end{equation*}
This implies \e_ref{psiint_e} by induction on $d$
(if $d\!=\!1$, $\ov\cM_{0,2|d}$ is a single point).
\end{proof}

\noindent
Our proof of Theorems~\ref{main_thm} and~\ref{equiv_thm} immediately leads to
a closed formula for certain twisted equivariant Hurwitz numbers; see
Theorem~\ref{equiv0_thm} in Section~\ref{equivthm_sec}.
We conclude this section with a non-equivariant version of this formula.\\

\noindent
Let $x\!\in\!H^2(\P^{\i})$ denote the hyperplane class.
For any $d\!\in\!\Z^+$, let
$$ \cS^*(x)\equiv \pi_{\P^{\i}}^*\cO_{\P^{\i}}(1)\otimes\pi_{\cU}^*\cS^*
\lra \P^{\i}\!\times\!\cU\lra \P^{\i}\!\times\!\ov\cM_{0,2|d},$$
where $\pi_{\P^{\i}},\pi_{\cU}\!:\P^{\i}\!\times\!\cU\lra\P^{\i},\cU$
are the two projections.
In particular,
$$\E\big(\cS^*(x)\big)=x\!\times\!1+1\!\times\!e(\cS^*)
\in H^*(\P^{\i}\!\times\!\cU)=
\Q[x]\otimes H^*(\cU).$$
Similarly to \e_ref{Vprdfn_e}, let
\BE{V0dfn_e} \dot\V_{\a}^{(d)}(x)=\bigoplus_{a_k>0}R^0\pi_*\big(\cS^*(x)^{a_k}(-\si_1)\big)
 \oplus \bigoplus_{a_k<0}R^1\pi_*\big(\cS^*(x)^{a_k}(-\si_1)\big)
 \lra \ov\cM_{0,2|d},\EE
where $\pi\!:\cU\!\lra\!\ov\cM_{0,2|d}$ is the projection as before;
this sheaf is locally free.
We define power series   $L_{\a},\xi_{\a}\in\Q[x][[q]]$ by
\begin{alignat*}{2}
L_{\a}&\in x+q\Q[x][[q]], &\qquad L_{\a}(x,q)-q\a^{\a}L_{\a}(x,q)^{|\a|}&=x^n, \\
\xi_{\a}&\in q\Q[x][[q]],&\qquad  x+q\frac{\nd}{\nd q}\xi_{\a}(x,q)&=L_{\a}(x,q).
\end{alignat*}

\begin{thmnum}\label{main0_thm}
If $l\!\in\!\Z^{\ge0}$ and $\a\!\in\!(\Z^*)^l$, then
$$1+(\hb_1\!+\!\hb_2)\sum_{d=1}^{\i}\frac{q^d}{d!}
\int_{\ov\cM_{0,2|d}}\!\!\frac{e(\dot\V_{\a}^{(d)}(x))}{(\hb_1\!-\!\psi_1)(\hb_2\!-\!\psi_2)}
=\ne^{\frac{\xi_{\a}(x,q)}{\hb_1}+\frac{\xi_{\a}(x,q)}{\hb_2}}
\in \Q[x]\big[\big[\hb_1^{-1},\hb_2^{-1},q\big]\big].$$
\end{thmnum}

\begin{proof}
This is obtained from Theorem~\ref{equiv0_thm} by setting $n\!=\!1$, $i\!=\!1$, and
$\al_1\!=\!x$.
\end{proof}

\noindent
In the case $l\!=\!0$, the left-hand side of the expression in Theorem~\ref{main0_thm}
reduces~to
\begin{equation*}\begin{split}
&1+\sum_{\begin{subarray}{c}a_1,a_2\ge0\\  \end{subarray}} \!\!\! \big(\hb_1^{-a_1}\hb_1^{-(a_2+1)}+\hb_1^{-(a_1+1)}\hb_1^{-a_2}\big)
\frac{q^{a_1+a_2+1}}{(a_1\!+\!a_2\!+\!1)!}
\int_{\ov\cM_{0,2|a_1+a_2+1}}\psi_1^{a_1}\psi_2^{a_2}\\
&\hspace{.5in}
=1+\sum_{\begin{subarray}{c}a_1,a_2\ge0\\  \end{subarray}} \!\!\! \big(\hb_1^{-a_1}\hb_1^{-(a_2+1)}+\hb_1^{-(a_1+1)}\hb_1^{-a_2}\big)
\frac{q^{a_1+a_2+1}}{(a_1\!+\!a_2\!+\!1)!}\binom{a_1\!+\!a_2}{a_1}
=\ne^{\frac{q}{\hb_1}+\frac{q}{\hb_2}}\,;
\end{split}\end{equation*}
the first equality above holds by Lemma~\ref{M02_lmm}.
Since $\xi_{()}(x,q)\!=\!q$, this agrees with Theorem~\ref{main0_thm}.

\section{Equivariant cohomology}
\label{equivsetup_sec}

\noindent
In this section, we review the notion of equivariant cohomology and set up related notation
that will be used throughout the rest of the paper.
For the most part, our notation agrees with \cite[Chapters~29,30]{MirSym};
the main difference is that we work with $\P^{n-1}$ instead of~$\P^n$.\\

\noindent
For any $n\!\in\!\Z^+$, let
$$[n]=\{1,\ldots,n\}.$$
We denote by $\T$ the $n$-torus $(\C^*)^n$.
It acts freely on $E\T\!=\!(\C^{\i})^n\!-\!0$:
$$(t_1,\ldots,t_n)\cdot (z_1,\ldots,z_n)
=\big(t_1z_1,\ldots,t_nz_n\big).$$
Thus, the classifying space for $\T$ and its group cohomology are given by
$$\qquad B\T\equiv E\T/\T=(\P^{\i})^n  \qquad\hbox{and}\qquad
H_{\T}^*\equiv H^*(B\T;\Q)=\Q[\al_1,\ldots,\al_n],$$
where $\al_i\!=\!\pi_i^*c_1(\ga^*)$ if
$$\pi_i\!: (\P^{\i})^n\lra\P^{\i} \qquad\hbox{and}\qquad
\ga\lra\P^{\i}$$
are the projection onto the $i$-th component and the tautological line bundle, respectively.
Let
$$\H_{\T}^*=\Q_{\al}\equiv \Q(\al_1,\ldots,\al_n)$$
the field of fractions of $H_{\T}^*$.\\

\noindent
A representation $\rho$ of $\T$, i.e.~a linear action of $\T$ on $\C^k$,
induces a vector bundle over~$B\T$:
$$V_{\rho}\equiv E{\T}\times_{\T}\C^k.$$
If $\rho$ is one-dimensional, we will call
$$c_1(V_{\rho}^*)=-c_1(V_{\rho})\in H_{\T}^*\subset \Q_{\al}$$
the \sf{weight} of $\rho$.
For example, $\al_i$ is the weight of the representation
\BE{ifactorrep_e}
\pi_i\!: \T\lra\C^*, \qquad
(t_1,\ldots,t_n)\cdot z =  t_iz.\EE
More generally, if a representation $\rho$ of $\T$ on $\C^k$ splits
into one-dimensional representations with weights $\be_1,\ldots,\be_k$,
we will call $\be_1,\ldots,\be_k$ the \sf{weights} of~$\rho$.
In such a case,
\begin{equation}\label{weightspord_e}
e(V_{\rho}^*)=\be_1\cdot\ldots\cdot\be_k.
\end{equation}
We will call the representation $\rho$ of $\T$ on $\C^n$ with weights $\al_1,\ldots,\al_n$
the \sf{standard representation} of~$\T$.\\

\noindent
If $\T$ acts on a topological space $M$, let
$$H_{\T}^*(M)\equiv H^*(BM;\Q), \qquad\hbox{where}\qquad BM=E\T\!\times_{\T}\!M,$$
denote the corresponding \sf{equivariant cohomology} of $M$.
The projection map $BM\!\lra\!B\T$ induces an action of $H_{\T}^*$ on $H_{\T}^*(M)$.
Let
$$\H_{\T}^*(M)=H_{\T}^*(M)\otimes_{H_{\T}^*}\H_{\T}^*.$$
If the $\T$-action on $M$ lifts to an action on a (complex) vector bundle $V\!\lra\!M$,
then
$$BV\equiv E{\T}\!\times_{\T}\!V$$
is a vector bundle over $BM$.
Let
$$\E(V)\equiv e(BV)\in H_{\T}^*(M)\subset \H_{\T}^*(M)$$
denote the \sf{equivariant Euler class of} $V$.\\

\noindent
Throughout the paper, we work with the standard action of $\T$ on $\P^{n-1}$,
i.e.~the action induced by the standard action $\rho$ of $\T$ on~$\C^n$:
$$(t_1,\ldots,t_n)\cdot [z_1,\ldots,z_n]
=\big[t_1z_1,\ldots,t_nz_n\big].$$
Since $B\P^{n-1}=\P V_{\rho}$,
$$H_{\T}^*(\P^{n-1})\equiv H^*\big(\P V_{\rho};\Q\big)
= \Q[\x,\al_1,\ldots,\al_n]
\big/\big(\x^n\!+\!c_1(V_{\rho})\x^{n-1}\!+\!\ldots\!+\!c_n(V_{\rho})\big),$$
where $\x\!=\!c_1(\ti\ga^*)$ and $\ti\ga\!\lra\!\P V_{\rho}$
is the tautological line bundle.
Since
$$c(V_{\rho})=(1-\al_1)\ldots(1-\al_n),$$
it follows that
\begin{equation}\label{pncoh_e}\begin{split}
H_{\T}^*(\P^{n-1}) &= \Q[\x,\al_1,\ldots,\al_n]\big/(\x\!-\!\al_1)\ldots(\x\!-\!\al_n),\\
\H_{\T}^*(\P^{n-1}) &= \Q_{\al}[\x]\big/(\x\!-\!\al_1)\ldots(\x\!-\!\al_n).
\end{split}\end{equation}\\

\noindent
The standard action of $\T$ on $\P^{n-1}$ has $n$~fixed points:
$$P_1=[1,0,\ldots,0], \qquad P_2=[0,1,0,\ldots,0],
\quad\ldots\quad P_n=[0,\ldots,0,1].$$
For each $i\!=\!1,2,\ldots,n$, let
\BE{phidfn_e} \phi_i= \prod_{k\neq i}(\x\!-\!\al_k) \in H_{\T}^*(\P^{n-1}).\EE
By equation~\e_ref{phiprop_e} below, $\phi_i$ is the equivariant Poincare dual of $P_i$.
We also note that $\ti\ga|_{BP_i}\!=\!V_{\pi_i}$, where $\pi_i$ is as in~\e_ref{ifactorrep_e}.
Thus, the restriction map on the equivariant cohomology induced by the inclusion
$P_i\!\lra\!\P^{n-1}$ is given~by
$$H_{\T}^*(\P^{n-1})=\Q[\x,\al_1,\ldots,\al_n]\big/\prod_{k=1}^{k=n}(\x\!-\!\al_k)
\lra H_{\T}^*(P_i)=\Q[\al_1,\ldots,\al_n], \qquad \x\lra\al_i.$$
In particular, if $\F\!\in\!H_{\T}^*(\Pn)$,  then
\BE{uniquecond_e}
\F=0\in H_{\T}^*(\Pn)  \quad\Llra\quad
\F(\x\!=\!\al_i)\equiv \F|_{P_i}=0
\in \Q[\al_1,\ldots,\al_n]\subset\Q_{\al} ~\forall~i\in[n].\EE\\

\noindent
The tautological line bundle $\ga_{n-1}\!\lra\!\P^{n-1}$ is a subbundle of
$\P^{n-1}\!\times\!\C^n$ preserved by the diagonal action of~$\T$.
Thus, the action of $\T$ on $\P^{n-1}$ naturally lifts to an action
on~$\ga_{n-1}$ and
\begin{equation}\label{garestr_e}
\E\big(\ga_{n-1}^*\big)\big|_{P_i}=\al_i \qquad\forall~i=1,2,\ldots,n.
\end{equation}
The $\T$-action on $\P^{n-1}$ also has a natural lift to the vector bundle
$T\P^{n-1}\!\lra\!\P^{n-1}$ so that  there is a short exact sequence
$$0\lra \ga_{n-1}^*\otimes\ga_{n-1}
\lra \ga_{n-1}^*\otimes \C^n
\lra T\P^{n-1} \lra0$$
of $\T$-equivariant vector bundles on $\P^{n-1}$.
By~\e_ref{weightspord_e}, \e_ref{garestr_e}, and~\e_ref{phidfn_e},
\begin{equation}\label{tangrestr_e}
\E\big(T\P^{n-1}\big)\big|_{P_i}=\prod_{k\neq i}(\al_i\!-\!\al_k)
=\phi_i|_{P_i}
\qquad\forall~i=1,2,\ldots,n.
\end{equation}\\

\noindent
If $\T$ acts smoothly on a smooth compact oriented manifold $M$, there is a well-defined
integration-along-the-fiber homomorphism
$$\int_M\!: H_{\T}^*(M)\lra H_{\T}^*$$
for the fiber bundle $BM\!\lra\!B\T$.
The classical localization theorem of~\cite{ABo} relates it to
integration along the fixed locus of the $\T$-action.
The latter is a union of smooth compact orientable manifolds~$F$;
$\T$ acts on the normal bundle $\N F$ of each~$F$.
Once an orientation of $F$ is chosen, there is a well-defined
integration-along-the-fiber homomorphism
$$\int_F\!: H_{\T}^*(F)\lra H_{\T}^*.$$
The localization theorem states that
\begin{equation}\label{ABothm_e}
\int_M\eta = \sum_F\int_F\frac{\eta|_F}{\E(\N F)} \in \Q_{\al}
\qquad\forall~\eta\in H_{\T}^*(M),
\end{equation}
where the sum is taken over all components $F$ of the fixed locus of $\T$.
Part of the statement of~\e_ref{ABothm_e} is that $\E(\N F)$
is invertible in~$\H_{\T}^*(F)$.
In the case of the standard action of $\T$ on~$\P^{n-1}$, \e_ref{ABothm_e} implies that
\begin{equation}\label{phiprop_e}
\eta|_{P_i}=\int_{\P^{n-1}}\eta\phi_i \in\Q_{\al}
\qquad\qquad\forall~\eta\!\in\!\H_{\T}^*(\P^{n-1}),~i=1,2,\ldots,n;
\end{equation}
see also \e_ref{tangrestr_e}.\\

\noindent
Finally, if $f\!:M\!\lra\!M'$ is a $\T$-equivariant map between two compact oriented
manifolds, there is a well-defined pushforward homomorphism
$$f_*\!: H_{\T}^*(M) \lra H_{\T}^*(M').$$
It is characterized by the property that
\BE{pushdfn_e}
\int_{M'}(f_*\eta)\,\eta'=\int_M\eta\,(f^*\eta')
\qquad~\forall~\eta\!\in\! H_{\T}^*(M),\, \eta'\!\in\! H_{\T}^*(M').\EE
The homomorphism $\int_M$ of the previous paragraph corresponds to $M'$ being a point.
It is immediate from~\e_ref{pushdfn_e} that
\BE{pushprop_e}
f_*\big(\eta\,(f^*\eta')\big)=(f_*\eta)\,\eta'
\qquad~\forall~\eta\!\in\! H_{\T}^*(M),\, \eta'\!\in\! H_{\T}^*(M').\EE

\section{Equivariant mirror theorem}
\label{equivthm_sec}

\noindent
In this section, we state an equivariant version of Theorem~\ref{main_thm},
Theorem~\ref{equiv_thm}, which immediately implies Theorems~\ref{main_thm}.
It is proved in the rest of this paper,
as outlined in Section~\ref{intro_sec} after the statement of Theorem~\ref{main_thm}.
We then formulate an equivariant version of Theorem~\ref{main0_thm},
Theorem~\ref{equiv0_thm},
providing a closed formula for equivariant Hurwitz numbers.
This theorem immediately implies Theorem~\ref{main0_thm}
and is obtained in Section~\ref{mainpf_sec}
by combining Proposition~\ref{cYcZrel_prp} in this paper with some results
from~\cite{g0ci}.
Throughout the paper, we use calligraphic letters, e.g.~$\cY$ and~$\cZ$,
for equivariant generating functions.\\

\noindent
The standard $\T$-representation on $\C^n$ (as well as any other representation) induces
a $\T$-action on the trivial rank~$n$ sheaf over any quasi-stable curve $(\cC,y_1,\ldots,y_m)$,
$$\T\cdot \C^n\!\otimes\!\cO_{\cC}\lra \C^n\!\otimes\!\cO_{\cC}, \qquad
(t_1,\ldots,t_n)\cdot(f_1,\ldots,f_n)=(t_1f_1,\ldots,t_nf_n),$$
and thus on the rank~1 subsheaves of this sheaf.
This action preserves the degree of the subsheaf and
the torsion and stability properties of Section~\ref{SQdfn_sec} and
thus induces a $\T$-action on the moduli space $\ov{Q}_{g,m}(\Pn,d)$,
with respect to which the evaluation maps
$$\ev_i\!:\ov{Q}_{g,m}(\Pn,d)\lra\Pn, \qquad i=1,2,\ldots,m,$$
are $\T$-equivariant.
This action lifts to a $\T$-action on the universal subsheaf $\cS\!\lra\!\cU$
and thus to~$\T$-actions on the locally free sheaves
$$\pi_*(\si_i^2)\lra\ov{Q}_{g,m}(\Pn,d),\qquad
\dot\V_{n;\a}^{(d)}\lra\ov{Q}_{0,2}(\Pn,d).$$
This gives rise to $\T$-equivariant cohomology classes,
$$\psi_i\in H_{\T}^*\big(\ov{Q}_{g,m}(\Pn,d)\big), \qquad
\E(\dot\V_{n;\a}^{(d)})\in H_{\T}^*\big(\ov{Q}_{0,2}(\Pn,d)\big).$$
The stable quotients analogue of the equivariant version of Givental's $J$-function
is given~by
\BE{cZdfn_e}
\cZ_{n;\a}(\x,\hb,q) \equiv
1+\sum_{d=1}^{\i}\!q^d\ev_{1*}\left[\frac{\E(\dot\V_{n;\a}^{(d)})}{\hb\!-\!\psi_1}\right]
\in H_{\T}^*(\Pn)\big[\big[\hb^{-1},q\big]\big],\EE
where $\ev_1:\ov{Q}_{0,2}(\Pn,d)\lra\Pn$ is as before.
The equivariant analogue of the power series~\e_ref{Ydfn_e} is given~by
\BE{cYdfn_e}
\cY_{n;\a}(\x,\hb,q)\equiv\sum_{d=0}^{\i}q^d
\frac{\prod\limits_{a_k>0}\prod\limits_{r=1}^{a_kd}\!(a_k\x\!+\!r\hb)
\prod\limits_{a_k<0}\!\!\prod\limits_{r=0}^{-a_kd-1}\!\!(a_k\x\!-\!r\hb)}
{\prod\limits_{r=1}^d\prod\limits_{k=1}^n(\x\!-\!\al_k\!+\!r\hb)}
\in \Q[\al_1,\ldots,\al_n,\x]\big[\big[\hb^{-1},q\big]\big]\,.\EE
We view~\e_ref{cZdfn_e} and~\e_ref{cYdfn_e} as power series in~$\hb^{-1}$ and~$q$,
by expanding around~$\hb\!=\!\i$ and \hbox{$q\!=\!0$}.
The coefficients of powers of $\hb^{-1}$ and~$q$ in~\e_ref{cYdfn_e}
are polynomials in $\al_1,\ldots,\al_n$ and~$\x$;
the coefficients in~\e_ref{cYdfn_e} are $\T$-equivariant cohomology classes
on~$\Pn$, which can also be represented by polynomials.

\begin{thmnum}\label{equiv_thm}
If $l\!\in\!\Z^{\ge0}$, $n\!\in\!\Z^+$, and $\a\!\in\!(\Z^*)^l$ are such that
$|\a|\!\le\!n$, then the equivariant stable quotients analogue of Givental's $J$-function
satisfies
\BE{equivMS_e}\cZ_{n;\a}(\x,\hb,q)=\frac{\cY_{n;\a}(\x,\hb,q)}{I_{n;\a}(q)}
\in H_{\T}^*(\Pn)\big[\big[\hb^{-1},q\big]\big].\EE
\end{thmnum}

\noindent
Restricting to a fiber of the projection
$$B\P^{n-1}\equiv E\T\!\times_{\T}\!\P^{n-1}\lra \P^{n-1},$$
we send $\x$ to $x$ and $\al_i$ to 0; this gives Theorem~\ref{main_thm}.
The relation of Theorem~\ref{equiv_thm} to its Gromov-Witten analogue
is the same as the relation of Theorem~\ref{main_thm} to its Gromov-Witten analogue;
see the paragraph following the statement of Theorem~\ref{main_thm} in Section~\ref{intro_sec}.
In particular, the twisted equivariant stable quotients invariants of~$\P^{n-1}$
determined by a tuple~$\a$ are the same as the corresponding Gromov-Witten invariants
if $|\a|\!-\!\ell^-(\a)\!\le\!n\!-\!2$, but not if $|\a|\!-\!\ell^-(\a)\!=\!n\!-\!1,n$.\\

\noindent
We prove Theorem~\ref{equiv_thm} through a two-pronged approach.
We show that the power series $\cY_{n;\a}$ and $\cZ_{n;\a}$ are
$\fC$-recursive in the sense of Definition~\ref{recur_dfn}
with the collection~$\fC$ given by~\e_ref{Cdfn_e} and
satisfy the self-polynomiality condition of Definition~\ref{SPC_dfn};
see Lemma~\ref{cYrec_lmm} and Propositions~\ref{cZrec_prp} and~\ref{cZSPC_prp}.
Proposition~\ref{uniqueness_prp} then implies that
these power series are determined by their $\mod\hb^{-2}$-parts,
i.e.~the coefficients of~$\hb^0$ and~$\hb^{-1}$ in this case.
It is straightforward to determine the $\mod\hb^{-2}$-part of $\cY_{n;\a}$
in all cases ($\cY_{n;\a}$ is given by an explicit algebraic expression)
and the $\mod\hb^{-2}$-part of $\cZ_{n;\a}$ if $|\a|\!\le\!n\!-\!2$,
thus establishing Theorem~\ref{equiv_thm} whenever $|\a|\!\le\!n\!-\!2$;
see Corollary~\ref{MirSym_crl}.\\

\noindent
In order to establish Theorem~\ref{equiv_thm} in all cases, we show that
the secondary coefficients $\cY_i^r(d)$ and $\cZ_i^r(d)$,
instead of~$\F_i^r(d)$, in the recursions~\e_ref{recurdfn_e2}
for $\cY_{n;\a}$ and~$\cZ_{n;\a}$ are the same.
By induction on~$d$, this implies that the coefficients of~$q^d$
on the two sides of~\e_ref{equivMS_e} are the same because this is the case for $d\!=\!0$
(when both coefficients are~1).
As part of the proof of $\fC$-recursivity for~$\cY_{n;\a}$, we show that
$\cY_i^r(d)$ is determined by the expansion of $\cY_{n;\a}(\al_i,\hb,q)$ around
$h\!=\!0$; see Lemma~\ref{cYrec_lmm}.
As part of the proof of $\fC$-recursivity for~$\cZ_{n;\a}$, we show that
$\cZ_i^r(d)$ is also determined by the expansion of $\cZ_{n;\a}(\al_i,\hb,q)$ around
$h\!=\!0$; see Proposition~\ref{cZrec_prp}.
In contrast to~$\cY_i^r(d)$, $\cZ_i^r(d)$ is determined by lower-degree coefficients
of~$\cZ_{n;\a}$ or equivalently by~$\cZ_j^s(d')$ with $d'\!<\!d$;
this relation thus completely determines~$\cZ_{n;\a}$ (assuming $\fC$-recursivity).
It follows that~\e_ref{equivMS_e} holds if and only~if the coefficients
$\cY_i^r(d)$ for $\cY_{n;\a}$ satisfy the same relation; see Lemma~\ref{cYcZcomp_lmm}.
The coefficients in this relation involve twisted Hurwitz numbers over
the moduli spaces~$\ov\cM_{0,2|d}$.
These are not easy to compute, but
they can be described qualitatively in way independent of~$n$.
This implies that the validity of the desired recursion
for the secondary coefficients~$\cY_i^r(d)$ for~$\cY_{n;\a}$ is {\it independent} of~$n$.
Since this recursion is equivalent to~\e_ref{equivMS_e} whenever $|\a|\!\le\!n$
and \e_ref{equivMS_e} holds whenever $|\a|\!\le\!n\!-\!2$ (by Corollary~\ref{MirSym_crl}),
it follows that the recursion holds in all cases (see Proposition~\ref{cYcZrel_prp})
and \e_ref{equivMS_e} holds whenever $|\a|\!\le\!n$, as claimed.\\

\noindent
As stated in Section~\ref{intro_sec},
Theorem~\ref{equiv_thm} extends to products of projective spaces and
concavex sheaves~\e_ref{gensheaf_e}.
The relevant torus action is then the product of the actions on the components
described in Section~\ref{equivsetup_sec}.
If its weights are denoted by $\al_{s;k}$, with $s\!=\!1,\ldots,p$ and $k\!=\!1,\ldots,n_s$,
then
\begin{alignat}{1}
\label{gencY_e}
\cY_{n_1,\ldots,n_p;\a}(\x_1,\ldots,\x_p,\hb,q_1,\ldots,q_p)
&\in \Q[\al_{1;1},\ldots,\al_{p;n_p},\x_1,\ldots,\x_p]\big[\big[\hb^{-1},q_1,\ldots,q_p\big]\big],\\
\label{gencZ_e}
\cZ_{n_1,\ldots,n_p;\a}(\x_1,\ldots,\x_p,\hb,q_1,\ldots,q_p)&\in
H_{\T}^*\big(\P^{n_1-1}\!\times\!\ldots\!\times\!\P^{n_p-1}\big)
\big[\big[\hb^{-1},q_1,\ldots,q_p\big]\big],
\end{alignat}
and $\x_1,\ldots,\x_p\!\in\!H^*(\P^{n_1-1}\!\times\!\ldots\!\times\!\P^{n_p-1})$
correspond to the pullbacks of the equivariant hyperplane classes by the projection maps.
The coefficient of $q_1^{d_1}\!\ldots\!q_p^{d_p}$ in~\e_ref{gencZ_e}
is defined by the same pushforward as in~\e_ref{cZdfn_e}, with the degree~$d$
of the stable quotients replaced by~$(d_1,\ldots,d_p)$.
The coefficient of $q_1^{d_1}\!\ldots\!q_p^{d_p}$ in~\e_ref{gencY_e}
is given~by
$$\frac{\prod\limits_{a_{k;1}\ge0}
\prod\limits_{r=1}^{\sum\limits_{s=1}^p\!a_{k;s}d_s}
\hspace{-.2in}\big(\sum\limits_{s=1}^p\!a_{k;s}\x_s\!+\!r\hb\big)
\prod\limits_{a_{k;1}<0}\!\!
\prod\limits_{r=0}^{-\sum\limits_{s=1}^p\!a_{k;s}d_s\,-1}
\hspace{-.3in}\big(\sum\limits_{s=1}^p\!a_{k;s}\x_s\!-\!r\hb\big)}
{\prod\limits_{s=1}^p\prod\limits_{r=1}^{d_s}\prod\limits_{k=1}^{n_s}(\x_s\!-\!\al_{s;k}+\!r\hb)}\,.$$
Our proof of Theorem~\ref{equiv_thm} extends directly to this situation.\\

\noindent
We conclude this section with an equivariant version of Theorem~\ref{main0_thm}.
For any $d\!\in\!\Z^+$ and $\be\!\in\!H_{\T}^2$, denote by
\BE{cSbe_e} \cS^*(\be)\lra \cU\lra \ov\cM_{0,2|d}\EE
the universal sheaf with the $\T$-action so that
$$\E\big(\cS^*(\be)\big)=\be\!\times\!1+1\!\times\!e(\cS^*)
\in H_{\T}^*(\cU)=H_{\T}^*\otimes H^*(\cU).$$
Similarly to \e_ref{Vprdfn_e}, let\BE{V0prdfn_e}
\dot\V_{\a}^{(d)}(\be)=\bigoplus_{a_k>0}R^0\pi_*\big(\cS^*(\be)^{a_k}(-\si_1)\big)
 \oplus \bigoplus_{a_k<0}R^1\pi_*\big(\cS^*(\be)^{a_k}(-\si_1)\big)
 \lra \ov\cM_{0,2|d},\EE
where $\pi\!:\cU\!\lra\!\ov\cM_{0,2|d}$ is the projection as before;
this sheaf is locally free.
The bundle
\BE{V1dfn_e}\dot\V_{1}^{(d)}(\be)
\equiv  \dot\V_{(1)}^{(d)}(\be)
=R^0\pi_*\big(\cS^*(\be)(-\si_1)\big) \lra \ov\cM_{0,2|d}\EE
plays a central role  in the deformation theory of stable quotients
as explained in Section~\ref{SQlocal_sec}.
We define power series   $L_{n;\a},\xi_{n;\a}\in\Q_{\al}[\x][[q]]$ by
\begin{alignat*}{2}
L_{n;\a}&\in \x+q\Q_{\al}[\x][[q]], &\qquad
\prod_{k=1}^n\!\!\big(L_{n;\a}(\x,q)\!-\!\al_k\big)-q\a^{\a}L_{n;\a}(\x,q)^{|\a|}&=
\prod_{k=1}^n\!(\x\!-\!\al_k), \\
\xi_{n;\a}&\in q\Q_{\al}[\x][[q]],&\qquad
\x+q\frac{\nd}{\nd q}\xi_{n;\a}(\x,q)&=L_{n;\a}(\x,q).
\end{alignat*}

\begin{thmnum}\label{equiv0_thm}
If $l\!\in\!\Z^{\ge0}$, $n\!\in\!\Z^+$, and $\a\!\in\!(\Z^*)^l$, then
\begin{equation*}\begin{split}
&1+(\hb_1\!+\!\hb_2)\sum_{d=1}^{\i}\frac{q^d}{d!}
\int_{\ov\cM_{0,2|d}}\frac{\E(\dot\V_{\a}^{(d)}(\al_i))}
{\prod\limits_{k\neq i}\!\!\E(\dot\V_1^{(d)}(\al_i\!-\!\al_k))\,(\hb_1\!-\!\psi_1)(\hb_2\!-\!\psi_2)}\\
&\hspace{3in}
=\ne^{\frac{\xi_{n;\a}(\al_i,q)}{\hb_1}+\frac{\xi_{n;\a}(\al_i,q)}{\hb_2}}
\in \Q_{\al}\big[\big[\hb_1^{-1},\hb_2^{-1},q\big]\big]
\end{split}\end{equation*}
for every $i\!=\!1,\ldots,n$.
\end{thmnum}

\section{Algebraic observations}
\label{algebra_sec}

\noindent
In this section, we describe a number of properties of power series,
such as $\cY_{n;\a}$ in~\e_ref{cYdfn_e} and $\cZ_{n;\a}$ in~\e_ref{cZdfn_e},
that determine them completely.
We also show that $\cY_{n;\a}$ indeed satisfies these properties.\\

\noindent
If $R$ is a ring, denote by
$$R\Lau{\hb}\equiv R[[\hb^{-1}]]+R[\hb]$$
the $R$-algebra of Laurent series in $\hb^{-1}$ (with finite principal part).
If $f\!\in\!R[[q]]$ and $d\!\in\!\Z^{\ge0}$,
let $\coeff{f}_{q;d}\!\in\!R$ denote the coefficient of~$q^d$ in~$f$.
If $p\!\in\!\Z^{\ge0}$ and
$$\F(\hb,q)=\sum_{d=0}^{\i}
\bigg(\sum_{r=-N_d}^{\i}\!\!\!\F_d^{(r)}\hb^{-r}\bigg)q^d
\in R\Lau{\hb}\big[\big[q\big]\big]$$
for some $\F_d^{(r)}\!\in\!R$, we define
$$\F(\hb,q) \cong \sum_{d=0}^{\i}
\bigg(\sum_{r=-N_d}^{p-1}\!\!\!\F_d^{(r)}\hb^{-r}\bigg)q^d\quad(\mod \hb^{-p}),$$
i.e.~we drop $\hb^{-p}$ and higher powers of $\hb^{-1}$,
instead of higher powers of~$\hb$.
If $R$ is a field, let
$$R(\hb) \lhook\joinrel\lra R\Lau{\hb}$$
be the embedding given by taking the Laurent series of rational functions at $\hb^{-1}\!=\!0$.\\

\noindent
If $f\!=\!f(z)$ is a rational function in $z$ and possibly some other
variables, for any $z_0\!\in\!\P^1\!\supset\!\C$ let
\BE{Rsdfn_e}\Rs{z=z_0}f(z) \equiv \frac{1}{2\pi\I}\oint f(z)\nd z,\EE
where the integral is taken over a positively oriented loop around $z\!=\!z_0$
with no other singular points of $f\nd z$,
denote the residue of the 1-form~$f\nd z$.
If $z_1,\ldots,z_k\!\in\!\P^1$ is any collection of points, let
\BE{Rssumdfn_e}\Rs{z=z_1,\ldots,z_k}f(z)
\equiv\sum_{i=1}^{i=k}\Rs{z=z_i}f(z).\EE
By the Residue Theorem on $S^2$,
$$\sum_{\x_0\in S^2}\Rs{\x=\x_0}\big\{f(\x)\big\}=0$$
for every  rational function $f\!=\!f(\x)$ on $S^2\!\supset\!\C$.
If $f$ is regular at $z\!=\!0$, let $\left\llbracket f\right\rrbracket_{z;p}$ denote
the coefficient of $z^p$ in the power series expansion of $f$ around $z\!=\!0$.

\begin{dfn}\label{recur_dfn}
Let $C\equiv(C_i^j(d))_{d,i,j\in\Z^{+}}$ be any collection of elements of~$\Q_{\al}$.
A~power series $\F\!\in\!H_{\T}^*(\Pn)\Lau{\hb}[[q]]$ is \sf{$C$-recursive} if
the following holds:
if $d^*\!\in\!\Z^{\ge0}$ is such~that
$$\bigcoeff{\F(\x\!=\!\al_i,\hb,q)}_{q;d^*-d}\in \Q_{\al}(\hb)
\subset \Q_{\al}\Lau{\hb}
\qquad\forall\,d\!\in\![d^*],\,i\!\in\![n],$$
and $\bigcoeff{\F(\al_i,\hb,q)}_{q;d}$ is regular at
$\hb\!=\!(\al_i\!-\!\al_j)/d$ for all $d\!<\!d^*$ and $i\!\neq\!j$, then
\BE{recurdfn_e}
\bigcoeff{\F(\al_i,\hb,q)}_{q;d^*}-
\sum_{d=1}^{d^*}\sum_{j\neq i}\frac{C_i^j(d)}{\hb-\frac{\al_j-\al_i}{d}}
\bigcoeff{\F(\al_j,z,q)}_{q;d^*-d}\big|_{z=\frac{\al_j-\al_i}{d}}
\in \Q_{\al}\big[\hb,\hb^{-1}\big]\subset \Q_{\al}\Lau{\hb}.\EE
\end{dfn}

\noindent
Thus, if $\F\!\in\!H_{\T}^*(\Pn)\Lau{\hb}[[q]]$ is $C$-recursive, for any collection~$C$, then
$$\F(\x\!=\!\al_i,\hb,q)\in \Q_{\al}(\hb)\big[\big[q\big]\big]
\subset \Q_{\al}\Lau{\hb}[[q]]
\qquad\forall\,i\!\in\![n],$$
as can be seen by induction on $d$, and
\BE{recurdfn_e2} \F(\al_i,\hb,q)=\sum_{d=0}^{\i}\sum_{r=-N_d}^{N_d}
\F_i^r(d)\hb^rq^d+
\sum_{d=1}^{\i}\sum_{j\neq i}\frac{C_i^j(d)q^d}{\hb-\frac{\al_j-\al_i}{d}}
\F(\al_j,(\al_j\!-\!\al_i)/d,q) \qquad\forall~i\!\in\![n],\EE
for some $\F_i^r(d)\!\in\!\Q_{\al}$.
The nominal issue with defining $C$-recursivity by~\e_ref{recurdfn_e2}, as is normally done,
is that a priori the evaluation
of $\F(\al_j,\hb,q)$ at $\hb\!=\!(\al_j\!-\!\al_i)/d$ need not be well-defined,
since $\F(\al_j,\hb,q)$ is a power series in~$q$ with coefficients in
the Laurent series in~$\hb^{-1}$;
a priori they  may not converge anywhere.
However, taking the coefficient of each power of~$q$ in \e_ref{recurdfn_e2}
shows by induction on the degree~$d$
that this evaluation does make sense;
this is the substance of Definition~\ref{recur_dfn}.

\begin{dfn}\label{SPC_dfn}
For any $\F\!\equiv\!\F(\x,\hb,q)\in H_{\T}^*(\Pn)\Lau{\hb}[[q]]$, let
\BE{PhiZdfn_e}
\Phi_{\F}(\hb,z,q)\equiv \sum_{i=1}^n
\frac{\lr{\a}\al_i^{\ell(\a)}\ne^{\al_iz}}{\prod\limits_{k\neq i}(\al_i\!-\!\al_k)}
\F\big(\al_i,\hb,q\ne^{\hb z}\big)\F(\al_i,-\hb,q)
\in \Q_{\al}\Lau{\hb}[[z,q]].\EE
A power series $\F\!\in\!H_{\T}^*(\Pn)\Lau{\hb}[[z,q]]$
\sf{satisfies the self-polynomiality condition} if
$\Phi_{\F}\in \Q_{\al}[\hb][[z,q]]$.
\end{dfn}

\begin{prp}[{\cite[Lemma 30.3.2]{MirSym}}]\label{uniqueness_prp}
Let $\F,\F'\in H_{\T}^*(\Pn)\Lau{\hb}[[q]]$.
If $\F$ and $\F'$ are $C$-recursive, for some collection
$C\equiv(C_i^j(d))_{d,i,j\in\Z^{+}}$ of elements of~$\Q_{\al}$,
satisfy the self-polynomiality condition, and
$$\F(\x\!=\!\al_i,\hb,q),\F'(\x\!=\!\al_i,\hb,q)
\in \Q_{\al}^*+q\cdot \Q_{\al}\Lau{\hb}\big[\big[q\big]\big]
\subset \Q_{\al}\Lau{\hb}\big[\big[q\big]\big]
 \qquad\forall\,i\!\in\![n],$$
then  $\F\cong\F'~(\mod \hb^{-2})$ if and only if $\F=\F'$.
\end{prp}

\noindent
Let
\BE{Cdfn_e}
\fC_i^j(d)\equiv  \frac{
\prod\limits_{a_k>0}\prod\limits_{r=1}^{a_kd}\!\!\left(a_k\al_i+r\,\frac{\al_j-\al_i}{d}\right)
\prod\limits_{a_k<0}\!\!\prod\limits_{r=0}^{-a_kd-1}\!\!\left(a_k\al_i-r\,\frac{\al_j-\al_i}{d}\right)}
{d\underset{(r,k)\neq(d,j)}{\prod\limits_{r=1}^d\prod\limits_{k=1}^n}
            \left(\al_i-\al_k+r\,\frac{\al_j-\al_i}{d}\right)}\in\Q_{\al} \,.\EE

\begin{lmm}\label{cYrec_lmm}
If $l\!\in\!\Z^{\ge0}$, $n\!\in\!\Z^+$, and $\a\!\in\!(\Z^*)^l$
are such that $|\a|\!\le\!n$, the power series $\cY_{n;\a}(\x,\hb,q)$
given by~\e_ref{cYdfn_e}
is $\fC$-recursive, with the auxiliary coefficients
in the recursion~\e_ref{recurdfn_e2} for~$\cY_{n;\a}$ given~by
\BE{cYird_e}
\sum_{d=0}^{\i}\cY_i^r(d)q^d=
\begin{cases}
\Rs{h=0}\big\{\hb^{-r-1}\cY_{n;\a}(\al_i,\hb,q)\big\},& \hbox{if}~r\!<\!0;\\
I_{n;\a}(q),& \hbox{if}~r\!=\!0;\\
0,&\hbox{if}~r\!>\!0.
\end{cases}\EE
Furthermore, $\cY_{n;\a}(\x,\hb,q)$ satisfies the self-polynomiality condition.
\end{lmm}

\begin{proof}
This is well-known from the various proofs of mirror symmetry for Gromov-Witten invariants
(e.g.~\cite[Section~11]{Gi2}, \cite[Chapter~30]{MirSym}, \cite[Section~4]{Elezi});
we include a proof for the sake of completeness.\\

\noindent
We first view $\cY_{n;\a}$ as an element of $\Q_{\al}(\x,\hb)[[q]]$.
Splitting the coefficient of $q^{d+d'}$ in~\e_ref{cYdfn_e} into the factors
with $r\!\le\!d$ and $r\!>\!d$, plugging in $(\al_j\!-\!\al_i)/d$ into
all factors other than the $(r,k)\!=\!(d,j)$ factor in the denominator,
and simplifying, we obtain
$$\Rs{z=\frac{\al_j-\al_i}{d}}\bigg\{\frac{1}{\hb\!-\!z}\cY_{n;\a}(\al_i,z,q)\bigg\}
=\frac{\fC_i^j(d)q^d}{\hb-\frac{\al_j-\al_i}{d}}
\, \cY_{n;\a}\big(\al_j,(\al_j\!-\!\al_i)/d,q\big)\,.$$
By the Residue Theorem on~$S^2$,
\begin{equation}\label{recurpf_e1}\begin{split}
\sum_{d=1}^{\i}\sum_{j\neq i}\frac{\fC_i^j(d)q^d}{\hb-\frac{\al_j-\al_i}{d}}
&\cY_{n;\a}\big(\al_j,(\al_j\!-\!\al_i)/d,q\big)
=-\Rs{z=\hb,0,\i}\bigg\{\frac{1}{\hb\!-\!z}\cY_{n;\a}(\al_i,z,q)\bigg\}\\
&\qquad\qquad\qquad\qquad
=\cY_{n;\a}(\al_i,\hb,q)-\Rs{z=0,\i}\bigg\{\frac{1}{\hb\!-\!z}\cY_{n;\a}(\al_i,z,q)\bigg\}.
\end{split}\end{equation}
Since the coefficients of $(\hb^{-1})^0$ in $\cY_{n;\a}(\al_i,\hb,q)$
and $Y_{n;\a}(\al_i,\hb,q)$ are the same,
$$\Rs{z=\i}\bigg\{\frac{1}{\hb\!-\!z}\cY_{n;\a}(\al_i,z,q)\bigg\}
=I_{n;\a}(q)$$
by~\e_ref{Idfn_e}.
Since the coefficient of $q^d$ in $\cY_{n;\a}(\al_i,\hb,q)$ has a pole of order~$d$
at $\hb\!=\!0$,
\begin{alignat*}{1}
\Rs{z=0}\bigg\{\frac{1}{\hb\!-\!z}\coeff{\cY_{n;\a}(\al_i,z,q)}_{q;d}\bigg\}
&=\left\llbracket\frac{1}{\hb\!-\!z}
\frac{\prod\limits_{a_k>0}\prod\limits_{r=1}^{a_kd}\!\!(a_k\al_i\!+\!rz)
\prod\limits_{a_k<0}\!\!\prod\limits_{r=0}^{-a_kd-1}\!\!(a_k\al_i\!-\!rz)}
{d!\prod\limits_{r=1}^{d}\prod\limits_{k\neq i}(\al_i\!-\!\al_k\!+\!rz)}
\right\rrbracket_{z;d-1}\,;
\end{alignat*}
the last expression is a polynomial in $\hb^{-1}$ with coefficients in $\Q_{\al}$
of degree at most~$d$.
This establishes that $\cY_{n;\a}$ is $\fC$-recursive and \e_ref{cYird_e} holds;
the $r\!<\!0$ case of~\e_ref{cYird_e} follows from~\e_ref{recurdfn_e2}
with $\F$ replaced by~$\cY_{n;\a}$.\\

\noindent
We now expand $\cY_{n;\a}$ as a power series in $\hb^{-1}$ and $q$
with coefficients in~$\Q_{\al}[\x]$.
Thus,
$$\frac{\lr{\a}\al_i^{\ell(\a)}\ne^{\x z}}{\prod\limits_{k=1}^n(\x\!-\!\al_k)}
\cY_{n;\a}\big(\x,\hb,q\ne^{\hb z}\big)\cY_{n;\a}(-\x,\hb,q)
\in \Q_{\al}(\x)[[\hb^{-1},z,q]]$$
viewed as a function of $\x$ has residues only at $\x\!=\!\al_i$ with $i\!\in\![n]$
and $\x\!=\!\i$.
By \e_ref{cYdfn_e},
$$\frac{\lr{\a}\al_i^{\ell(\a)}\ne^{\al_iz}}{\prod\limits_{k\neq i}(\al_i\!-\!\al_k)}
\cY_{n;\a}\big(\al_i,\hb,q\ne^{\hb z}\big)\cY_{n;\a}(\al_i,-\hb,q)
=\Rs{\x=\al_i}\left\{\frac{\lr{\a}\x^{\ell(\a)}\ne^{\x z}}{\prod\limits_{k=1}^n(\x\!-\!\al_k)}
\cY_{n;\a}\big(\x,\hb,q\ne^{\hb z}\big)\cY_{n;\a}(\x,-\hb,q)\right\}.$$
Thus, by the Residue Theorem on~$S^2$,
$$\Phi_{\cY_{n;\a}}(\hb,z,q) =-
\Rs{\x=0,\i}\left\{\frac{\lr{\a}\x^{\ell(\a)}\ne^{\x z}}{\prod\limits_{k=1}^n(\x\!-\!\al_k)}
\cY_{n;\a}\big(\x,\hb,q\ne^{\hb z}\big)\cY_{n;\a}(\x,-\hb,q)\right\}
\equiv -\fR_0-\fR_{\i}\,.$$
Since the coefficients of positive powers of $q$ in $\cY_{n;\a}$ are divisible by $\x^{\ell^-(\a)}$,
$$\fR_0=\lr\a
\left\llbracket \frac{\ne^{\x z}}{\prod\limits_{k=1}^n(\x\!-\!\al_k)}
\right\rrbracket_{\x;-\ell(\a)-1}\in\Q_{\al}[z]
\subset\Q_{\al}[\hb]\big[\big[z,q\big]\big].$$
The residue $\fR_{\i}$ is computed by replacing $\x$ with $1/w$ and simplifying.
Since the coefficient of~$q^d$ in $\cY_{n;\a}(1/w,\hb,q^d)$ vanishes to
order $(n\!-\!|\a|)d$ at $w\!=\!0$, a direct computation gives
\begin{equation*}\begin{split}
&-\fR_{\i}=\lr\a\sum_{d_1,d_2=0}^{\i}\sum_{p=0}^{\i}
\frac{z^{n-1-\ell(\a)+p+(n-|\a|)(d_1+d_2)}}
{(n\!-\!1\!-\!\ell(\a)+p+(n\!-\!|\a|)(d_1\!+\!d_2))!}
q^{d_1+d_2}\ne^{\hb d_1z}\\
&\qquad\times\left\llbracket
\frac{\prod\limits_{a_k>0}\prod\limits_{r=1}^{a_kd_1}\!(a_k\!+\!r\hb w)
\prod\limits_{a_k<0}\!\!\!\!\prod\limits_{r=0}^{-a_kd_1-1}\!\!\!\!\!\!(a_k\!-\!r\hb w)
\cdot \prod\limits_{a_k>0}\!\prod\limits_{r=1}^{a_kd_2}\!(a_k\!-\!r\hb w)
\prod\limits_{a_k<0}\!\!\!\!\prod\limits_{r=0}^{-a_kd_2-1}\!\!\!\!\!\!(a_k\!+\!r\hb w)}
{\prod\limits_{k=1}^n(1\!-\!\al_kw)
\cdot\prod\limits_{r=1}^{d_1}\prod\limits_{k=1}^{n}(1\!-\!(\al_k\!-\!r\hb)w)\prod\limits_{r=1}^{d_2}\prod\limits_{k=1}^{n}(1\!-\!(\al_k\!+\!r\hb)w)}\right\rrbracket_{w;p}.
\end{split}\end{equation*}
The $(d_1,d_2,p)$-summand above is $q^{d_1\!+\!d_2}$\!
times an element of $\Q_{\al}[\hb][[z]]$.
\end{proof}

\noindent
In the case of  products of projective spaces and concavex sheaves~\e_ref{gensheaf_e},
Definition~\ref{recur_dfn} becomes inductive on the total degree
$d_1\!+\!\ldots\!+\!d_p$ of $q_1^{d_1}\!\ldots\!q_p^{d_p}$.
The power series~$\F$ is evaluated at $(\x_1,\ldots,\x_p)\!=\!(\al_{1;i_1},\ldots,\al_{p;i_p})$
for the purposes of the $C$-recursivity condition~\e_ref{recurdfn_e}
and~\e_ref{recurdfn_e2}.
The relevant primary structure coefficients are of the form
$$\fC_{i_1\ldots i_p}^j(s;d)\equiv  \frac{
\prod\limits_{a_{k;1}\ge0}
\prod\limits_{r=1}^{a_{k;s}d}\!\!\left(
 \sum\limits_{t=1}^p a_{k;t}\al_{t;i_t}+r\,\frac{\al_{s;j}-\al_{s;i_s}}{d}\right)
\prod\limits_{a_{k;1}<0}\!\!
\prod\limits_{r=0}^{-a_{k;s}d-1}\!\!\left(\sum\limits_{t=1}^p a_{k;t}\al_{t;i_t}
-r\,\frac{\al_{s;j}-\al_{s;i_s}}{d}\right)}
{d\underset{(r,k)\neq(d,j)}{\prod\limits_{r=1}^d\prod\limits_{k=1}^{n_s}}
            \left(\al_{s;i_s}-\al_{s;k}+r\,\frac{\al_{s;j}-\al_{s;i_s}}{d}\right)}$$
with $s\!\in\![p]$ and $j\!\neq\!i_s$.
The double sums in these equations are then replaced by triple sums over $s\!\in\![p]$,
$j\!\in\![n_s]\!-\!i_s$, and $d\!\in\!\Z^+$, and with $\F$ evaluated at
$$\x_t=\begin{cases}
\al_{s;j},&\hbox{if}~t\!=\!s;\\
\al_{t;i_t},&\hbox{if}~t\!\neq\!s;
\end{cases} \qquad
z=\frac{\al_{s;j}\!-\!\al_{s;i_s}}{d}.$$
The secondary coefficients~$\F_i^r(d)$ in~\e_ref{recurdfn_e2} now become
$\F_{i_1\ldots i_p}^r(d_1,\ldots,d_p)$, with $i_s\!\in\![n_s]$ and $d_s\!\in\!\Z^{\ge0}$.
In the analogue of Definition~\ref{SPC_dfn},
$\Phi_{\F}$  is a power series in $z_1,\ldots,z_p$
and $q_1,\ldots,q_p$, the sum taken is over all elements $(i_1,\ldots,i_p)$
of $[n_1]\!\times\!\ldots\!\times\![n_p]$, the leading fraction is replaced~by
$$\frac{\prod\limits_{a_{k;1}\ge0}\sum\limits_{s=1}^pa_{k;s}\al_{s;i_s}}
{\prod\limits_{a_{k;1}<0}\sum\limits_{s=1}^pa_{k;s}\al_{s;i_s}}
\cdot\frac{\ne^{\al_{1;i_1}z_1+\ldots+\al_{p;i_p}z_p}}
{\prod\limits_{s=1}^p\prod\limits_{k\neq i_s}(\al_{s;i_s}\!-\!\al_{s;k})}\,,$$
and the $q\ne^{\hb z}$-insertion in the first power series is replaced
by the insertions $q_1\ne^{\hb z_1},\ldots,q_p\ne^{\hb z_p}$.
The conclusion of Lemma~\ref{cYrec_lmm} holds with $i$, $d$, and $q^d$ replaced by
$(i_1,\ldots,i_p)$, $(d_1,\ldots,d_p)$, and $q_1^{d_1}\ldots q_p^{d_p}$, respectively.
The proof is nearly identical, except the last claim involves $p$~applications
of the Residue Theorem on~$S^2$.
Instead of the residue at $\x\!=\!0$ of the coefficient of~$q^0$,
there may be a residue at a value of $\x_s$ dependent on the values of the other variables~$\x_t$,
but it again would not involve~$\hb$.

\section{Recursivity for stable quotients}
\label{SQlocal_sec}

\noindent
In this section, we use the classical localization theorem~\cite{ABo}
to show that the equivariant stable quotients analogue of Givental's
$J$-function, the power series~$\cZ_{n;\a}$ given~by \e_ref{cZdfn_e},
is $\fC$-recursive with the collection~$\fC_i^j(d)$ given by~\e_ref{Cdfn_e}.
We also describe the secondary terms~$\cZ_i^r(d)$ in
the recursion~\e_ref{recurdfn_e2} for~$\cZ_{n;\a}$, establishing the following
statement.

\begin{prp}\label{cZrec_prp}
If $l\!\in\!\Z^{\ge0}$, $n\!\in\!\Z^+$, and $\a\!\in\!(\Z^*)^l$,
the power series $\cZ_{n;\a}(\x,\hb,q)$ is $\fC$-recursive, with the auxiliary coefficients
in the recursion~\e_ref{recurdfn_e2} for~$\cZ_{n;\a}$ given by
$$\cZ_i^r(d)=0\quad\forall~r\in\Z^+, \qquad
\cZ_i^0(d)=\de_{0d},$$
and for all $r\!\in\!\Z^-$
$$\sum_{d=1}^{\i}\cZ_i^r(d)q^d=
\sum_{d=1}^{\i}\frac{q^{d}}{d!}\sum_{b=0}^{d+r}
\Bigg(\bigg(\int_{\ov\cM_{0,2|d}}\frac{\E(\dot\V_{\a}^{(d)}(\al_i))\psi_1^{-r-1}\psi_2^b}
{\prod\limits_{k\neq i}\E(\dot\V_1^{(d)}(\al_i\!-\!\al_k))}\bigg)
\Rs{\hb=0}\bigg\{\frac{(-1)^b}{\hb^{b+1}}\cZ_{n;\a}(\al_i,\hb,q)\bigg\}\Bigg).$$
\end{prp}

\noindent
The proof involves a localization computation on $\ov{Q}_{0,2}(\Pn,d)$.
Thus, we need to describe the fixed loci of the $\T$-action on $\ov{Q}_{0,2}(\Pn,d)$,
their normal bundles, and the restrictions of the relevant cohomology classes
to these fixed loci.\\

\noindent
As in the case of stable maps described in \cite[Section~27.3]{MirSym},
the fixed loci of the $\T$-action on $\ov{Q}_{0,m}(\P^{n-1},d)$
are indexed by connected \sf{decorated graphs} that have no loops.
However, in
the case $m\!=\!2$, the relevant graphs consist of a single \sf{strand} (possibly consisting
of a single vertex)
with the two marked points attached at the opposite ends of the strand.
Such a graph can be described by an ordered set $(\Ver,<)$ of vertices,
where $<$ is a strict order on the finite set~$\Ver$.
Given such a strand, denote by $v_{\min}$ and $v_{\max}$ its minimal and maximal elements
and by
$\Edg$ its set of~\sf{edges},
i.e.~of pairs of consecutive elements.
A \sf{decorated strand} is a tuple
\BE{decortgraphdfn_e} \Ga = \big(\Ver,<;\mu,\d\big),\EE
where $(\Ver,<)$ is a strand as above and
$$\mu\!:\Ver\lra [n] \qquad\hbox{and}\qquad \d\!: \Ver\!\sqcup\!\Edg\lra\Z^{\ge0}$$
are maps such that
\BE{decorgraphcond_e}
\mu(v_1)\neq\mu(v_2)  ~~~\hbox{if}~~ \{v_1,v_2\}\in\Edg, \qquad
\d(e)\neq0~~\forall\,e\!\in\!\Edg.\EE
In Figure~\ref{loopgraph_fig}, the vertices of a decorated strand $\Ga$
are indicated by dots in the increasing order, with respect to~$<$, from left to right.
The values of the map $(\mu,\d)$ on some of the vertices
are indicated next to those vertices.
Similarly, the values of the map $\d$ on some of the edges are indicated next to~them.
By~\e_ref{decorgraphcond_e}, no two consecutive vertices have the same first label
and thus $j\!\neq\!i$.\\

\noindent
With $\Ga$ as in~\e_ref{decortgraphdfn_e}, let
$$|\Ga|\equiv\sum_{v\in\Ver}\!\!\d(v)+\sum_{e\in\Edg}\!\!\d(e)$$
be \textsf{the degree of~$\Ga$}.
If $e\!=\!\{v_1,v_2\}\!\in\!\Edg$ is any edge in $\Ga$, let
$\Ga_e$ denote the single-edge graph with vertices $v_1$ and~$v_2$,
which are ordered in the same way as in~$\Ga$
and assigned values $(\mu(v_1),0)$ and $(\mu(v_2),0)$, and with the
edge assigned the value~$\d(e)$ as in the original graph;
see Figure~\ref{subgraphs_fig}.\\

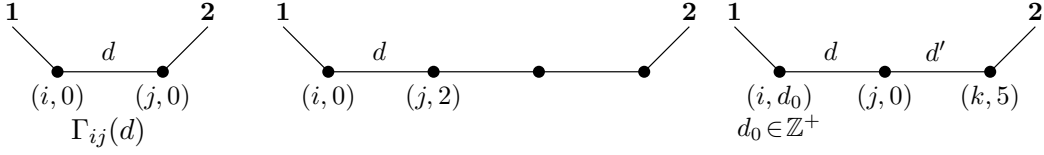
\begin{figure}
\begin{pspicture}(-.7,-1)(10,1.2)
\psset{unit=.4cm}
\psline[linewidth=.04](2.5,0)(6,0)\rput(4.2,.7){\smsize{$d$}}
\pscircle*(2.5,0){.2}\rput(2.5,-.85){\smsize{$(i,0)$}}
\pscircle*(6,0){.2}\rput(6,-.85){\smsize{$(j,0)$}}
\psline[linewidth=.04](2.5,0)(1,1.5)\rput(1,2){\smsize{$\bf 1$}}
\psline[linewidth=.04](6,0)(7.5,1.5)\rput(7.5,2){\smsize{$\bf 2$}}
\rput(4.2,-2){$\Ga_{ij}(d)$}
\psline[linewidth=.04](11.5,0)(15,0)\rput(13.2,.7){\smsize{$d$}}
\psline[linewidth=.04](15,0)(18.5,0)\pscircle*(18.5,0){.2}
\psline[linewidth=.04](18.5,0)(22,0)\pscircle*(22,0){.2}
\pscircle*(11.5,0){.2}\rput(11.5,-.85){\smsize{$(i,0)$}}
\pscircle*(15,0){.2}\rput(15,-.85){\smsize{$(j,2)$}}
\psline[linewidth=.04](11.5,0)(10,1.5)\rput(10,2){\smsize{$\bf 1$}}
\psline[linewidth=.04](22,0)(23.5,1.5)\rput(23.5,2){\smsize{$\bf 2$}}
\psline[linewidth=.04](26.5,0)(30,0)\pscircle*(30,0){.2}
\rput(28.2,.7){\smsize{$d$}}\rput(31.7,.7){\smsize{$d'$}}
\pscircle*(26.5,0){.2}\rput(26.5,-.85){\smsize{$(i,d_0)$}}
\rput(26.5,-1.85){\smsize{$d_0\!\in\!\Z^+$}}
\rput(30,-.85){\smsize{$(j,0)$}}
\rput(33.5,-.85){\smsize{$(k,5)$}}
\psline[linewidth=.04](30,0)(33.5,0)\pscircle*(33.5,0){.2}
\psline[linewidth=.04](26.5,0)(25,1.5)\rput(25,2){\smsize{$\bf 1$}}
\psline[linewidth=.04](33.5,0)(35,1.5)\rput(35,2){\smsize{$\bf 2$}}
\end{pspicture}
\caption{Two strands with $\d(v_{\min})\!=\!0$ and a strand with $\d(v_{\min})\!>\!0$}
\label{loopgraph_fig}
\end{figure}

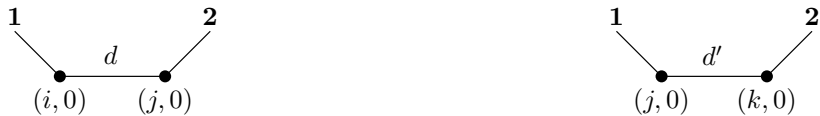
\begin{figure}
\begin{pspicture}(-2.5,-.6)(10,1)
\psset{unit=.4cm}
\psline[linewidth=.04](2.5,0)(6,0)\rput(4.2,.7){\smsize{$d$}}
\pscircle*(2.5,0){.2}\rput(2.5,-.85){\smsize{$(i,0)$}}
\pscircle*(6,0){.2}\rput(6,-.85){\smsize{$(j,0)$}}
\psline[linewidth=.04](2.5,0)(1,1.5)\rput(1,2){\smsize{$\bf 1$}}
\psline[linewidth=.04](6,0)(7.5,1.5)\rput(7.5,2){\smsize{$\bf 2$}}
\psline[linewidth=.04](22.5,0)(26,0)\rput(24.2,.7){\smsize{$d'$}}
\pscircle*(22.5,0){.2}\rput(22.5,-.85){\smsize{$(j,0)$}}
\pscircle*(26,0){.2}\rput(26,-.85){\smsize{$(k,0)$}}
\psline[linewidth=.04](22.5,0)(21,1.5)\rput(21,2){\smsize{$\bf 1$}}
\psline[linewidth=.04](26,0)(27.5,1.5)\rput(27.5,2){\smsize{$\bf 2$}}
\end{pspicture}
\caption{The sub-strands corresponding to the edges of the last graph in Figure~\ref{loopgraph_fig}.}
\label{subgraphs_fig}
\end{figure}

\noindent
As described in \cite[Section~7.3]{MOP09},
the fixed locus $Q_{\Ga}$ of $\ov{Q}_{0,2}(\P^{n-1},|\Ga|)$ corresponding to a
decorated strand $\Ga$ consists of the stable quotients
$$(\cC,y_1,y_2,S\subset \C^n\!\otimes\!\cO_{\cC})$$
over quasi-stable rational 2-marked curves that satisfy the following conditions.
The components of $\cC$ on which the corresponding quotient is torsion-free
are rational and correspond to the edges of~$\Ga$;
the restriction of~$S$ to any such component corresponds
to a morphism to~$\P^{n-1}$ of the opposite degree to that of the subsheaf.
Furthermore, if $e\!=\!\{v_1,v_2\}$ is an edge, the corresponding morphism~$f_e$
is a degree-$\d(e)$ cover of the line
$$\P^1_{\mu(v_1),\mu(v_2)}\subset\P^{n-1}$$
passing through the fixed points $P_{\mu(v_1)}$ and $P_{\mu(v_2)}$;
it is ramified only over $P_{\mu(v_1)}$ and~$P_{\mu(v_2)}$.
In particular, $f_e$ is unique up to isomorphism.
The remaining components of $\cC$ are
indexed by the vertices $v\!\in\!\Ver$ with $\d(v)\!\in\!\Z^+$.
The restriction of~$S$ to such a component~$\cC_v$ of~$\cC$ (or possibly a connected union
of irreducible components) is a subsheaf
of the trivial subsheaf $P_{\mu(v)}\!\subset\!\C^n\!\otimes\!\cO_{\cC_v}$
of degree~$-\d(v)$;
thus, the induced morphism takes~$\cC_v$ to the fixed point~$P_{\mu(v)}\!\in\!\Pn$.
Each such component~$\cC_v$ also carries two distinguished marked points, corresponding
to the nodes and/or the marked points of~$\cC$;
if neither of the marked points of~$\cC$ lies on~$\cC_v$, we denote the marked point
corresponding to the node of~$\cC_v$ separating~$\cC_v$ from the first marked point
by~1 and the other marked point by~2.
Thus, as stacks,
\BE{Zlocus_e}\begin{split}
Q_{\Ga}&\approx \prod_{\begin{subarray}{c}v\in\Ver\\ \d(v)>0\end{subarray}}
\!\!\!\ov{Q}_{0,2}(\P^0,\d(v))\times\prod_{e\in\Edg}\!\!Q_{\Ga_e}
\approx  \prod_{\begin{subarray}{c}v\in\Ver\\ \d(v)>0\end{subarray}}
\!\!\!\ov\cM_{0,2|\d(v)}/\bS_{\d(v)}\times\prod_{e\in\Edg}\!\!Q_{\Ga_e}\\
&\approx
\Bigg(\prod_{\begin{subarray}{c}v\in\Ver\\ \d(v)>0\end{subarray}}
\!\!\!\ov\cM_{0,2|\d(v)}/\bS_{\d(v)}\bigg)\Bigg/\prod_{e\in\Edg}\!\!\Z_{\d(e)},
\end{split}\EE
with each cyclic group $\Z_{\d(e)}$ acting trivially.
For example, in the case of the last diagram in Figure~\ref{loopgraph_fig},
$$Q_{\Ga}\approx
\Big(\ov\cM_{0,2|d_0}/\bS_{d_0}\times \ov\cM_{0,2|5}/\bS_5\Big)
\big/\Z_d\!\times\!\Z_{d'}$$
is a fixed locus in $\ov{Q}_{0,2}(\P^{n-1},d_0\!+\!5\!+\!d\!+\!d')$.\\

\noindent
If $\Ga$ is a decorated strand as above and $e\!\!\in\!\Edg$,
let
$$\pi_e\!:Q_{\Ga}\lra Q_{\Ga_e}\subset \ov{Q}_{0,2}(\Pn,\d(e))$$
be the projection in the decomposition~\e_ref{Zlocus_e}.
Similarly, for each $v\!\in\!\Ver$ such that $\d(v)\!>\!0$, let
$$\pi_v\!:Q_{\Ga}\lra\ov\cM_{0,2|\d(v)}/\bS_{\d(v)}$$
be the corresponding projection.
If $e\!=\!\{v_1,v_2\}\!\in\!\Edg$ with $v_1\!<\!v_2$, let
\BE{omdfn_e}\om_{e;v_1}=-\pi_e^*\psi_1,\,\om_{e;v_2}=-\pi_e^*\psi_2,\,
\psi_{v_1;e}=\pi_{v_1}^*\psi_2,\,\psi_{v_2;e}=\pi_{v_2}^*\psi_1
\in H^2(Q_{\Ga})\,.\EE
By \cite[Section~27.2]{MirSym},
\BE{psiform_e} \om_{e;v_i}=\frac{\al_{\mu(v_i)}-\al_{\mu(v_{3-i})}}{\d(e)}\
\qquad i=1,2.\EE
For each $v\!\in\!\Ver-\{v_{\min}\}$, let $e_-(v)\!=\!\{v_-,v\}\!\in\!\Edg$ denote
the edge with $v_-\!<\!v$;
for each $v\!\in\!\Ver\!-\!\{v_{\max}\}$, let $e_+(v)\!=\!\{v,v_+\}\!\in\!\Edg$ denote
the edge with $v\!<\!v_+$.\\

\noindent
By \cite[Section~7.4]{MOP09}, the Euler class of the normal bundle of $Q_{\Ga}$ in
$\ov{Q}_{0,2}(\Pn,|\Ga|)$ is given~by
\BE{NZform_e}\begin{split}
&\frac{\E(\N Q_{\Ga})}{\E(T_{\mu(v_{\min})}\Pn)}
=\prod_{\begin{subarray}{c}v\in\Ver\\ \d(v)>0\end{subarray}}
\prod_{k\neq\mu(v)}\!\!\!\pi_v^*\E\big(\dot\V_1^{(\d(v))}(\al_{\mu(v)}\!-\!\al_k)\big)~
\prod_{e\in\Edg}\!\!\!\pi_e^*\E\big(H^0(f_e^*T\P^n\!\otimes\!\cO(-y_1))/\C\big)\\
&\quad
\times\prod_{\begin{subarray}{c}v\in\Ver-v_{\min}-v_{\max}\\ \d(v)=0\end{subarray}}
\!\!\!\!\!\!\!\!\!\!\!\!\!\!\big(\om_{e_-(v);v}\!+\!\om_{e_+(v);v}\big)
\prod_{\begin{subarray}{c}v\in\Ver-v_{\min}\\ \d(v)>0\end{subarray}}
\!\!\!\!\!\!\!\!\!\big(\om_{e_-(v);v}\!-\!\psi_{v;e_-(v)}\big)
\prod_{\begin{subarray}{c}v\in\Ver-v_{\max}\\ \d(v)>0\end{subarray}}
\!\!\!\!\!\!\!\!\!\big(\om_{e_+(v);v}\!-\!\psi_{v;e_+(v)}\big),
\end{split}\EE
where $\C\!\subset\!H^0(f_e^*T\P^n\!\otimes\!\cO(-y_1))$ is the trivial $\T$-representation
and $\dot\V_1^{(\d(v))}(\al_{\mu(v)}\!-\!\al_k)$ is as in \e_ref{V1dfn_e}.
The terms on the second line in~\e_ref{NZform_e} describe the standard deformations
of the domain; they are given by the direct sum of the tensor products of
the tangent line bundles at the two branches of each node.
The terms on the first line in~\e_ref{NZform_e} correspond to the deformations of
the sheaf without changing the domain~$\cC$;
they are obtained by relating these deformations to the deformations on each components
of~$\cC$ and applying~\e_ref{tangrestr_e} to the deformations over the components~$\cC_v$
corresponding to the vertices.
The first term on the right-hand side of~\e_ref{NZform_e} and the first two terms
on the second line of~\e_ref{NZform_e} are the contributions
of nondegenerate vertices described in \cite[Section~7.4.2]{MOP09}.
The second term on the right-hand side of~\e_ref{NZform_e} is the edge contributions,
which are the same as in Gromov-Witten theory.
Finally, by~\e_ref{Vprdfn_e} and~\e_ref{V0prdfn_e},
\BE{cVform_e}
\E(\dot\V_{n;\a}^{(|\Ga|)})\big|_{Q_{\Ga}}
=\prod_{\begin{subarray}{c}v\in\Ver\\ \d(v)>0\end{subarray}}
\!\!\!\pi_v^*\E\big(\dot\V_{\a}^{(\d(v))}(\al_{\mu(v)})\big)
\cdot \prod_{e\in\Edg}\!\!\!\pi_e^*\E\big(\dot\V_{n;\a}^{\d(e)}\big).\EE

\begin{lmm}\label{edgecontr_lmm}
For every edge $e\!=\!\{v_1,v_2\}$ with $v_1\!<\!v_2$ in $\Ga$ as above,
\BE{EdgContr_e}
\int_{Q_{\Ga_e}}
\frac{\E(\dot\V_{n;\a}^{\d(e)})}{\E\big(H^0(f_e^*T\P^n\!\otimes\!\cO(-y_1))/\C\big)}
=\fC_{\mu(v_1)}^{\mu(v_2)}\big(\d(e)\big)\,,\EE
with $\fC_{\mu(v_1)}^{\mu(v_2)}(\d(e))$  given by~\e_ref{Cdfn_e}.
\end{lmm}

\begin{proof}
Since the edge contributions are the same as in Gromov-Witten theory,
\e_ref{EdgContr_e} is standard;
we recall its derivation for the sake of completeness.
Let $i\!=\!\mu(v_1)$, $j\!=\!\mu(v_2)$, and $d\!=\!\d(e)$.\\

\noindent
By \cite[Exercise~27.2.3]{MirSym},
\BE{H0equiv_e} \E\big(H^0(f_e^*\cO_{\Pn}(a_k))\big)=
\prod_{r=0}^{a_kd}\frac{(a_kd\!-\!r)\al_i+r\al_j}{d}\,
\qquad\forall~a_k\!\in\!\Z^{\ge0}.\EE
Since $\E(\cO_{\Pn}(a_k))|_{P_i}\!=\!a_k\al_i$ and the sequence
$$0\lra H^0\big(f_e^*\cO_{\Pn}(a_k)\!\otimes\!\cO(-y_1)\big)\lra
H^0\big(f_e^*\cO_{\Pn}(a_k)\big)\lra \cO_{\Pn}(a_k)|_{P_i} \lra 0$$
is exact, the product of \e_ref{H0equiv_e} without the $r\!=\!0$ factor
over~$k$ with~$a_k\!>\!0$, i.e.~the first product in the numerator of~\e_ref{Cdfn_e},
is the equivariant Euler class of the first summand in~\e_ref{Vprdfn_e} restricted to~$f_e$.
By~Serre Duality and \cite[Exercises~27.2.2, 27.2.3]{MirSym},
\BE{H1equiv_e} \E\big(H^1(f_e^*\cO_{\Pn}(a_k))\big)=
\prod_{r=1}^{-a_kd-1}\frac{(a_kd\!+\!r)\al_i-r\al_j}{d}\,
\qquad\forall~a_k\!\in\!\Z^-.\EE
Since the sequence
$$0\lra \cO_{\Pn}(a_k)|_{P_i}
\lra H^1\big(f_e^*\cO_{\Pn}(a_k)\!\otimes\!\cO(-y_1)\big)\lra
H^1\big(f_e^*\cO_{\Pn}(a_k)\big)\lra 0$$
is exact,
the product of \e_ref{H1equiv_e} with the extra $r\!=\!0$ factor
over~$k$ with~$a_k\!<\!0$, i.e.~the second product in the numerator of~\e_ref{Cdfn_e},
is the equivariant Euler class of the second summand in~\e_ref{Vprdfn_e} restricted to~$f_e$.
Thus, the numerators in~\e_ref{EdgContr_e} and~\e_ref{Cdfn_e} are the same.\\

\noindent
The denominator  in~\e_ref{EdgContr_e} is computed using the exact sequence
\BE{Tses_e}\begin{split}
0&\lra H^0\big(f_e^*T\P_{i,j}^1\!\otimes\!\cO(-y_1)\big)\big/\C
\lra H^0\big(f_e^*T\P^n\!\otimes\!\cO(-y_1)\big)/\C\\
&\lra \bigoplus_{k\neq i,j}\!
H^0\big(f_e^*\cO_{\Pn}(1)\!\otimes\!\C_{\al_i\!-\!\al_k}\!\otimes\!\cO(-y_1)\big)\lra0,
\end{split}\EE
where $\C_{\al_i\!-\!\al_k}$ is the topologically trivial line bundle with equivariant Euler
class~$\al_i\!-\!\al_k$;
this sequence  is obtained from the equivariant Euler sequence for~$\Pn$.
The equivariant Euler class of each summand on the second line of~\e_ref{Tses_e}
is given by~\e_ref{H0equiv_e} with $a_k\!=\!1$, each factor increased
by $\al_i\!-\!\al_k$ (because of the tensor product with the line bundle~$\C_{\al_i\!-\!\al_k}$),
and the $r\!=\!0$ factor again dropped.
Thus, the equivariant Euler class of the vector space on the second line of~\e_ref{Tses_e}
is the product of the factors in the denominator of~\e_ref{Cdfn_e} with~$k\!\neq\!i,j$.
By \cite[Exercise~27.2.3]{MirSym},
\BE{H0Tequiv_e}
\E\big(H^0(f_e^*T\P_{i,j}^1)\big)
=\prod_{r=0}^{2d}\frac{(d\!-\!r)(\al_i\!-\!\al_j)
+r(\al_j\!-\!\al_i)}{d}\,.\EE
Since $\E(T\P_{i,j}^1)|_{P_i}\!=\!\al_i\!-\!\al_j$
and the sequence
$$0\lra H^0\big(f_e^*T\P_{i,j}^1\!\otimes\!\cO(-y_1)\big)\lra
H^0\big(f_e^*T\P_{i,j}^1\big)\lra   T\P_{i,j}^1|_{P_i} \lra 0$$
is exact, \e_ref{Tses_e} and \e_ref{H0Tequiv_e} give
$$\E\big(H^0(f_e^*T\P_{i,j}^1\!\otimes\!\cO(-y_1))/\C\big)
=\prod_{r=1}^d\frac{r(\al_j\!-\!\al_i)}{d}  \cdot
\prod_{r=1}^{d-1}\frac{r(\al_i\!-\!\al_j)}{d}\,.$$
Thus, the denominator in~\e_ref{EdgContr_e} equals to the product in
the denominator of~\e_ref{Cdfn_e}.
The remaining factor, $d$, in the denominator of~\e_ref{Cdfn_e} accounts for
the automorphism group of~$Q_{\Ga_e}$.
\end{proof}

\noindent
Proposition~\ref{cZrec_prp} is proved by applying
the localization theorem~to
\BE{Zeval_e}\cZ_{n;\a}(x\!=\!\al_i,\hb,q)
=1+\sum_{d=1}^{\i}q^d\int_{\ov{Q}_{0,2}(\Pn,d)}
\frac{\E(\dot\V_{n;\a}^{(d)})\ev_1^*\phi_i}{\hb\!-\!\psi_1}
\in \Q_{\al}\big[\big[\hb^{-1},q\big]\big],\EE
where $\phi_i$ is the equivariant Poincare dual of the fixed point $P_i\!\in\!\Pn$;
see~\e_ref{phidfn_e}, \e_ref{phiprop_e}, and~\e_ref{pushdfn_e}.
Since $\phi_i|_{P_j}\!=\!0$ unless $j\!=\!i$, a decorated strand as in~\e_ref{decortgraphdfn_e}
contributes to~\e_ref{Zeval_e} only if the first marked point is attached
to a vertex labeled~$i$, i.e.~$\mu(v_{\min})\!=\!i$ for the smallest element $v_{\min}\!\in\!\Ver$.
We show that, just as with Givental's $J$-function, the $(d,j)$-summand in~\e_ref{recurdfn_e2}
with $C\!=\!\fC$ and $\F\!=\!\cZ_{n;\a}$, i.e.
$$\frac{\fC_i^j(d)q^d}{\hb-\frac{\al_j-\al_i}{d}}\cZ_{n;\a}(\al_j,(\al_j\!-\!\al_i)/d,q)\,,$$
is the sum over all strands such that $\mu(v_{\min})\!=\!i$, i.e.~the first marked point
is mapped to the fixed point $P_i\!\in\!\Pn$,
$v_{\min}$ is a bivalent vertex,  i.e.~$\d(v_{\min})\!=\!0$,
the only edge leaving this vertex is labeled~$d$, and
the other vertex of this edge is labeled~$j$.
We also show that the first sum on the right-hand side of~\e_ref{recurdfn_e2} is~1
(for the degree~0 term) plus the sum over  all strands such that $\mu(v_{\min})\!=\!i$ and
$\d(v_{\min})\!>\!0$.\\

\noindent
If $\Ga$ is a decorated strand with $\mu(v_{\min})\!=\!i$ as above,
\BE{phirestr_e} \ev_1^*\phi_i\big|_{Q_{\Ga}}
=\prod_{k\neq i}(\al_i\!-\!\al_k)
=\E\big(T_{\mu(v_{\min})}\Pn\big).\EE
Suppose in addition that $\d(v_{\min})\!=\!0$.
Let $v_1\!\equiv\!(v_{\min})_+$ be the immediate successor of $v_{\min}$ in~$\Ga$
and $e_1\!=\!\{v_{\min},v_1\}$ be the edge leaving~$v_{\min}$.
If $|\Edg|\!>\!1$ or $\d(v_1)\!>\!0$
(i.e.~$\Ga$ is not as in the first diagram in Figure~\ref{loopgraph_fig}),
we break~$\Ga$ at $v_1$ into two ``sub-strands":
\begin{enumerate}[label=(\roman*)]
\item $\Ga_1\!=\!\Ga_{e_1}$ consisting of the vertices $v_{\min}\!<\!v_1$,
the edge $\{v_{\min},v_1\}$, and the $\d$-value of~0 at both vertices;
\item $\Ga_2$ consisting all vertices and edges of $\Ga$, other than
the vertex $v_{\min}$ and the edge $\{v_{\min},v_1\}$;
\end{enumerate}
see Figure~\ref{splitgraph_fig}.
By~\e_ref{Zlocus_e},
$$Q_{\Ga}\approx Q_{\Ga_1}\times Q_{\Ga_2}.$$
Let $\pi_1,\pi_2\!:Q_{\Ga}\lra Q_{\Ga_1},Q_{\Ga_2}$ be the two component projection maps.
By~\e_ref{cVform_e} and \e_ref{NZform_e},
\begin{equation*}\begin{split}
\E\big(\dot\V_{n;\a}^{(|\Ga|)}\big)\big|_{Q_{\Ga}}&=
  \pi_1^*\E\big(\dot\V_{n;\a}^{(|\Ga_1|)}\big)\cdot\pi_2^*\E(\dot\V_{n;\a}^{(|\Ga_2|)}\big)\,,\\
\frac{\E(\N Q_{\Ga})}{\E(T_{P_i}\Pn)}&=
\pi_1^*\bigg(\frac{\E(\N Q_{\Ga_1})}{\E(T_{P_i}\P^{n-1})}\bigg)
\cdot\pi_2^*\bigg(\frac{\E(\N Q_{\Ga_2})}{\E(T_{P_{\mu(v_1)}}\P^{n-1})}\bigg)
\cdot \big(\om_{e_1;v_1}-\pi_2^*\psi_1\big).
\end{split}\end{equation*}
Combining this with~\e_ref{psiform_e}, \e_ref{EdgContr_e}, and~\e_ref{phirestr_e}, we find that
\BE{bndlsplit_e6}\begin{split}
&q^{|\Ga|}\!\int_{Q_{\Ga}}
\frac{\E(\dot\V_{n;\a}^{(|\Ga|)})\ev_1^*\phi_i}{(\hb\!-\!\psi_1)\E(\N Q_{\Ga})}\\
&\hspace{.6in}
=\frac{\fC_i^{\mu(v_1)}(\d(e_1))q^{\d(e_1)}}{\hb-\frac{\al_{\mu(v_1)}-\al_i}{\d(e_1)}}\cdot
\Bigg(q^{|\Ga_2|}\!\bigg\{\int_{Q_{\Ga_2}}
\frac{\E(\dot\V_{n;\a}^{(|\Ga_2|)})\ev_1^*\phi_{\mu(v_1)}}
{(\hb\!-\!\psi_1)\E(\N Q_{\Ga_2})}\bigg\}\bigg|_{\hb=\frac{\al_{\mu(v_1)}-\al_i}{\d(e_1)}}\Bigg).
\end{split}\EE
By \e_ref{Zeval_e} with $i$ replaced by $\mu(v_1)$ and
the localization formula~\e_ref{ABothm_e}, the sum of the last factors over all possibilities
for~$\Ga_2$, with $\Ga_1$ held fixed,~is
$$\cZ_{n;\a}\big(\al_{\mu(v_1)},(\al_{\mu(v_1)}\!-\!\al_i)/\d(e_1),q\big)-1.$$
On the other hand, the contribution of the graph $\Ga_{i\mu(v_1)}(\d(e_1))$
as in the first diagram in Figure~\ref{loopgraph_fig} is precisely the first factor
on the right-hand side of~\e_ref{bndlsplit_e6}.
Thus, the contribution to~\e_ref{Zeval_e} from all strands~$\Ga$
such that $\mu(v_1)\!=\!j$ and $\d(e_1)\!=\!d$ is
$$\frac{\fC_i^j(d)q^d}{\hb-\frac{\al_j-\al_i}{d}}
\cZ_{n;\a}(\al_j,(\al_j\!-\!\al_i)/d,q\big),$$
i.e.~the $(d,j)$-summand in the recursion~\e_ref{recurdfn_e2} for~$\cZ_{n;\a}$.\\

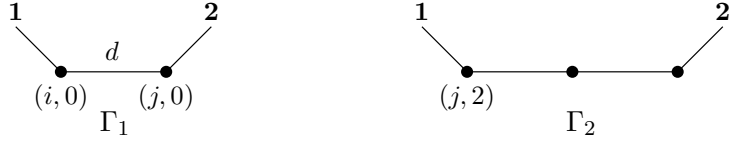
\begin{figure}
\begin{pspicture}(-3.5,-1)(10,1)
\psset{unit=.4cm}
\psline[linewidth=.04](1.5,0)(5,0)\rput(3.2,.7){\smsize{$d$}}
\pscircle*(1.5,0){.2}\rput(1.5,-.85){\smsize{$(i,0)$}}
\pscircle*(5,0){.2}\rput(5,-.85){\smsize{$(j,0)$}}
\psline[linewidth=.04](1.5,0)(0,1.5)\rput(0,2){\smsize{$\bf 1$}}
\psline[linewidth=.04](5,0)(6.5,1.5)\rput(6.5,2){\smsize{$\bf 2$}}
\rput(3.3,-1.7){$\Ga_1$}
\psline[linewidth=.04](15,0)(18.5,0)\pscircle*(18.5,0){.2}
\psline[linewidth=.04](18.5,0)(22,0)\pscircle*(22,0){.2}
\pscircle*(15,0){.2}\rput(15,-.85){\smsize{$(j,2)$}}
\psline[linewidth=.04](15,0)(13.5,1.5)\rput(13.5,2){\smsize{$\bf 1$}}
\psline[linewidth=.04](22,0)(23.5,1.5)\rput(23.5,2){\smsize{$\bf 2$}}
\rput(18.8,-1.7){$\Ga_2$}\end{pspicture}
\caption{The two sub-strands of the second strand in Figure~\ref{loopgraph_fig}.}
\label{splitgraph_fig}
\end{figure}

\noindent
Suppose next that $\Ga$ is a strand such that $\mu(v_{\min})\!=\!i$ and $\d(v_{\min})\!>\!0$.
If $|\Ver|\!>\!1$, i.e.~$\Ga$ is not as in the first diagram in Figure~\ref{splitgraph_fig2},
we break~$\Ga$ at $v_{\min}$ into two ``sub-strands":
\begin{enumerate}[label=(\roman*)]
\item $\Ga_0$ consisting of the vertex $\{v_{\min}\}$ only,
with the same $\mu$ and $\d$-values as in~$\Ga$;
\item $\Ga_c$ consisting all vertices and edges of $\Ga$, but
with the $\d$-value of $v_{\min}$ replaced by~0;
\end{enumerate}
see Figure~\ref{splitgraph_fig2}.
By~\e_ref{Zlocus_e},
\BE{Zlocus_e2}
Q_{\Ga}\approx Q_{\Ga_0}\times Q_{\Ga_c}=
(\ov\cM_{0,2|\d(v_{\min})}/\bS_{\d(v_{\min})})\times Q_{\Ga_c};\EE
if $|\Ver|\!=\!1$, this decomposition holds with $Q_{\Ga_c}\!\equiv\!\{pt\}$
and $\d(v_{\min})\!=\!|\Ga|$.
Let $\pi_0,\pi_c$ be the two component projection maps in~\e_ref{Zlocus_e2}.
Since
$$\psi_1|_{Q_{\Ga}}=\pi_0^*\psi_1\,,$$
$\T$ acts trivially on $\ov\cM_{0,2|\d(v_{\min})}$,
$$\psi_1=1\times\psi_1\in H_{\T}^*\big(\ov\cM_{0,2|\d(v_{\min})}\big)
=H_{\T}^*\otimes H^*\big(\ov\cM_{0,2|\d(v_{\min})}\big),$$
i.e.~$\T$ acts trivially on the universal cotangent line bundle for the first marked point
on $\ov\cM_{0,2|\d(v_{\min})}$, and
the dimension of $\ov\cM_{0,2|\d(v_{\min})}$ is $\d(v_{\min})\!-\!1$,
\BE{hpsirestr_e}\frac{1}{\hb-\psi_1}\big|_{Q_{\Ga}}
=\sum_{r=0}^{\d(v_{\min})-1}\hb^{-(r+1)}\pi_0^*\psi_1^r\,.\EE
Since $|\d(v_{\min})|\!\le\!|\Ga|$ and $\Ga$ contributes to the coefficient of $q^{|\Ga|}$
in~\e_ref{Zeval_e}, it follows that~$\cZ_{n;\a}$ satisfies~\e_ref{recurdfn_e2} with
$\F\!=\!\cZ_{n;\a}$, $C_i^j(d)\!=\!\fC_i^j(d)$, $N_d\!=\!d$,
$\cZ_i^r(d)\!=\!0$ for $r\!\in\!\Z^+$, and $\cZ_i^0(d)\!=\!\de_{0d}$.
In particular, $\cZ_{n;\a}$ is $\fC$-recursive.\\

\noindent
It remains to verify the last identity in Proposition~\ref{cZrec_prp}.
We continue with the notation as in the previous paragraph.
If $|\Ver|\!=\!1$, the second factor in~\e_ref{Zlocus_e2} is trivial;
in this case, \e_ref{cVform_e} and \e_ref{NZform_e} immediately give
\BE{V1contr_e}
q^{|\Ga|}\!\int_{Q_{\Ga}}
\frac{\E(\dot\V_{n;\a}^{(|\Ga|)})\ev_1^*\phi_i}{(\hb\!-\!\psi_1)\E(\N Q_{\Ga})}
=\sum_{r=0}^{|\Ga|-1}\!\hb^{-(r+1)}
\frac{q^{|\Ga|}}{(|\Ga|)!}\int_{\ov\cM_{0,2||\Ga|}}\frac{\E(\dot\V_{\a}^{(|\Ga|)}(\al_i))\psi_1^r}
{\prod\limits_{k\neq i}\E(\dot\V_1^{(|\Ga|)}(\al_i\!-\!\al_k))} \,.
\EE
Suppose next that $|\Ver|\!>\!1$.
By~\e_ref{cVform_e} and \e_ref{NZform_e},
\begin{equation*}\begin{split}
\E\big(\dot\V_{n;\a}^{(|\Ga|)}\big)\big|_{Q_{\Ga}}&=
  \pi_0^*\E\big(\dot\V_{\a}^{(|\Ga_0|)}(\al_i)\big)\cdot\pi_c^*\E(\dot\V_{n;\a}^{(|\Ga_c|)}\big)\,,\\
\frac{\E(\N Q_{\Ga})}{\E(T_{P_i}\Pn)}&=
\pi_0^*\prod_{k\neq i}\E\big(\dot\V_1^{(|\Ga_0|)}(\al_i\!-\!\al_k)\big)
\cdot\pi_c^*\bigg(\frac{\E(\N Q_{\Ga_c})}{\E(T_{P_i}\P^{n-1})}\bigg)
\cdot \big(\om_{e_1;v_{\min}}-\pi_0^*\psi_2\big),
\end{split}\end{equation*}
where $e_1$ is the edge leaving~$v_{\min}$.
By~\e_ref{omdfn_e},
$$\frac1{\om_{e_1;v_{\min}}-\pi_0^*\psi_2}
=\sum_{b=0}^{\i}\pi_0^*\psi_2^b\,(-\pi_c^*\psi_1)^{-(b+1)}\,.$$
Combining the last four identities, we find that
\BE{ghsum_e}\begin{split}
q^{|\Ga|}\!\int_{Q_{\Ga}}
\frac{\E(\dot\V_{n;\a}^{(|\Ga|)})\ev_1^*\phi_i}{(\hb\!-\!\psi_1)\E(\N Q_{\Ga})}
=\sum_{r=0}^{d_0-1}\sum_{b=0}^{d_0-1-r}\!\!\!\hb^{-(r+1)}\Bigg(
\frac{q^{d_0}}{d_0!}\int_{\ov\cM_{0,2|d_0}}
\frac{\E(\dot\V_{\a}^{(d_0)}(\al_i))\psi_1^r\psi_2^b}
{\prod\limits_{k\neq i}\E(\dot\V_1^{(d_0)}(\al_i\!-\!\al_k))}\qquad&\\
\times(-1)^{b+1}q^{|\Ga_c|}
\int_{Q_{\Ga_c}}\psi_1^{-(b+1)}\frac{\E(\dot\V_{n;\a}^{(|\Ga_c|)})\ev_1^*\phi_i}{\E(\N Q_{\Ga_c})}&\Bigg),
\end{split}\EE
where $d_0\!=\!\d(v_{\min})\!=\!|\Ga_0|$.\\

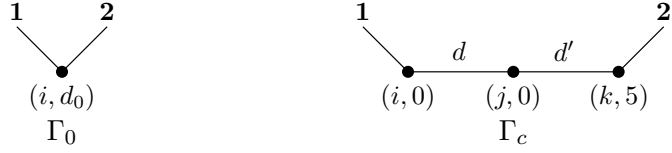
\begin{figure}
\begin{pspicture}(1.5,-.8)(10,1)
\psset{unit=.4cm}
\pscircle*(15,0){.2}\rput(15,-.85){\smsize{$(i,d_0)$}}
\psline[linewidth=.04](15,0)(13.5,1.5)\rput(13.5,2){\smsize{$\bf 1$}}
\psline[linewidth=.04](15,0)(16.5,1.5)\rput(16.5,2){\smsize{$\bf 2$}}
\rput(15,-2){$\Ga_0$}
\psline[linewidth=.04](26.5,0)(30,0)\pscircle*(30,0){.2}
\rput(28.2,.7){\smsize{$d$}}\rput(31.7,.7){\smsize{$d'$}}
\pscircle*(26.5,0){.2}\rput(26.5,-.85){\smsize{$(i,0)$}}
\rput(30,-.85){\smsize{$(j,0)$}}
\rput(33.5,-.85){\smsize{$(k,5)$}}
\psline[linewidth=.04](30,0)(33.5,0)\pscircle*(33.5,0){.2}
\psline[linewidth=.04](26.5,0)(25,1.5)\rput(25,2){\smsize{$\bf 1$}}
\psline[linewidth=.04](33.5,0)(35,1.5)\rput(35,2){\smsize{$\bf 2$}}
\rput(30,-2){$\Ga_c$}
\end{pspicture}
\caption{The two sub-strands of the last strand in Figure~\ref{loopgraph_fig}.}
\label{splitgraph_fig2}
\end{figure}

\noindent
We now sum up the last factors in~\e_ref{ghsum_e}
 over all possibilities for $\Ga_c$ with $|\Ga_c|\!>\!0$ by
decomposing $\Ga_c$
into sub-strands $\Ga_1\!=\!\Ga_{ij}(d)$, for some $j\!\in\![n]\!-\!i$ and $d\!\in\!\Z^+$,
and~$\Ga_2$, as in the case $\d(v_{\min})\!=\!0$ above.
If $\Ga_c\!\neq\!\Ga_1$, \e_ref{bndlsplit_e6} with $\Ga$ replaced by $\Ga_c$ gives
\begin{equation*}\begin{split}
&q^{|\Ga_c|}
\int_{Q_{\Ga_c}}\psi_1^{-(b+1)}\frac{\E(\dot\V_{n;\a}^{(|\Ga_c|)})\ev_1^*\phi_i}{\E(\N Q_{\Ga_c})}
=\fC_i^{\mu(v_1)}(\d(e_1))q^{\d(e_1)}
\bigg(\frac{\al_{\mu(v_1)}-\al_i}{\d(e_1)}\bigg)^{-(b+1)}\\
&\hspace{1.5in}\times
\Bigg(q^{|\Ga_2|}\!\bigg\{\int_{Q_{\Ga_2}}
\frac{\E(\dot\V_{n;\a}^{(|\Ga_2|)})\ev_1^*\phi_{\mu(v_1)}}{\hb\!-\!\psi_1}
\frac{1}{\E(\N Q_{\Ga_2})}\bigg\}\bigg|_{\hb=\frac{\al_{\mu(v_1)}-\al_i}{\d(e_1)}}\Bigg).
\end{split}\end{equation*}
The sum of the last factors above over all possibilities for $\Ga_2$, with $\Ga_1$ held fixed,
including the case $\Ga_2$ is empty
(when this factor is taken to be~1 for the equality to hold), is
$$\cZ_{n;\a}\big(\al_{\mu(v_1)},(\al_{\mu(v_1)}\!-\!\al_i)/\d(e_1),q\big),$$
as before.
Comparing with the recursion~\e_ref{recurdfn_e2} for~$\cZ_{n;\a}$, we conclude
\begin{equation*}\begin{split}
\sum_{\begin{subarray}{c}\Ga_c,\,|\Ga_c|>0\\ \mu(v_1)=j,\d(e_1)=d\end{subarray}}
\!\!\!\!\!\!\!\!\!q^{|\Ga_c|}\!\!\!
\int_{Q_{\Ga_c}}\!\!\psi_1^{-(b+1)}\frac{\E(\dot\V_{n;\a}^{(|\Ga_c|)})\ev_1^*\phi_i}{\E(\N Q_{\Ga_c})}
&=\fC_i^j(d)q^d \bigg(\frac{\al_j-\al_i}{d}\bigg)^{-(b+1)}
\!\!\cZ_{n;\a}\big(\al_j,(\al_j\!-\!\al_i)/d,q\big)\\
&=\Rs{\hb=\frac{\al_j-\al_i}{d}}\big\{\hb^{-(b+1)}\cZ_{n:\a}(\al_i,\hb,q)\big\}.
\end{split}\end{equation*}
Thus, by the recursion~\e_ref{recurdfn_e2} for $\cZ_{n;\a}$
and the Residue Theorem on~$S^2$,
\begin{equation*}\begin{split}
\sum_{\Ga_c,\,|\Ga_c|>0}
q^{|\Ga_c|}
\int_{Q_{\Ga_c}}\psi_1^{-(b+1)}\frac{\E(\dot\V_{n;\a}^{(|\Ga_c|)})\ev_1^*\phi_i}{\E(\N Q_{\Ga_c})}
&=-\Rs{\hb=0,\i}\big\{\hb^{-(b+1)}\cZ_{n;\a}(\al_i,\hb,q)\big\}\\
&=-\Rs{\hb=0}\big\{\hb^{-(b+1)}\cZ_{n;\a}(\al_i,\hb,q)\big\}+\de_{0b}\,.
\end{split}\end{equation*}
Combining this with \e_ref{ghsum_e} and \e_ref{V1contr_e}, we obtain
\begin{equation*}\begin{split}
&\sum_{\Ga,\,\d(v_{\min})>0}q^{|\Ga|}
\!\int_{Q_{\Ga}}
\frac{\E(\dot\V_{n;\a}^{(|\Ga|)})\ev_1^*\phi_i}{\hb\!-\!\psi_1}\Big|_{Q_{\Ga}}
\frac{1}{\E(\N Q_{\Ga})}\\
&\qquad
=\sum_{d=1}^{\i}\frac{q^{d}}{d!}\sum_{r=0}^{d-1}\hb^{-(r+1)}\sum_{b=0}^{d-1-r}
\Bigg(\bigg(\int_{\ov\cM_{0,2|d}}\frac{\E(\dot\V_{\a}^{(d)}(\al_i))\psi_1^r\psi_2^b}
{\prod\limits_{k\neq i}\E(\dot\V_1^{(d)}(\al_i\!-\!\al_k))}\bigg)
\Rs{\hb=0}\bigg\{\frac{(-1)^b}{\hb^{b+1}}\cZ_{n;\a}(\al_i,\hb,q)\bigg\}\Bigg).
\end{split}\end{equation*}
This concludes the proof of Proposition~\ref{cZrec_prp}.\\

\noindent
In the case of products of projective spaces and concavex sheaves~\e_ref{gensheaf_e},
we need analogues of~\e_ref{cSbe_e} and~\e_ref{V0prdfn_e} for every pair of tuples
$$\bfd\equiv(d_1,\ldots,d_p)\in(\Z^{\ge0})^p-0, \qquad
\be=(\be_1,\ldots,\be_p)\!\in\!H^2_{\T}.$$
Thus, we define sheaves  $\cS_1^*,\ldots,\cS_p^*$ over the universal curve
$\cU\!\lra\!\ov\cM_{0,2||\bfd|}$ by
$$\cS_1^*\!\equiv\!\cO_{\cU}(\si_1\!+\!\ldots\!+\!\si_{d_1}),\,
\cS_2^*\!\equiv\!\cO_{\cU}(\si_{d_1+1}\!+\!\ldots\!+\!\si_{d_1+d_2}),
\,\ldots \lra \cU$$
and denote by $\cS_i^*(\be_i)$, with $i\!=\!1,\ldots,p$, the sheaves such that
$$\E\big(\cS_i^*(\be_i)\big)=\be_i\!\times\!1+1\!\times\!e(\cS_i^*)
\in H_{\T}^*(\cU)=H_{\T}^*\otimes H^*(\cU).$$
Similarly to \e_ref{V0prdfn_e}, let
\begin{equation*}\begin{split}
\dot\V_{\a}^{(\bfd)}(\be)
= &\bigoplus_{a_{k;1}\ge0} R^0\pi_*\big(\cS_1^*(\be_1)^{a_{k;1}}\!\otimes\!\ldots\!
                                    \otimes\!\cS_1^*(\be_p)^{a_{k;p}}(-\!\si_1)\big)\\
&\oplus \bigoplus_{a_{k;1}<0} R^1\pi_*\big(\cS_1^*(\be_1)^{a_{k;1}}\!\otimes\!\ldots\!
     \otimes\!\cS_1^*(\be_p)^{a_{k;p}}(-\!\si_1)\big)
\lra \ov\cM_{0,2||\bfd|}.
\end{split}\end{equation*}
The fixed points of the $\T$-action on $\P^{n_1-1}\!\times\!\ldots\!\times\!\P^{n_p-1}$ are
$$P_{i_1\ldots i_p}\equiv P_{i_1}\times\ldots\times P_{i_p}\,,
\qquad i_s\in[n_s];$$
thus, the function $\mu$ on vertices now takes values in the tuples $(i_1,\ldots,i_p)$.
The function~$\d$ on vertices now takes values in $(\Z^{\ge0})^p$,
with the space $\ov\cM_{0,2|\d(v)}/\bS_{\d(v)}$ above replaced by
$$\ov\cM_{0,2|\d_1(v)+\ldots+\d_p(v)}\big/
\bS_{\d_1(v)}\!\times\!\ldots\!\times\!\bS_{\d_p(v)},$$
in light of~\e_ref{curvquot_e}.
The $\T$-fixed curves are the lines between the points $P_{i_1\ldots i_p}$
and $P_{j_1\ldots j_p}$ such~that
$$\big|\{s\!\in\![p]\!:~i_s\!\neq\!j_s\}\big|=1;$$
thus, the vertices of any edge now differ by precisely one of the indices $(i_1,\ldots,i_p)$,
with the $\om$-classes in~\e_ref{psiform_e}  described by the difference in the weights
of this index.
The strands with $\d(v_{\min})\!=\!0$ now give rise to a triple sum,
with the summation index $s\!\in\![p]$ on the outer sum indicating which of the indices
$(i_1,\ldots,i_p)$ changes.
The computation of the contribution from the strands with $\d(v_{\min})\!>\!0$
proceeds exactly as above, but the denominator in
the integrand for $\ov\cM_{0,2|d_0}$ above is replaced by the product of
factors corresponding to each of the $p$~factors.
This results in a similar formula  for the secondary coefficients
$\cZ_{i_1\ldots i_p}^r$ in~\e_ref{recurdfn_e2}:
\BE{cZrelext_e}\begin{split}
&\sum_{(d_1,\ldots,d_p)\in(\Z^{\ge0})-0}^{\i}
\hspace{-.4in}\cZ_{i_1\ldots i_p}^r(d_1,\ldots,d_p)
q_1^{d_1}\ldots q_p^{d_p}
=\sum_{\bfd\in(\Z^{\ge0})-0}\frac{q_1^{d_1}\ldots q_p^{d_p}}{d_1!\ldots d_p!}
\sum_{b=0}^{|\bfd|+r}\\
&\hspace{.1in}
\Bigg(\!\!\bigg(\int_{\ov\cM_{0,2||\bfd|}}
\frac{\E(\dot\V_{\a}^{(\bfd)}(\al_{i_1},\ldots,\al_{i_p}))\psi_1^{-r-1}\psi_2^b}
{\prod\limits_{s=1}^p\prod\limits_{k\neq i_s}\E(\dot\V_{e_s}^{(d_s)}(\al_{s;i_s}\!-\!\al_{s;k}))}\bigg)
\Rs{\hb=0}\bigg\{\frac{(-1)^b}{\hb^{b+1}}\cZ_{n;\a}(\al_{i_1},\ldots,\al_{i_p},
\hb,q_1,\ldots,q_p)\bigg\}\!\!\Bigg),
\end{split}\EE
whenever $r\!\in\!\Z^-$ and $i_s\!\in\![n_s]$, if
$e_s\!\in\!(\Z^+)^p$ is the $s$-th coordinated vector.

\section{Polynomiality for stable quotients}
\label{SPC_sec}

\noindent
In this section, we adopt the argument in \cite[Section~30.2]{MirSym}, showing
that the equivariant version of Givental's $J$-function
satisfies the self-polynomiality condition of Definition~\ref{SPC_dfn},
to show that the equivariant stable quotients analogue of Givental's
$J$-function, the power series~$\cZ_{n;\a}$ defined by~\e_ref{cZdfn_e},
also satisfies the self-polynomiality condition.
Proposition~\ref{cZSPC_prp} is an immediate consequence of Lemma~\ref{PhiZstr_lmm}
below, which provides a geometric description of the power series~$\Phi_{\cZ_{n;\a}}$.

\begin{prp}\label{cZSPC_prp}
If $l\!\in\!\Z^{\ge0}$, $n\!\in\!\Z^+$, and $\a\!\in\!(\Z^*)^l$,
the power series $\cZ_{n;\a}(\x,\hb,q)$ satisfies the self-polynomiality
condition.
\end{prp}

\noindent
The proof involves applying the classical localization theorem~\cite{ABo} with $(n\!+\!1)$-torus
$$\wt\T\equiv\C^*\times\T,$$
where $\T=(\C^*)^n$ as before.
We denote the weight of the standard action of the one-torus $\C^*$ on $\C$ by~$\hb$.
Thus, by Section~\ref{equivsetup_sec},
$$H_{\C^*}^*\approx\Q[\hb], \quad H_{\wt\T}^*\approx\Q[\hb,\al_1,\ldots,\al_n]
\qquad\Lra\qquad \H_{\wt\T}^*\approx\Q_{\al}(\hb).$$
Throughout this section, $V\!=\!\C\!\oplus\!\C$ denotes the representation
of $\C^*$ with the weights $0$ and~$-\hb$.
The induced action on $\P V$ has two fixed points:
$$q_1\equiv[1,0], \qquad q_2\equiv[0,1].$$
With $\ga_1\!\lra\!\P V$ denoting the tautological line bundle,
\BE{hbaract_e}
\E(\ga_1^*)\big|_{q_1}=0, \quad \E(\ga_1^*)\big|_{q_2}=-\hb,
\quad \E(T_{q_1}\P V)=\hb, \quad \E(T_{q_2}\P V)=-\hb;\EE
this follows from our definition of the weights in Section~\ref{equivsetup_sec}.\\

\noindent
For each $d\!\in\!\Z^{\ge0}$, the action of $\wt\T$ on $\C^n\!\otimes\!\Sym^dV^*$ induces
an action on
$$\ov\X_d\equiv\P\big(\C^n\!\otimes\!\Sym^dV^*\big).$$
It has $(d\!+\!1)n$ fixed points:
$$P_i(r)\equiv \big[\ti{P}_i\otimes u^{d-r}v^r\big], \qquad
i\in[n],~r\in\{0\}\!\cup\![d],$$
if $(u,v)$ are the standard coordinates on $V$ and $\ti{P}_i\!\in\!\C^n$ is
the $i$-th coordinate vector (so that $[\ti{P}_i]\!=\!P_i\!\in\!\P^{n-1}$).
Let
$$\Om\equiv \E(\ga^*)\in H_{\wt\T}^*\big(\ov\X_d\big)$$
denote the equivariant hyperplane class.\\

\noindent
For all $i\!\in\![n]$ and $r\in\{0\}\!\cup\![d]$,
\begin{equation}\label{Xrestr_e}
\Om|_{P_i(r)}=\al_i\!+\!r\hb, \qquad
\E(T_{P_i(r)}\ov\X_d)=\Bigg\{\underset{(s,k)\neq(r,i)}{\prod_{s=0}^d\prod_{k=1}^n}
(\Om\!-\!\al_k\!-\!s\hb)\Bigg\}\bigg|_{\Om=\al_i+r\hb}.~\footnotemark
\end{equation}
\footnotetext{The weight (i.e.~negative first Chern class) of the $\wt\T$-action
on the line $P_i(r)\!\subset\!\C^n\!\otimes\!\Sym^dV^*$ is $\al_i\!+\!r\hb$.
The tangent bundle of $\ov\X_d$ at $P_i(r)$ is the direct sum of the lines
$P_i(r)^*\!\otimes\!P_k(s)$ with $(k,s)\!\neq\!(i,r)$.}
Since
\begin{gather*}
B\ov\X_d=\P\big(B(\C^n\!\otimes\!\Sym^dV^*)\big)\lra B\wt\T \qquad\hbox{and}\\
c\big(B(\C^n\!\otimes\!\Sym^dV^*)\big)
=\prod_{s=0}^d\prod_{k=1}^n\big(1-(\al_k\!+\!s\hb)\big)
\in H^*\big(B\wt\T),\footnotemark
\end{gather*}
\footnotetext{The vector space $\C^n\!\otimes\!\Sym^dV^*$ is the direct sum of
the one-dimensional representations $P_k(s)$ of~$\wt\T$.}
the $\wt\T$-equivariant cohomology of $\ov\X_d$ is given~by
\begin{equation*}\begin{split}
H_{\ti\T}^*\big(\ov\X_d\big)&\equiv H^*\big(B\ov\X_d\big)
=H^*\big(B\wt\T\big)\big[\Om\big]\Big/
 \prod_{s=0}^d\prod_{k=1}^n\big(\Om-(\al_k\!+\!s\hb)\big)\\
&\approx \Q\big[\Om,\hb,\al_1,\ldots,\al_n\big]\Big/
 \prod_{s=0}^d\prod_{k=1}^n\big(\Om-\al_k-\!s\hb\big)\\
&\subset \Q_{\al}[\hb,\Om]\Big/
 \prod_{s=0}^d\prod_{k=1}^n\big(\Om-\al_k-s\hb\big).
\end{split}\end{equation*}
In particular, every element of $H_{\ti\T}^*(\ov\X_d)$ is a polynomial in
$\Om$ with coefficients in $\Q_{\al}[\hb]$ of degree at most $(d\!+\!1)n\!-\!1$.\\

\noindent
By \cite[Lemma~2.6]{LLY}, there is a natural $\wt\T$-equivariant morphism
$$\Th\!:\ov\M_{0,m}\big(\P V\!\times\!\P^{n-1},(1,d)\big)\lra \ov\X_d.$$
A general element of $b$ of $\ov\M_{0,m}\big(\P V\!\times\!\P^{n-1},(1,d)\big)$
determines a morphism
$$(f,g)\!:\P^1\lra(\P V,\P^{n-1}),$$
up to an automorphism  of the domain~$\P^1$.
Thus, the morphism
$$g\circ f^{-1}\!: \P V\lra \P^{n-1}$$
is well-defined and determines an element $\Th(b)\!\in\!\ov\X_d$.
Let
\begin{alignat}{1}
\X_d&=\big\{b\!\in\!\ov\M_{0,2}\big(\P V\!\times\!\P^{n-1},(1,d)\big)\!:
\ev_1(b)\!\in\!q_1\!\times\!\P^{n-1},~\ev_2(b)\!\in\!q_2\!\times\!\P^{n-1}\big\},\notag\\
\label{Xdfn_e}
\X_d'&=\big\{b'\!\in\!\ov{Q}_{0,2}\big(\P V\!\times\!\P^{n-1},(1,d)\big)\!:
\ev_1(b')\!\in\!q_1\!\times\!\P^{n-1},~\ev_2(b')\!\in\!q_2\!\times\!\P^{n-1}\big\}.
\end{alignat}
Since the morphism to $\P^1$ corresponding to any element of $b'\!\in\!\X_d'$
takes the two marked points to $q_1$ and~$q_2$,
it is not constant.
Thus, the restriction of the morphism~$\Th$ to~$\X_d$ is constant along
the fibers of the natural surjective morphism
$c\!:\X_d\!\lra\!\X_d'$.\footnote{For a stable map $b$, $\Th(b)$ depends only on the restriction
of $b$ to the irreducible component $\cC_{b;1}$ of its domain~$\cC_b$ on which
the degree of the map to~$\P^1$ is not zero, the nodes of~$\cC_{b;1}$, and
the degrees of the restrictions of~$b$ to the connected components of $\cC_b\!-\!\cC_{b;1}$.
In contrast, $c(b)$ depends on the restriction of $b$ to the minimal connected union (chain)
of irreducible components~$\cC_b'$ of its domain which contains the two marked points,
the nodes of~$\cC_b'$, and
the degrees of the restrictions of~$b$ to the connected components of $\cC_b\!-\!\cC_b'$.
Whenever $b\!\in\!\X_d$, $\cC_{b;1}\!\subset\!\cC_b'$.
Thus, the restriction of $\Th$ to $\X_d$ contracts everything that the restriction
of~$c$ contracts.}
It follows that the restriction of~$\Th$ to~$\X_d$ descends via~$c$ to a morphism
$$\th\!=\!\th_d\!:\X_d'\lra  \ov\X_d.$$\\

\noindent
For $d\!>\!0$, there is also a natural forgetful morphism
$$F\!: \ov{Q}_{0,2}\big(\P V\!\times\!\P^{n-1},(1,d)\big)\lra
\ov{Q}_{0,2}\big(\P^{n-1},d\big),$$
which drops the first sheaf in the pair and contracts one component of the domain if necessary.
Similarly to~\e_ref{Vprdfn_e}, for each $d\!\in\!\Z^+$ let
$$\V_{n;\a}^{(d)}=\bigoplus_{a_k>0}R^0\pi_*\big(\cS^{*a_k}\big)
 \oplus \bigoplus_{a_k<0}R^1\pi_*\big(\cS^{*a_k}\big)
 \lra \ov{Q}_{0,2}(\Pn,d).$$
From the usual short exact sequence for the restriction along $\si_1$, we find that
\BE{VvsVpr_e}
\E\big(\V_{n;\a}^{(d)}\big)
=\lr\a \ev_1^*\x^{\ell(\a)}\E\big(\dot\V_{n;\a}^{(d)}\big)
\in H_{\T}^*\big(\ov{Q}_{0,2}(\Pn,d)\big).\EE
In the case $d\!=\!0$, we set
$$F^*\E(\V_{n;\a}^{(0)})=\lr\a\ev_1^*\big(1\!\times\!\x^{\ell(\a)}\big)\in
H^*\big( \ov{Q}_{0,2}(\P V\!\times\!\P^{n-1},(1,0))\big);$$
this is used in Lemma~\ref{PhiZstr_lmm} below.

\begin{lmm}\label{PhiZstr_lmm}
If $l\!\in\!\Z^{\ge0}$, $n\!\in\!\Z^+$, and $\a\!\in\!(\Z^*)^l$, then
\BE{PhiZstr_e}\Phi_{\cZ_{n;\a}}(\hb,z,q) =
\sum_{d=0}^{\i}q^d\!\!
\int_{\X_d'}\!\!\!\ne^{(\th^*\Om)z}F^*\E(\V_{n;\a}^{(d)})
\in H_{\wt\T}^*\big[\big[z,q\big]\big]\subset\Q_{\al}[\hb]\big[\big[z,q\big]\big].\EE
\end{lmm}

\noindent
We prove Lemma~\ref{PhiZstr_lmm} in the remainder of this section by applying the localization
theorem of~\cite{ABo} to the $\wt\T$-action on~$\X_d'$.
We show that each fixed locus of the $\wt\T$-action on $\X_d'$ contributing
to the right-hand side of~\e_ref{PhiZstr_e} corresponds to a pair $(\Ga_1,\Ga_2)$
of decorated strands as in~\e_ref{decortgraphdfn_e}, with $\Ga_1$ and $\Ga_2$ contributing to
$\cZ_{n;\a}(\al_i,\hb,q\ne^{\hb z})$ and
$\cZ_{n;\a}(\al_i,-\hb,q)$, respectively, for some $i\!\in\![n]$.\\

\noindent
Similarly to Section~\ref{SQlocal_sec}, the fixed loci of the $\wt\T$-action on
$\ov{Q}_{0,2}(\P V\!\times\!\P^{n-1},(d',d))$
correspond to decorated strands $\Ga$ with 2~marked points at the opposite ends.
The map~$\d$ should now take values in pairs of nonnegative integers, indicating the degrees of
the two subsheaves.
The map~$\mu$ should similarly take values in the pairs~$(i,j)$
with $i\!\in\![2]$ and $j\!\in\![n]$,
indicating the fixed point~$(q_i,P_j)$ to which the vertex is mapped.
The $\mu$-values on consecutive vertices must differ by precisely one of the two components.\\

\noindent
The situation for the  $\wt\T$-action on
$$\X_d'\subset\ov{Q}_{0,2}\big(\P V\!\times\!\P^{n-1},(1,d)\big)$$
is simpler, however.
There is a unique edge of positive $\P V$-degree; we draw it as a thick line
in Figure~\ref{Xlocus_fig}.
The first component of the value of~$\d$ on all other edges
and on all vertices must be~0; so we drop~it.
The first component of the value of~$\mu$ on the vertices changes only when the thick
edge is crossed.
Thus, we drop the first components of the vertex labels as well,
with the convention that these components are~1 on the left side of the thick
edge and~2 on the right.
In particular, the vertices to the left of the thick edge (including the left endpoint)
lie in $q_1\!\times\!\P^{n-1}$ and the vertices to its right lie in $q_2\!\times\!\P^{n-1}$.
Thus, by~\e_ref{Xdfn_e}, the marked point~1 is attached to a vertex to the left of
the thick edge and the marked point~2 is attached to a vertex to the right.
Finally, the remaining, second component of~$\mu$ takes the same value $i\!\in\![n]$
on the two vertices of the thick edge.\\

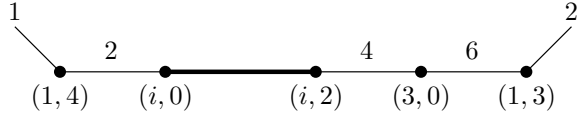
\begin{figure}
\begin{pspicture}(0.3,-.8)(10,1)
\psset{unit=.4cm}
\psline[linewidth=.15](17,0)(22,0)
\pscircle*(17,0){.2}\pscircle*(22,0){.2}
\psline[linewidth=.04](17,0)(13.5,0)\pscircle*(13.5,0){.2}
\psline[linewidth=.04](13.5,0)(12,1.5)\rput(12,2){\smsize{$1$}}
\rput(15.2,.7){\smsize{2}}\rput(13.5,-.85){\smsize{$(1,4)$}}
\rput(17,-.85){\smsize{$(i,0)$}}
\psline[linewidth=.04](22,0)(25.5,0)\pscircle*(25.5,0){.2}
\psline[linewidth=.04](25.5,0)(29,0)\pscircle*(29,0){.2}
\psline[linewidth=.04](29,0)(30.5,1.5)\rput(30.5,2){\smsize{$2$}}
\rput(22,-.85){\smsize{$(i,2)$}}\rput(25.5,-.85){\smsize{$(3,0)$}}
\rput(29,-.85){\smsize{$(1,3)$}}
\rput(23.7,.7){\smsize{4}}\rput(27.2,.7){\smsize{6}}
\end{pspicture}
\caption{A strand representing a fixed locus in $\X_d'$; $i\!\neq\!1,3$}
\label{Xlocus_fig}
\end{figure}

\noindent
Let $\cA_i$ denote the set of strands as above so that the $\mu$-value on
the two endpoints of the thick edge is labeled~$i$; see Figure~\ref{Xlocus_fig}.
We break each strand $\Ga\!\in\!\cA_i$ into three sub-strands:
\begin{enumerate}[label=(\roman*)]
\item $\Ga_1$ consisting of all vertices of $\Ga$ to the left of the thick edge,
including its left vertex~$v_1$ with its $\d$-value,
but in the opposite order, and a new marked point attached to~$v_1$;
\item $\Ga_0$ consisting of the thick edge~$e_0$, its two vertices $v_1$ and $v_2$,
with $\d$-values set to~0, and
new marked points $1$ and $2$ attached to $v_1$ and $v_2$, respectively;
\item $\Ga_2$ consisting of all vertices to the right of the thick edge,
including its right vertex~$v_2$ with its $\d$-value, and a new marked point attached to~$v_2$;
\end{enumerate}
see Figure~\ref{Xsplit_fig}.
From~\e_ref{Zlocus_e}, we then obtain a splitting of
the fixed locus in $\X_d'$ corresponding to~$\Ga$:
\BE{Xlocus_e1}
Q_{\Ga}\approx Q_{\Ga_1}\times Q_{\Ga_0}\times Q_{\Ga_2}
\subset\ov{Q}_{0,2}(\P^{n-1},|\Ga_1|)\times
\ov{Q}_{0,2}(\P V,1)\times\ov{Q}_{0,2}(\P^{n-1},|\Ga_2|).\EE
The exceptional cases are $|\Ga_1|\!=\!0$ and $|\Ga_2|\!=\!0$;
the above isomorphism then holds with the corresponding component replaced by a point.\\

\begin{figure}
\begin{pspicture}(0,-.8)(10,1)
\psset{unit=.4cm}
\psline[linewidth=.04](7.5,0)(11,0)\rput(9.2,.7){\smsize{2}}
\pscircle*(7.5,0){.2}\rput(7.5,-.85){\smsize{$(1,4)$}}
\pscircle*(11,0){.2}\rput(11,-.85){\smsize{$(i,0)$}}
\psline[linewidth=.04](7.5,0)(6,1.5)\rput(6,2){\smsize{$\bf 2$}}
\psline[linewidth=.04](11,0)(12.5,1.5)\rput(12.5,2){\smsize{$\bf 1$}}
\rput(9.3,-2){$\Ga_1$}
\psline[linewidth=.04](17,0)(15.5,1.5)\rput(15.5,2){\smsize{$\bf 1$}}
\psline[linewidth=.04](22,0)(23.5,1.5)\rput(23.5,2){\smsize{$\bf 2$}}
\psline[linewidth=.15](17,0)(22,0)
\rput(17,-.85){\smsize{$(i,0)$}}\rput(22,-.85){\smsize{$(i,0)$}}
\rput(19.5,-2){$\Ga_0$}
\psline[linewidth=.04](27,0)(30.5,0)\pscircle*(30.5,0){.2}
\psline[linewidth=.04](30.5,0)(34,0)\pscircle*(34,0){.2}
\pscircle*(27,0){.2}\rput(27,-.85){\smsize{$(i,2)$}}
\rput(30.5,-.85){\smsize{$(3,0)$}}\rput(34,-.85){\smsize{$(1,3)$}}
\psline[linewidth=.04](27,0)(25.5,1.5)\rput(25.5,2){\smsize{$\bf 1$}}
\psline[linewidth=.04](34,0)(35.5,1.5)\rput(35.5,2){\smsize{$\bf 2$}}
\rput(28.7,.7){\smsize{4}}\rput(32.2,.7){\smsize{6}}
\rput(30.5,-2){$\Ga_2$}
\end{pspicture}
\caption{The three sub-strands of the strand in Figure~\ref{Xlocus_fig}}
\label{Xsplit_fig}
\end{figure}
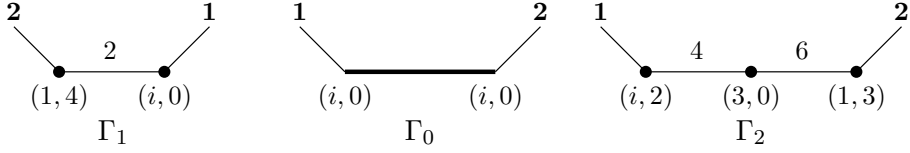

\noindent
Let $\pi_1$, $\pi_0$, and $\pi_2$ denote the three component projection maps in~\e_ref{Xlocus_e1}.
By~\e_ref{VvsVpr_e}, \e_ref{cVform_e}, and~\e_ref{NZform_e},
\BE{bndsplit3_e}\begin{split}
F^*\E\big(\V_{n;\a}^{(|\Ga|)}\big)\big|_{Q_{\Ga}}&=
\lr\a\al_i^{\ell(\a)}\cdot
  \pi_1^*\E\big(\dot\V_{n;\a}^{(|\Ga_1|)}\big)\cdot\pi_2^*\E(\dot\V_{n;\a}^{(|\Ga_2|)}\big)\,,\\
\frac{\E(\N Q_{\Ga})}{\E(T_{P_i}\Pn)}&=
\pi_1^*\bigg(\frac{\E(\N Q_{\Ga_1})}{\E(T_{P_i}\P^{n-1})}\bigg)
\cdot\pi_2^*\bigg(\frac{\E(\N Q_{\Ga_2})}{\E(T_{P_i}\P^{n-1})}\bigg)
\cdot \big(\om_{e_0;v_1}-\pi_1^*\psi_1\big) \big(\om_{e_0;v_2}-\pi_2^*\psi_1\big).
\end{split}\EE
Since $Q_{\Ga_0}$ consists of a degree~1 map,
by the last two identities in~\e_ref{hbaract_e}
\BE{omrestr_e} \om_{e_0;v_1}=\hb, \qquad \om_{e_0;v_2}=-\hb.\EE
The morphism $\th$ takes the locus $Q_{\Ga}$ to a fixed point $P_k(r)\!\in\!\ov\X_d$.
It is immediate that $k\!=\!i$.
By continuity considerations, $r\!=\!|\Ga_1|$.
Thus, by the first identity in \e_ref{Xrestr_e},
\BE{thOmrestr_e}\th^*\Om\big|_{Q_{\Ga}}=\al_i+|\Ga_1|\hb.\EE
Combining \e_ref{bndsplit3_e}-\e_ref{thOmrestr_e}, we obtain
\begin{equation}\label{Xbndlsplit_e5}\begin{split}
q^{|\Ga|}\int_{Q_{\Ga}}\!\!\!
\frac{\ne^{(\th^*\Om)z}F^*\E(\V_{n;\a}^{(|\Ga|)})|_{Q_{\Ga}}}{\E(\N Q_{\Ga})}
=\frac{\lr\a\al_i^{\ell(\a)}\ne^{\al_iz}}{\prod\limits_{k\neq i}(\al_i\!-\!\al_k)}
&\Bigg\{\ne^{|\Ga_1|\hb z}q^{|\Ga_1|}\!\! \int_{Q_{\Ga_1}} \!\!\!\!\!\!
\frac{\E(\dot\V_{n;\a}^{(|\Ga_1|)})\ev_1^*\phi_i}{\hb\!-\!\psi_1}
\Big|_{Q_{\Ga_1}}\frac{1}{\E(\N Q_{\Ga_1})}\Bigg\}\\
\times& \Bigg\{q^{|\Ga_2|}\int_{Q_{\Ga_2}} \!\!\!\!\!
\frac{\E(\dot\V_{n;\a}^{(|\Ga_2|)})\ev_1^*\phi_i}{(-\hb)\!-\!\psi_1}
\Big|_{Q_{\Ga_2}}\frac{1}{\E(\N Q_{\Ga_2})}\Bigg\}.
\end{split}\end{equation}
This identity remains valid with $|\Ga_1|\!=\!0$ and/or $|\Ga_2|\!=\!0$
if we set the corresponding integral to~1.\\

\noindent
We now sum up \e_ref{Xbndlsplit_e5} over all $\Ga\!\in\!\cA_i$.
This is the same as summing over all pairs $(\Ga_1,\Ga_2)$ of decorated strands such that
\begin{enumerate}[label=(\arabic*)]
\item $\Ga_1$ is a 2-point strand of degree $d_1\!\ge\!0$ such that the marked point~1
is attached to a vertex labeled~$i$;
\item $\Ga_2$ is a 2-point strand of degree $d_2\!\ge\!0$ such that the marked point~1
is attached to a vertex labeled~$i$.
\end{enumerate}
By the localization formula~\e_ref{ABothm_e},
\begin{equation}\label{Xbndlsplit_e7}\begin{split}
&1+\sum_{\Ga_1}(q\ne^{\hb z})^{|\Ga_1|}\Bigg\{\! \int_{Q_{\Ga_1}} \!\!\!\!\!\!
\frac{\E(\dot\V_{n;\a}^{(|\Ga_1|)})\ev_1^*\phi_i}{\hb\!-\!\psi_1}
\Big|_{Q_{\Ga_1}}\frac{1}{\E(\N Q_{\Ga_1})}\Bigg\}\\
&\hspace{1in}=1+\sum_{d=1}^{\i}(q\ne^{\hb z})^d
\int_{\ov{Q}_{0,2}(\P^{n-1},d)}\frac{\E(\dot\V_{n;\a}^{(d)})\ev_1^*\phi_i}{\hb\!-\!\psi_1}
=\cZ_{n;\a}\big(\al_i,\hb,q\ne^{\hb z}\big);\\
&1+\sum_{\Ga_2}q^{|\Ga_2|}\Bigg\{\int_{Q_{\Ga_2}} \!\!\!\!\!\!
\frac{\E(\dot\V_{n;\a}^{(|\Ga_2|)})\ev_1^*\phi_i}{(-\hb)\!-\!\psi_1}
\Big|_{Q_{\Ga_2}}\frac{1}{\E(\N Q_{\Ga_2})}\Bigg\}\\
&\hspace{1in}=1+\sum_{d=0}^{\i}q^d
\int_{\ov{Q}_{0,2}(\P^{n-1},d)}\!\!\!
\frac{\E(\dot\V_{n;\a}^{(d)})\ev_1^*\phi_i}{(-\hb)\!-\!\psi_1}
= \cZ_{n;\a}\big(\al_i,-\hb,q\big).
\end{split}\end{equation}
Finally, by~\e_ref{ABothm_e}, \e_ref{Xbndlsplit_e5}, and~\e_ref{Xbndlsplit_e7},
\begin{equation*}\begin{split}
\sum_{d=0}^{\i}q^d\!\!
\int_{\X_d'}\ne^{(\th^*\Om)z}F^*\E(\V_{n;\a}^{(d)})
&=\sum_{i=1}^n\frac{\lr\a\al_i^{\ell(\a)}\ne^{\al_iz}}{\prod\limits_{k\neq i}(\al_i\!-\!\al_k)}
\cZ_{n;\a}\big(\al_i,\hb,q\ne^{\hb z}\big)\cZ_{n;\a}\big(\al_i,-\hb,q\big)\\
&=\Phi_{\cZ_{n;\a}}(\hb,z,q),
\end{split}\end{equation*}
as claimed in~\e_ref{PhiZstr_e}.\\

\noindent
In the case of products of projective spaces and concavex sheaves~\e_ref{gensheaf_e},
the spaces
$$\ov{Q}_{0,2}(\P V\!\times\!\Pn,(1,d)) \qquad\hbox{and}\qquad
\ov\X_d=\P\big(\C^n\!\otimes\!\Sym^dV^*\big) $$
are replaced~by
$$\ov{Q}_{0,2}\big(\P V\!\times\!\P^{n_1-1}\!\times\!\ldots\!\times\!\P^{n_p-1},
(1,d_1,\ldots,d_p)\big)  \quad\hbox{and}\quad
\P\big(\C^{n_1}\!\otimes\!\Sym^{d_1}V^*\big)\!\!\times\!\ldots\!\times\!
\P\big(\C^{n_p}\!\otimes\!\Sym^{d_p}V^*\big),$$
respectively.
Lemma~\ref{PhiZstr_lmm} then becomes
\begin{equation*}\begin{split}
&\Phi_{\cZ_{n_1,\ldots,n_p;\a}}(\hb,z_1,\ldots,z_p,q_1,\ldots,q_p)
=\sum_{d_1,\ldots,d_p\ge0}\hspace{-.15in}q_1^{d_1}\ldots q_p^{d_p} \!\!
\int_{\X_{d_1,\ldots,d_p}'}\hspace{-.4in}\ne^{(\th^*\Om_1)z_1+\ldots+(\th^*\Om_p)z_p}
\pi_1^*\E(\V_{n_1,\ldots,n_p;\a}^{(d_1,\ldots,d_p)}).
\end{split}\end{equation*}
The vertices of the thick edge in Figure~\ref{Xlocus_fig}
are now labeled by a tuple $(i_1,\ldots,i_p)$ with $i_s\!\in\![n_s]$,
as needed for the extension of~\e_ref{PhiZdfn_e} described at
the end of Section~\ref{algebra_sec}.
The relation~\e_ref{thOmrestr_e} becomes
$$\th^*\Om_s\big|_{Q_{\Ga}}=\al_{s;i_s}+|\Ga_1|_s\hb\,,$$
where $|\Ga_1|_s$ is the sum of the $s$-th components of the values of~$\d$
on the vertices and edges of~$\Ga_1$ (corresponding to the degree of the maps to~$\P^{n_s-1}$).
Otherwise, the proof is identical.

\section{Proof of Theorems~\ref{equiv_thm} and~\ref{equiv0_thm}}
\label{mainpf_sec}

\noindent
This section concludes the proof of Theorem~\ref{equiv_thm} stated in Section~\ref{equivthm_sec}.
Sections~\ref{algebra_sec}-\ref{SPC_sec} reduce this theorem to
conditions on the power series~$\cY_{n;\a}$ defined in~\e_ref{cYdfn_e}; see Lemma~\ref{cYcZcomp_lmm}.
Based on qualitative, primarily algebraic, considerations,
we show in the proof of Proposition~\ref{cYcZrel_prp}
that this power series does indeed satisfy these conditions
and thus establish Theorem~\ref{equiv_thm}.
The only geometric considerations entering the proof of Proposition~\ref{cYcZrel_prp} concern
moduli spaces of stable curves $\ov\cM_{0,2|d}$, not
moduli spaces of stable quotients $\ov{Q}_{0,2}(\Pn,d)$.
We conclude this section by showing that these conditions  on~$\cY_{n;\a}$
determine certain integrals on $\ov\cM_{0,2|d}$ and
finish the proof of Theorem~\ref{equiv0_thm} stated in Section~\ref{equivthm_sec}.

\begin{crl}\label{MirSym_crl}
Let $l\!\in\!\Z^{\ge0}$, $n\!\in\!\Z^+$ and $\a\!\in\!(\Z^*)^l$.
If $|\a|\!\le\!n\!-\!2$,
$$\cZ_{n;\a}(\x,\hb,q)=\cY_{n;\a}(\x,\hb,q)
\in H_{\T}^*(\Pn)\big[\big[\hb^{-1},q\big]\big].$$
\end{crl}

\begin{proof}
Both sides of this identity are $\fC$-recursive and satisfy the self-polynomiality condition
(no matter what $n$ and $\a$ are);
see Lemma~\ref{cYrec_lmm} and Propositions~\ref{cZrec_prp} and~\ref{cZSPC_prp}.
By~\e_ref{cYdfn_e},
$$\cY_{n;\a}(\x,\hb,q)\cong1 \quad\mod\hb^{-2}\,,$$
whenever $|\a|\!\le\!n\!-\!2$.
If in addition $d\!\in\!\Z^+$,
$$\dim\ov{Q}_{0,2}(\Pn,d)-\rk\,\dot\V_{n;\a}^{(d)}
=(n\!-\!|\a|)d+(n\!-\!2)>n\!-\!1=\dim\Pn.$$
Thus,
$$\cZ_{n;\a}(\x,\hb,q)\cong1 \quad\mod\hb^{-2}\,,$$
whenever $|\a|\!\le\!n\!-\!2$.
The claim now follows from Proposition~\ref{uniqueness_prp}.
\end{proof}

\begin{lmm}\label{cYcZcomp_lmm}
If $l\!\in\!\Z^{\ge0}$, $n\!\in\!\Z^+$, and $\a\!\in\!(\Z^*)^l$ are such that $|\a|\!\le\!n$,
then
\BE{equivMS_e2}\cZ_{n;\a}(\x,\hb,q)=\frac{\cY_{n;\a}(\x,\hb,q)}{I_{n;\a}(q)}
\in H_{\T}^*(\Pn)\big[\big[\hb^{-1},q\big]\big]\EE
if and only if
\BE{Yrec_e}\begin{split}
&\Rs{\hb=0}\big\{\hb^r\cY_{n;\a}(\al_i,\hb,q)\big\}\\
&\qquad=\sum_{d=1}^{\i}\frac{q^{d}}{d!}\sum_{b=0}^{d-1-r}\!\!
\Bigg(\bigg(\int_{\ov\cM_{0,2|d}}\frac{\E(\dot\V_{\a}^{(d)}(\al_i))\psi_1^r\psi_2^b}
{\prod\limits_{k\neq i}\E(\dot\V_1^{(d)}(\al_i\!-\!\al_k))}\bigg)
\Rs{\hb=0}\bigg\{\frac{(-1)^b}{\hb^{b+1}}\cY_{n;\a}(\al_i,\hb,q)\bigg\}\Bigg)
\end{split}\EE
for all $i\!\in\![n]$ and $r\!\in\!\Z^{\ge0}$.
\end{lmm}

\begin{proof}
Since both sides of~\e_ref{equivMS_e2} are $\fC$-recursive
with the {\it same} collection~\e_ref{Cdfn_e} of the primary coefficients
(see Lemma~\ref{cYrec_lmm} and Proposition~\ref{cZrec_prp}) and
have the same $q^0$-coefficients, \e_ref{equivMS_e2} holds if and only if
the secondary coefficients
$$\frac{1}{I_{n;\a}(q)}\sum_{d=0}^{\i}\cY_i^r(d)q^d
\qquad\hbox{and}\qquad
\sum_{d=0}^{\i}\cZ_i^r(d)q^d,$$
instead of~$\F_i^r(d)$, in the recursions~\e_ref{recurdfn_e2}
for $\cY_{n;\a}/I_{n;\a}$ and~$\cZ_{n;\a}$ are the same
(this would make the two recursions the same).
Since Proposition~\ref{cZrec_prp} describes the coefficients~$\cZ_i^r(d)$ recursively
on~$d$, \e_ref{equivMS_e2} holds if and only if
the coefficients~$\cY_i^r(d)$ satisfy the same description.
By Lemma~\ref{cYrec_lmm} and Proposition~\ref{cZrec_prp},
this is the case if and only if \e_ref{Yrec_e} holds
($r$ in Lemma~\ref{cYrec_lmm} and Proposition~\ref{cZrec_prp} corresponds
to $-r\!-\!1$ in the notation of~\e_ref{Yrec_e}).
\end{proof}

\begin{prp}\label{cYcZrel_prp}
If $l\!\in\!\Z^{\ge0}$, $n\!\in\!\Z^+$, and $\a\!\in\!(\Z^*)^l$, then
\BE{cYcZrel_e}\begin{split}
&\Rs{\hb=0}\big\{\hb^r\cY_{n;\a}(\al_i,\hb,q)\big\}\\
&\qquad=\sum_{d=1}^{\i}\frac{q^d}{d!}\sum_{b=0}^{d-1-r}\!\!
\Bigg(\bigg(\int_{\ov\cM_{0,2|d}}\frac{\E(\dot\V_{\a}^{(d)}(\al_i))\psi_1^r\psi_2^b}
{\prod\limits_{k\neq i}\E(\dot\V_1^{(d)}(\al_i\!-\!\al_k))}\bigg)
\Rs{\hb=0}\bigg\{\frac{(-1)^b}{\hb^{b+1}}\cY_{n;\a}(\al_i,\hb,q)\bigg\}\Bigg)
\end{split}\EE
for all  $i\!\in\![n]$ and $r\!\in\!\Z^{\ge0}$.
\end{prp}

\begin{proof} Let $i\!\in\![n]$ be fixed throughout the proof.\\

\noindent
(1) Whenever $d\!\in\!\Z^+$ and $s,t\!\in\![d]$,
where $[d]\!=\!\{1,\ldots,d\}$ as before, let
$$\De_{st}=\big\{[\cC,y_1,y_2,\hat{y}_1,\ldots,\hat{y}_d]\!\in\!\ov\cM_{0,2|d}\!:~
\hat{y}_s\!=\!\hat{y}_t\big\}\in H^2\big(\ov\cM_{0,2|d}\big)$$
denote the class of the corresponding ``diagonal" and define
$$\De_s=\sum_{t=s+1}^d\!\De_{st}\in H^2(\ov\cM_{0,2|d}).$$
For each $a_k\!>\!0$, $s\!\in\![d]$, and $r\!\in\![a_k]$, there is a short
exact sequence
\begin{equation*}\begin{split}
0\lra R^0\pi_*\O\bigg((r\!-\!1)\hat\si_s+\sum_{t=s+1}^d\!\!a_k\hat\si_t-\si_1\bigg)
&\lra R^0\pi_*\O\bigg(r\hat\si_s+\sum_{t=s+1}^d\!\!a_k\hat\si_t-\si_1\bigg)\\
&\lra R^0\pi_*\O\bigg(\bigg(r\hat\si_s+\sum_{t=s+1}^d\!\!a_k\hat\si_t
-\si_1\bigg)\bigg|_{\hat\si_s}\bigg)
\lra 0.
\end{split}\end{equation*}
This gives
\BE{Vplus_e}\begin{split}
a_k\!>\!0\qquad\Lra\qquad \E(\dot\V_{a_k}^{(d)}(\al_i))
&=\prod_{s=1}^d\prod_{r=1}^{a_k}\big(a_k\al_i-r\hat\psi_s+a_k\De_s\big)\\
&=a_k^{a_kd}\al_i^{a_kd}
\prod_{s=1}^d\prod_{r=1}^{a_k}\big(1-\frac{r}{a_k}\al_i^{-1}\hat\psi_s+\al_i^{-1}\De_s\big).
\end{split}\EE
For each $a_k\!<\!0$, $s\!\in\![d]$, and $r\!=\!0,1,\ldots,-a_k\!-\!1$,
there is a short exact sequence
\begin{equation*}\begin{split}
0\lra
R^0\pi_*\O\bigg(\bigg(-r\hat\si_s+\sum_{t=s+1}^d\!\!a_k\hat\si_t-\si_1\bigg)\bigg|_{\hat\si_s}\bigg)
&\lra R^1\pi_*\O\bigg((-r\!-\!1)\hat\si_s+\sum_{t=s+1}^d\!\!a_k\hat\si_t-\si_1\bigg)\\
&\lra R^1\pi_*\O\bigg(-r\hat\si_s+\sum_{t=s+1}^d\!\!a_k\hat\si_t-\si_1\bigg)\lra0.
\end{split}\end{equation*}
This gives
\BE{Vminus_e}\begin{split}
a_k\!<\!0 \qquad\Lra\qquad \E(\dot\V_{a_k}^{(d)}(\al_i))
&=\prod_{s=1}^d\prod_{r=0}^{-a_k-1}\!\!\big(a_k\al_i+r\hat\psi_s+a_k\De_s\big)\\
&=a_k^{-a_kd}\al_i^{-a_kd}
\prod_{s=1}^d\prod_{r=0}^{a_k-1}\big(1+\frac{r}{a_k}\al_i^{-1}\hat\psi_s+\al_i^{-1}\De_s\big).
\end{split}\EE
Similarly to \e_ref{Vplus_e},
\BE{Vdenom_e}
\E(\dot\V_1^{(d)}(\al_i\!-\!\al_k))
=(\al_i\!-\!\al_k)^d\prod_{s=1}^d\big(1-(\al_i\!-\!\al_k)^{-1}\hat\psi_s
+(\al_i\!-\!\al_k)^{-1}\De_s\big).\EE
(2) For $d\!\in\!\Z^{\ge0}$, let
$$C_i(\al)=\frac{\prod\limits_{a_k>0}\big(a_k^{a_k}\al_i^{a_k}\big)
\prod\limits_{a_k<0}\big(a_k^{-a_k}\al_i^{-a_k}\big)}
{\prod\limits_{k\neq i}(\al_i\!-\!\al_k)}\,.$$
We denote by $\fs_1,\fs_2,\ldots$ the elementary symmetric polynomials in
$$\{\be_k\}=\big\{(\al_i\!-\!\al_k)^{-1}\!:~k\!\neq\!i\big\}$$
for any given number of formal variables~$\be_k$.
Note that
\BE{cHnums_e}\begin{split}
 \frac{(-1)^b}{d!} \int_{\ov\cM_{0,2|d}}
\frac{\prod\limits_{a_k>0}\prod\limits_{s=1}^d\prod\limits_{r=1}^{a_k}
\!\!(1-\frac{r}{a_k}y\hat\psi_s+y\De_s)~
\prod\limits_{a_k<0}\prod\limits_{s=1}^d\prod\limits_{r=0}^{a_k-1}
\!\!(1+\frac{r}{a_k}y\hat\psi_s+y\De_s)\,\,\psi_1^r\psi_2^b}
{\prod\limits_{k=1}^{n-1}\prod\limits_{s=1}^d(1\!-\!\be_k\hat\psi_s+\!\be_k\De_s)}&\\
=\H_{\a;d}^{r,b}(y,\fs_1,\ldots,\fs_{d-1})
\in \Q[y,\be_1,\ldots,\be_{n-1}]&
\end{split}\EE
for some $\H_{\a;d}^{r,b}\!\in\!\Q[y,\fs_1,\ldots,\fs_{d-1}]$,
independent of~$n$.
Such $\H_{\a;d}^{r,b}$ exists because the integrand on the left-hand side of~\e_ref{cHnums_e}
is symmetric in $\{\be_k\}$ and whatever $\H_{\a;d}^{r,b}$ works for $n\!\ge\!d\!-\!r\!-\!b$
works for all~$n$ (this can be seen by setting the extra $\be_k$'s to~0).
By \e_ref{Vplus_e}-\e_ref{cHnums_e},
\BE{RHSexp_e1}
\frac{(-1)^b}{d!}\int_{\ov\cM_{0,2|d}}\frac{\E(\dot\V_{\a}^{(d)}(\al_i))\psi_1^r\psi_2^b}
{\prod\limits_{k\neq i}\E(\dot\V_1^{(d)}(\al_i\!-\!\al_k))}
=C_i(\al)^d\, \H_{\a;d}^{r,b}(\al_i^{-1},\fs_1,\ldots,\fs_{d-1})
\quad\forall\,d\!\in\!\Z^{\ge0}.\EE
Similarly, for any $d,d'\!\in\!\Z^{\ge0}$
there exists $\cY_{\a;d,d'}\!\in\!\Q[y,\fs_1,\ldots,\fs_{d'}]$, independent of~$n$,
such~that
\BE{cYnums_e}
\left\llbracket\frac{\prod\limits_{a_k>0}\prod\limits_{r=1}^{a_kd}\!\!(1\!+\!\frac{r}{a_k}y\hb)
\prod\limits_{a_k<0}\!\!\prod\limits_{r=0}^{-a_kd-1}\!\!(1\!-\!\frac{r}{a_k}ry\hb)}
{d!\prod\limits_{r=1}^d\prod\limits_{k=1}^{n-1}(1\!+\!r\be_k\hb)}
\right\rrbracket_{\hb;d'}
=\cY_{\a;d,d'}(y,\fs_1,\ldots,\fs_{d'})\,.\EE
By \e_ref{cYdfn_e} and~\e_ref{cYnums_e},
\BE{RHSexp_e2}
\left\llbracket \hb^d \left\llbracket
\cY_{n;\a}(\al_i,\hb,q)\right\rrbracket_{q;d}\right\rrbracket_{\hb;d'}
=C_i(\al)^d\, \cY_{\a;d,d'}(\al_i^{-1},\fs_1,\ldots,\fs_{d'})
\quad\forall\,d,d'\!\in\!\Z^{\ge0}.\EE
(3) By \e_ref{RHSexp_e1} and \e_ref{RHSexp_e2},
\e_ref{cYcZrel_e} is equivalent~to
\BE{cHcY_e}
\cY_{\a;d,d-1-r}(y,\fs_1,\fs_2,\ldots)
=\!\sum_{\begin{subarray}{c}d_1+d_2=d\\ d_1\ge1\end{subarray}}
\!\!\!\sum_{b=0}^{d_1-1-r}\!\!\!
\H_{\a;d_1}^{r,b}(y,\fs_1,\fs_2,\ldots)
\cY_{\a;d_2,d_2+b}(y,\fs_1,\fs_2,\ldots)~~~\forall\,d\!\in\!\Z^+\,.\EE
This equivalence is obtained by taking the coefficients of $q^d$ of the
two sides of~\e_ref{cYcZrel_e}, factoring out $C_i(\al)^d$,
replacing $\al_i^{-1}$ by $y$ and $\{(\al_i\!-\!\al_k)^{-1}\!:k\!\neq\!i\}$ by
$\{\be_1,\ldots,\be_{n-1}\}$.
By Lemma~\ref{cYcZcomp_lmm} and Corollary~\ref{MirSym_crl},
\e_ref{cHcY_e} holds whenever $|\a|\!\le\!n\!-\!2$.
Since~\e_ref{cHcY_e} does not involve~$n$, it holds for all~$\a$.
Thus, \e_ref{cYcZrel_e} holds for all pairs~$(n,\a)$.
\end{proof}

\begin{proof}[Proof of Theorem~\ref{equiv0_thm}]
For each $d\!\in\!\Z^+$, denote by $D_{1\hat{1};2}\subset\ov\cM_{0,2|d}$
the divisor whose general element is a two-component rational curve,
with one of the components carrying the marked point~1 and the fleck~$\hat{1}$
and the other component carrying the marked point~2.
The second component must then carry at least one of the remaining flecks.
The irreducible components~$D_{1\hat{1};2I}$ of~$D_{1\hat{1};2}$
thus correspond to the nonempty subsets~$I$ of $\{2,\ldots,d\}$
indexing the flecks on the second component.
There is a natural isomorphism
\BE{D112split_e}D_{1\hat{1};2I}\approx
\ov\cM_{0,2|(d-|I|)}\times \ov\cM_{0,2||I|}.\EE
If $\pi_1,\pi_2$ are the two component projection maps,
\BE{D112rest_e}\begin{split}
\psi_i\big|_{D_{1\hat{1};2I}}&=\pi_i^*\psi_i \qquad i=1,2,\\
\E\big(\dot\V_{\a}^{(d)}(\be)\big)\big|_{D_{1\hat{1};2I}}
&=
\pi_1^*\E\big(\dot\V_{\a}^{(d-|I|)}(\be)\big)\cdot
\pi_2^*\E\big(\dot\V_{\a}^{(|I|)}(\be)\big).
\end{split}\EE
On the other hand, by the first identity in \e_ref{psipullback_e} and induction on $d$,
\BE{TRR_e} \psi_2=D_{1\hat{1};2}\in H^2(\ov\cM_{0,2|d}).\EE
By \e_ref{D112split_e}-\e_ref{TRR_e},
\BE{Intspli_e}\begin{split}
&\int_{\ov\cM_{0,2|d}}\frac{\E(\dot\V_{\a}^{(d)}(\al_i))\psi_1^{b_1}\psi_2^{b_2}}
{\prod\limits_{k\neq i}\!\E(\dot\V_1^{(d)}(\al_i\!-\!\al_k))}
=\int_{D_{1\hat{1};2}}\frac{\E(\dot\V_{\a}^{(d)}(\al_i))\psi_1^{b_1}\psi_2^{b_2-1}}
{\prod\limits_{k\neq i}\!\E(\dot\V_1^{(d)}(\al_i\!-\!\al_k))} \\
&\hspace{.3in}
=\sum_{\begin{subarray}{c}d_1,d_2\ge1\\ d_1+d_2=d\end{subarray}}
\!\!\binom{d\!-\!1}{d_1\!-\!1}
\Bigg(\int_{\ov\cM_{0,2|d_1}}\frac{\E(\dot\V_{\a}^{(d_1)}(\al_i))\psi_1^{b_1}}
{\prod\limits_{k\neq i}\!\E(\dot\V_1^{(d_1)}(\al_i\!-\!\al_k))}\Bigg)
\Bigg(\int_{\ov\cM_{0,2|d_2}}\frac{\E(\dot\V_{\a}^{(d_2)}(\al_i))\psi_2^{b_2-1}}
{\prod\limits_{k\neq i}\!\E(\dot\V_1^{(d_2)}(\al_i\!-\!\al_k))}\Bigg)
\end{split}\EE
whenever $b_2\!\in\!\Z^+$.\\

\noindent
For any $b_1,b_2\in\Z^{\ge0}$, let
$$\F_{n;\a}^{(b_1,b_2)}(\al_i,q)=
\sum_{d=1}^{\i}\frac{q^d}{d!}
\int_{\ov\cM_{0,2|d}}\frac{\E(\dot\V_{\a}^{(d)}(\al_i))\psi_1^{b_1}\psi_2^{b_2}}
{\prod\limits_{k\neq i}\E(\dot\V_1^{(d)}(\al_i\!-\!\al_k))}
\in q\Q_{\al}[[q]].$$
By \e_ref{Intspli_e},
\BE{Intspli_e2}
D\F_{n;\a}^{(b_1,b_2)}(\al_i,q)= D\F_{n;\a}^{(b_1,0)}(\al_i,q)\cdot \F_{n;\a}^{(0,b_2-1)}(\al_i,q)
 \qquad\forall\,b_2\!\in\!\Z^+\,,\EE
where $D\F\equiv q\frac{\nd}{\nd q}\F$.
By induction on $b_2$, this gives
$$\F_{n;\a}^{(0,b_2)}(\al_i,q)= \frac{1}{(b_2\!+\!1)!}\F_{n;\a}^{(0,0)}(\al_i,q)^{b_2+1}\,.$$
Combining this with \e_ref{Intspli_e2} and using symmetry, we obtain
\begin{alignat}{1}
D\F_{n;\a}^{(b_1,b_2)}(\al_i,q)&=
\frac{1}{b_1!}\F_{n;\a}^{(0,0)}(\al_i,q)^{b_1}D\F_{n;\a}^{(0,0)}(\al_i,q)\cdot
\frac{1}{b_2!}\F_{n;\a}^{(0,0)}(\al_i,q)^{b_2} \qquad\Lra\notag\\
\label{Fb1bd2_e}
\F_{n;\a}^{(b_1,b_2)}(\al_i,q)&=\frac{1}{(b_1\!+\!b_2\!+\!1)!}\binom{b_1\!+\!b_2}{b_1}
\F_{n;\a}^{(0,0)}(\al_i,q)^{b_1+b_2+1}.
\end{alignat}
Thus, the $r\!=\!0$ case of \e_ref{cYcZrel_e} is equivalent to
\BE{YRs_e}\Rs{\hb=0}\bigg\{\ne^{-\frac{\F_{n;\a}^{(0,0)}(\al_i,q)}{\hb}}
\cY_{n;\a}(\al_i,\hb,q)\bigg\}=0.\EE
By \cite[Section~2.1]{bcov1}, this relation determines
$\F_{n;\a}^{(0,0)}(\al_i,q)\!\in\!q\Q_{\al}[[q]]$ uniquely.
Thus, by \cite[Remark~4.5]{g0ci},
$\F_{n;\a}^{(0,0)}(\al_i,q)\!=\!\xi_{n;\a}(\al_i;q)$.\footnote{Only the case $\ell^-(\a)\!=\!0$ is considered in \cite{g0ci}, but the same reasoning applies in all cases.}
It follows that~\e_ref{Fb1bd2_e} is equivalent to the identity
in Theorem~\ref{equiv0_thm}.
\end{proof}

\begin{rmk} By~\e_ref{Fb1bd2_e},
for any $r^*\!\in\!\Z^{\ge0}$ the set of equations~\e_ref{cYcZrel_e}
with $r\!=\!0,1,\ldots,r^*$ is an invertible linear combination of
the set of relations
$$\Rs{\hb=0}\bigg\{\hb^r\ne^{-\frac{\F_{n;\a}^{(0,0)}(\al_i,q)}{\hb}}
\cY_{n;\a}(\al_i,\hb,q)\bigg\}=0, \qquad  r\!=\!0,1,\ldots,r^*.$$
Thus, by~\e_ref{Fb1bd2_e}, the statement of Proposition~\ref{cYcZrel_prp}
is equivalent to the condition that
the coefficients of the power series
$$\ne^{-\F_{n;\a}^{(0,0)}(\al_i,q)/\hb}\cY_{n;\a}(\al_i,\hb,q)\in\Q_{\al}(\hb)[[q]]$$
are regular at $\hb\!=\!0$.
This is indeed the case for $\F_{n;\a}^{(0,0)}(\al_i,q)\!=\!\xi_{n;\a}(\al_i;q)$
by \cite[Remark~4.5]{g0ci}.
\end{rmk}

\begin{rmk}
The above approach can be used to eliminate $\psi$-classes from twisted integrals
over $\ov\cM_{0,m|d}$ with $m\!\ge\!3$.
For example, let
$$\F_{n;\a}^{(b_1,b_2,b_3)}(\al_i,q)=
\sum_{d=0}^{\i}\frac{q^d}{d!}
\int_{\ov\cM_{0,3|d}}\frac{\E(\dot\V_{\a}^{(d)}(\al_i))\psi_1^{b_1}\psi_2^{b_2}\psi_3^{b_3}}
{\prod\limits_{k\neq i}\E(\dot\V_1^{(d)}(\al_i\!-\!\al_k))}.$$
Using $\psi_3\!=\!D_{12;3}$ on $\ov\cM_{0,3|d}$, we find that
\begin{gather*}
\F_{n;\a}^{(b_1,b_2,b_3)}(\al_i,q)=\F_{n;\a}^{(b_1,b_2,0)}(\al_i,q)
\cdot\F_{n;\a}^{(0,b_3-1)}(\al_i,q) \qquad\forall\,b_3\!\in\!\Z^+\\
\Lra\qquad
\F_{n;\a}^{(b_1,b_2,b_3)}(\al_i,q)=\frac{\xi_{n;\a}(\al_i,q)^{b_1+b_2+b_3}}{b_1!b_2!b_3!}
\F_{n;\a}^{(0,0,0)}(\al_i,q).
\end{gather*}
Multiplying the last equation by $\hb_1^{-b_1-1}\hb_2^{-b_2-1}\hb_3^{-b_3-1}$
and summing over $b_1,b_2,b_3\!\ge\!0$, we obtain
\begin{equation*}\begin{split}
&\sum_{d=0}^{\i}\frac{q^d}{d!}
\int_{\ov\cM_{0,3|d}}\frac{\E(\dot\V_{\a}^{(d)}(\al_i))}
{\prod\limits_{k\neq i}\E(\dot\V_1^{(d)}(\al_i\!-\!\al_k))\,
(\hb_1\!-\!\psi_1)(\hb_2\!-\!\psi_2)(\hb_3\!-\!\psi_3)}\\
&\hspace{1in}
=\frac{1}{\hb_1\hb_2\hb_3}\ne^{\frac{\xi_{n;\a}(\al_i,q)}{\hb_1}+\frac{\xi_{n;\a}(\al_i,q)}{\hb_2}
+\frac{\xi_{n;\a}(\al_i,q)}{\hb_3}}\F_{n;\a}^{(0,0,0)}(\al_i,q)
\in \Q_{\al}\big[\big[\hb_1^{-1},\hb_2^{-1},\hb_3^{-1},q\big]\big]\,.
\end{split}\end{equation*}
The power series $\F_{n;\a}^{(0,0,0)}$ is described in \cite[Section~3]{CZ2}.
\end{rmk}

\noindent
In the case of  products of projective spaces and concavex sheaves~\e_ref{gensheaf_e},
$\al_i$ and $q$ in~\e_ref{Yrec_e} and~\e_ref{cYcZrel_e}
are replaced by $(\al_{i_1},\ldots,\al_{i_p})$ with $i_s\!\in\![n_s]$
and $(q_1,\ldots,q_p)$ with the right-hand sides modified as in~\e_ref{cZrelext_e}.
In the proof of Proposition~\ref{cYcZrel_prp}, we then obtain relations between
elementary symmetric polynomials~in
$$\{\al_{1;1},\ldots,\al_{1;n_1}\},\quad\ldots\quad,
\{\al_{p;1},\ldots,\al_{p;n_p}\}$$
that depend on $\a$, but not on $n_1,\ldots,n_p$.
They again hold if $|a_{1;s}|\!+\!\ldots\!+\!|a_{l;s}|\!\le\!n_s\!-\!2$ for all $s\!\in\![p]$
and thus in all cases.\\

\vspace{5mm}

\noindent
{\it Department of Mathematics, Princeton University, Princeton, NJ 08540\\
yaim@math.princeton.edu}\\

\noindent
{\it Department of Mathematics, SUNY Stony Brook, NY 11794-3651\\
azinger@math.sunysb.edu}


\begin{thebibliography}{99}

\bibitem{ABo} M.\,Atiyah and R.\,Bott,
{\it The moment map and equivariant cohomology}, Topology 23 (1984), 1--28

\bibitem{CdGP} P.~Candelas, X.~de la Ossa, P.~Green, and L.~Parkes,
{\it A pair of Calabi-Yau manifolds as an exactly soluble superconformal theory},
Nuclear Phys.~B359 (1991), 21--74

\bibitem{CK} I.~Ciocan-Fontanine and B.~Kim,
{\it Moduli stacks of stable toric quasimaps},
Adv.~Math.~225 (2010), no.~6, 3022–-3051

\bibitem{CKM} I.~Ciocan-Fontanine, B.~Kim, and D.~Maulik,
{\it Stable quasimaps to GIT quotients}, arXiv:1106.3724

\bibitem{Co11} Y.~Cooper, {\it The geometry of stable quotients in genus one},
arXiv:1109.0331, to appear in Math.~Ann.

\bibitem{CZ3} Y.~Cooper and A.~Zinger,
{\it Mirror symmetry for stable quotients invariants in genus 1},
in preparation.

\bibitem{Elezi} A.~Elezi,
{\it Mirror symmetry for concavex vector bundles on projective spaces},
Int.~J.~Math.~Math. Sci.~2003, no.~3, 159–-197

\bibitem{Gi} A.~Givental,
{\it The mirror formula for quintic threefolds},
Amer.\,Math.\,Soc.\,Transl.\,Ser.~2, 196 (1999), 49--62

\bibitem{Gi2} A.~Givental,
{\it Equivariant Gromov-Witten invariants}, IMRN no.~13 (1996), 613--663

\bibitem{GP} T.~Graber and R.~Pandharipande,
{\it Localization of virtual classes}, Invent.~Math.~135 (1999), no.~2, 487–-518

\bibitem{Hassett} B.~Hassett,
{\it Moduli spaces of weighted pointed stable curves},
Adv.~Math.~173 (2003), 316–-352

\bibitem{MirSym} K.~Hori, S.~Katz, A.~Klemm, R.~Pandharipande,
R.~Thomas, C.~Vafa, R.~Vakil, and E.~Zaslow, {\it Mirror Symmetry},
Clay Math.\ Inst., AMS, 2003

\bibitem{KR} C.~Krattenthaler and T.~Rivoal,
{\it On the integrality of the Taylor coefficients of mirror maps},
Duke Math.~J.~151 (2010), 175--218

\bibitem{LLY} B.~Lian, K.~Liu, and S.T.~Yau,
{\it Mirror Principle I}, Asian J.~of Math., 1 (1997), no.~4 , 729--763

\bibitem{LLY3} B.~Lian, K.~Liu, and S.T.~Yau,
{\it Mirror Principle III}, Asian J.~of Math., 3 (1999), no.~4, 771–-800

\bibitem{LY2} B.~Lian and S.-T.~Yau, {\it Integrality of certain exponential series},
Algebra and geometry (Taipei, 1995), 215–-227, Lect.~Algebra Geom.~2, Int.~Press, 1998

\bibitem{MO} A.~Marian and D.~Oprea,
{\it Virtual intersections on the Quot scheme and Vafa-Intriligator formulas},
Duke Math.~J.~136 (2007), no.~1, 81–-113

\bibitem{MOP09} A.~Marian, D.~Oprea, and R.~Pandharipande,
{\it The moduli space of stable quotients},
Geom.~Top.~15 (2011), no.~3, 1651--1706


\bibitem{MS} D.~McDuff and D.~Salamon,
{\it $J$-holomorphic Curves and Symplectic Topology}, AMS~2004

\bibitem{MP} D.~Morrison and R.~Plesser,
{\it Summing the instantons: quantum cohomology and mirror symmetry in toric varieties},
Nuclear Phys.~B 440 (1995), no.~1-2, 279–-354

\bibitem{Po} A.~Popa,
{\it Two-point Gromov-Witten formulas for symplectic toric manifolds},
arXiv:1206.2703

\bibitem{RT}  Y.~Ruan and G.~Tian, {\it A mathematical theory of quantum cohomology},
JDG 42 (1995),  no.~2, 259--367

\bibitem{Sernesi} E.~Sernesi,
{\it Deformations of Algebraic Schemes}, Springer, 2006

\bibitem{Witten}  E.~Witten, {\it Phases of $N=2$ theories in two dimensions},
Nuclear Phys.~B 403 (1993), no.~1--2, 159–-222


\bibitem{bcov1} A.~Zinger,
{\it The reduced genus 1 Gromov-Witten invariants of Calabi-Yau hypersurfaces},
J.~Amer.~Math.~Soc.~22 (2009), no.~3, 691–-737

\bibitem{bcov0} A.~Zinger,
{\it Genus zero two-point hyperplane integrals in Gromov-Witten theory},
Comm.~Analysis Geom.~17 (2010), no.~5, 1–-45

\bibitem{g0ci} A.~Zinger,
{\it The genus 0 Gromov-Witten invariants of projective complete intersections},
to appear in Geom.~Top., arXiv:1106.1633

\bibitem{CZ2} A.~Zinger,
{\it Double and triple Givental's $J$-function for stable quotients invariants},
math/1305.2142

\end{thebibliography}
\end{document}